\documentclass[a4paper,reqno]{amsart}
\usepackage{amsmath, amssymb, latexsym}
\usepackage{graphicx}
\usepackage{color}
\usepackage{array}
\usepackage{prettyref}
\usepackage{float}
\usepackage{units}
\usepackage{multirow}
\usepackage{amsthm}
\usepackage{mathtools}
\allowdisplaybreaks

\usepackage{natbib}
\bibpunct[, ]{(}{)}{,}{a}{}{,}%
\def\newblock{\ }%

\def\documenttitle{Exponential single server queues in an interactive random environment}

\usepackage[unicode=true,
bookmarks=true,bookmarksnumbered=false,bookmarksopen=false,
breaklinks=true,pdfborder={0 0 0},backref=false,colorlinks=true]
{hyperref}

\hypersetup{
	linkcolor=blue, citecolor=blue, urlcolor=blue, filecolor=blue,
	pdftitle={\documenttitle},
	pdfauthor={Sonja Otten, Hans Daduna, Ruslan Krenzler, Karsten Kruse},
	pdfsubject={queueing systems}
}

\usepackage{orcidlink}

\global\long\def\phantomeq{\mathrel{\phantom{=}}}

\newtheorem{theorem}{Theorem}
\newtheorem{corollary}[theorem]{Corollary}
\newtheorem{definition}[theorem]{Definition}
\newtheorem{lemma}[theorem]{Lemma}
\newtheorem{proposition}[theorem]{Proposition}
\newtheorem{example}[theorem]{Example}
\newtheorem{remark}[theorem]{Remark}
\newtheorem{assumption}[theorem]{Assumption}
\newtheorem{conjecture}[theorem]{Conjecture}

\numberwithin{theorem}{section}
\numberwithin{equation}{section}

\newcommand{\Rentry}[3]{r_{#3}(#1,#2)}
\newcommand{\myv}{v}

\newcommand{\cal}[1]{\mathcal{#1}}

\makeatletter
\newcommand{\fakephantomsection}{%
	\Hy@GlobalStepCount\Hy@linkcounter%
	\Hy@MakeCurrentHref{\@currenvir.\the\Hy@linkcounter}
	\Hy@raisedlink{\hyper@anchorstart{\@currentHref}\hyper@anchorend}%
}
\makeatother

\begin{document}
\title{\documenttitle}
\author{Sonja Otten\,\orcidlink{0000-0002-3124-832X}, 
	Ruslan Krenzler\,\orcidlink{0000-0002-6637-1168}, Hans Daduna\,\orcidlink{0000-0001-6570-3012}, Karsten Kruse\,\orcidlink{0000-0003-1864-4915}}




\begin{abstract}
	We consider exponential single server queues with state-dependent arrival and service rates which evolve under influences of external environments.
	The transitions of  the queues are influenced by the environment's state and the movements of the environment depend on the status of the queues (bi-directional interaction).
	The environment is constructed in a way to encompass various models from the recent Operations Research literature, where a queue is coupled with an inventory or with reliability issues.
	With a Markovian joint queueing-environment process we prove separability for a large class of such interactive systems, i.e.~the steady state distribution is of product form and explicitly given. The queue and the environment processes decouple asymptotically and in steady state.
	
	For non-separable systems we develop ergodicity and exponential ergodicity criteria via Lyapunov functions. By examples we explain principles for bounding departure rates of served customers (throughputs) of non-separable systems by throughputs of related separable systems as upper and lower bound.
\end{abstract}

\maketitle

\section{Introduction and literature review}\label{sect:intro}
Queueing systems which evolve under influences from external sources have found interest for long times. Today's emerging complex technological and logistic systems revived interest in such models. Features of interest in these systems are, e.g.~
services of different quality that are provided to individual customers, or to general differentiated demand, or even to data sets (messages in sensor networks), etc., under side-constraints of limited and shared resources and under restrictions which are external from the point of view of the service system.
The need for understanding the behaviour of these systems revived  investigations of queueing systems in a dedicated environment. Because in real life situations service systems are subject to randomness of interarrival times and service time requests, and environments are usually non-deterministic, the area of queues in random environments is today a field of active research in applied probability.

\subsection{Problem setting}\label{sect:ProblemSetting}
The model of a queue in a random environment studied in this note summarizes and unifies various models 
which emerged over the last two decades in
different research areas, especially in Operations Research.
The following standard example of a production-inventory system will serve as an introductory example and will be modified in  course of the article.

\begin{example}[Production-inventory system, see Figure~\ref{fig:BS-PER-figure-base-stock-model}]\label{ex:QueueInventory1}
	The production system  with an attached inventory considered here fits into the class of queueing systems in a random environment: The production system ($=$ exponential server) interacts with an  inventory and an associated replenishment system (supplier) (environment $=$ inventory-replenishment subsystem).
	
	\begin{figure}[h]
		\centering{}\includegraphics[width=0.95\textwidth]{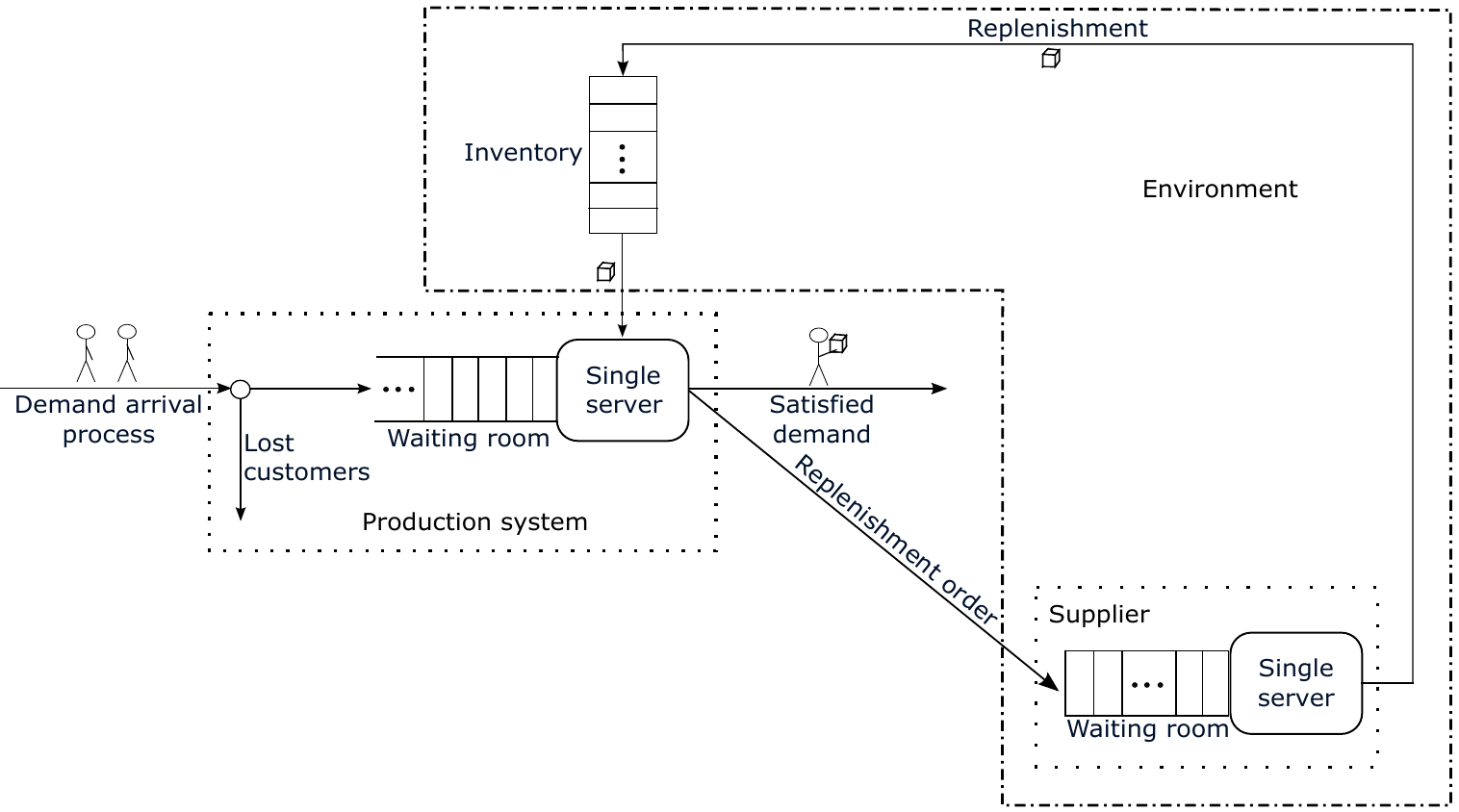}\caption{\label{fig:BS-PER-figure-base-stock-model}Production-inventory system}
	\end{figure}
	%
	The production system consists
	of a single server (machine) with infinite waiting room that serves
	demand of customers  on a make-to-order basis under first-come-first-served regime (FCFS).
	To satisfy a customer's demand the production system needs exactly one item of raw material from the associated inventory.
	
	Arriving customers join the queue unless the inventory is depleted ($=$ ``lost sales'' principle from inventory theory).	If customers are present and the inventory is not depleted, the customer at the head of the line is served and new arrivals are admitted. A customer departs from the system immediately after  service and the associated consumed raw material is formally removed from the inventory at this instant of time.	
	If the server is ready to serve a customer and the inventory is not depleted, service immediately starts. Otherwise, customers in the waiting line stay on and  service starts again  when the next replenishment arrives.

	The main characteristic of the production-inventory system  is in our setting the following: If the inventory is depleted, no service is possible and new arrivals are lost. A sketchy formal description of the queueing-environment interaction in this example is as follows. (More information will be given in Example \ref{ex:QueueInventory11}.)
	
	The production system is modelled as a standard exponential queueing system with state space $\mathbb{N}_0$ (queue lengths) and the inventory with state space $K:= \{0,1,\dots,b\}$ (inventory sizes) is its environment which influences the queue's development. The decisive properties for our class of models are: 
	(i) Whenever the state of the inventory is in the subset $K_{W}:= \left\{ 1,\ldots,b\right\} \subseteq K$, the queue is functioning properly  ({\textbf W}orks), service and arrival processes are ongoing. (ii) Whenever the state of the inventory is in the subset
	$K_{B}:= \left\{ 0\right\} \subseteq K$, the queueing system is completely stalled ({\textbf B}locked) (due to stock-out no production can be performed and because of the 
	lost sales regime no new arrivals occur).
	
	An important property of the system's dynamics is that neither the queue nor the inventory evolves autonomously.
	The production system can only serve if the inventory is not depleted (stock size $>0$), and the inventory can only decrease whenever  production is possible (queue length $>0$). We characterize this as a ``bi-directional interaction''. Note that the replenishment system is part of the environment, although it is only implicitly represented in $K$.
\end{example}

\subsection{Literature review}\label{sect:Literature}
\subsubsection{Queues in a random environment.} \label{sect:QueuesinRE}
A recent review of queueing-inventory systems (as in Example \ref{ex:QueueInventory1}) is the article of
\cite{krishnamoorthy;shajin;viswanath:19}, including 10 pages of references. The majority of the articles mentioned in this review is on non-separable systems, but separable queueing-inventory systems in the spirit of our investigations are compiled and discussed as well.

We observed that the dichotomy $K_B$ versus $K_W$, i.e.~a partition of $K$ occurs with many very different queueing models.		
Representative examples are e.g.:
\begin{enumerate}
	\item
	Supply chains of production facilities (modelled as  a queue) with an attached inventory as described in Example \ref{ex:QueueInventory1}.
	This model was investigated first in the articles of \cite{sigman;simchi-levi:92} and \cite{melikov;molchanov:92}, where demand in case of stock-out at the inventory is backordered. Intensive
	research on this model started with a series of articles by Berman and his coauthors, e.g.~\cite{berman;kim:99,berman;sapna:00,berman;sapna:02}. The first explicit result on stationary behaviour for such integrated production-inventory models is developed by \cite{schwarz;sauer;daduna;kulik;szekli:06}, where demand in case of stock-out at the inventory is lost. It was shown that the associated queueing-inventory process is separable. A review of separable queueing-inventory systems is the article of \cite{krishnamoorthy;lakshmy;manikandan:11}. More contributions which focus on product form steady states are the articles of \cite{saffari;haji;hassanzadeh:11,saffari;asmussen;haji:13}
	and the theses of \cite{vineetha:08} and \cite{otten:18}.
	\item
	Sensor networks, where a dedicated node  (test node, referenced node) featuring an internal message queue which
	interacts with a complex environment.  
	The environment  incorporates location and status of
	neighboured sensor nodes and  geographical conditions, as well as internal status information of the referenced node, e.g.~activity level or sleep mode. Here $K_B$ consists of those environmental states which indicate (among other properties of the system) that the referenced node is sleeping and can neither receive, process or forward messages. $K_W$ encompasses all other environment states
	and if the environment is in such state, the dedicated node's message queue is functioning properly.  A detailed study is given by \cite{krenzler;daduna:14},  where a compilation of related literature can be found in Section 1.
	\item
	Queues where the availability of service capacity  depends on external
	conditions and/or control decisions. These external conditions are collected in the environment set $K$ and the subset $K_B$ consists of those 
	states where the server is stalled,  e.g.~for preventive maintenance.
	This has been researched for decades, see a review in \cite{krishnamoorthy;pramod;chakravarthy:14}.
	Explicit formulas for the stationary distribution of such systems were derived by \cite{sauer;daduna:03}.
	A recent study of performability for a randomly degrading queue 
	with maintenance options in the spirit of our present research is the article of \cite{krenzler;daduna:15a}.
	\item
	Queueing networks where a node of special interest is embedded in the environment set $K$ constituted by the set of the other nodes, 
	are investigated e.g.~by \cite[Section 4.5.1]{dijk:93}  and \cite[Section 4.2.2]{krenzler;daduna:15}. 
	In a network with finite buffers a  typical example of a state in $K_B$  is defined by those states where the other nodes have full buffers and the node of special interest is therefore stalled.
	\item
	Queues in a random environment as an example for application of matrix-analytical methods in the framework of quasi-birth-death processes (QBDs). Typical  examples are Markov-arrival processes (MAPs) with queue-length-dependent service mechanisms and related structures, see e.g.\ \cite[Sections 3 and 6]{neuts:81}  for general principles.
	A control problem   for  queues in a random environment is investigated by \cite{helm;waldmann:84}. 
	The common feature of all these 
	quasi-birth-death process models is: The queue length process is the level process while the environment process is the phase process. It will become obvious that our systems can be formulated as QBDs, but most of the other mentioned QBDs do not obey the dichotomy $K_W$ versus $K_B$ for the states in the environment space $K$. 
	A note in the spirit of the present article is by \cite{economou:03a} where criteria for the existence of a  stationary distribution of product form are provided.
	\item  Closely related to our present research on queues  subject to external side-constraints  with accessible stationary distribution are the articles of  \cite{falin:96} and  \cite{krenzler;daduna:15}.
\end{enumerate}

\subsubsection{Related models.}\label{sect:RelatedModels}
A Markovian exponential queueing-environment process can be considered as  birth-death process in a random environment.
There exists a bulk of literature on that subject.
Classical birth-death processes (with discrete or continuous time) in a random environment found in the literature are not separable. 
We discuss some examples shortly.
Discrete time Markov chains in a Markovian random environment are investigated by \cite{cogburn:80, cogburn:84}. In the first paper classification of states are provided and in the second conditions for the existence of stationary distributions and for ergodicity are presented. \cite{cornez:87} investigated a discrete time birth-death process with absorbing state $0$ in a general environment with feedback (bi-directional interaction),
\cite{cogburn;torrez:81} investigated continuous time  birth-death processes in a Markovian random environment and provide criteria for recurrence and transience. Applications to queues in a Markovian environment are sketched.
\cite{yechiali:73} considered  continuous time birth-death processes in a finite ergodic Markovian environment. The birth and death rates depend on the environment state and the population size.
The focus is on stationary regime and it is shown that ``in general, closed-form results for the limiting probabilities are difficult to obtain''\cite[p.~604]{yechiali:73}. For population size independent birth-death rates a special condition is found that allows to obtain geometrical steady state distribution. 
\cite{prabhu;zhu:89, prabhu;zhu:95} investigated single server systems under Markov-modulation (which is similar to a Markovian environment). In the first paper the modulated Poissonian arrival stream generates single customer arrivals, while in the second paper group arrivals are considered. The focus is on stationary behaviour.

Typical problems with computing stationary distributions and performance metrics which occur even with finite (non-separable) birth-death processes in a random environment are described in \cite{gaver;jacobs;latouche:84}  and investigated using matrix-analytical techniques.

In \cite{gannon;pechersky;suhov;yambartsev:16} a reversible Jacksonian queueing network is the environment for a random walker (a distinguished customer) on the network lattice. This work was extended in \cite{daduna:16}, where the random walker was substituted by a travelling server (``moving queue'') on the network. This was motivated by modeling a referenced mobile sensor node with an internal message queue (test node) in a network of mobile sensor nodes. Although the stationary distributions occurring in \cite{daduna:16} are separable, the transition mechanism of the system does not fit in the class of models considered here in Section \ref{sect:Separable}.

A recent study of an $M/M/1$-queue  in  abstract ``interactive'' environments is the article of \cite{pang;sarantsev;belopolskaya;suhov:20}, where the environment is either diffusive (continuous environment space) or a jump environment (discrete environment space).
We consider  only discrete environment spaces.
Our systems and processes differ from those in \cite[Section 2]{pang;sarantsev;belopolskaya;suhov:20}
because we allow (i) that the arrival and service rates are queue-length-dependent, and (ii) that the queueing system and the environment may jump concurrently.
Recall that in Example \ref{ex:QueueInventory1}  with  departures  of served customers the environment ($=$ inventory size) decreases at the same moment by one. Such concurrent jumps are not allowed by \cite{pang;sarantsev;belopolskaya;suhov:20}. On the other side, the dependence on the environment's state for the  arrival and service rates  are more specific in our setting.

A related class of models are random walks on $\mathbb{Z}$ in a random medium, see  Part I of \cite{sznitman:02} for an introduction. Our research in this article is on different problems than those described there and our methods are different from those used generally in that field.

\subsection{Research plan}\label{sect:Plan}
We are interested in a unified Markovian description for a class of general queueing-environment systems. The above examples reveal general principles: {\bf(i)} The environment state space $K$ is divided into disjoint subsets: $K_W$ and $K_B$. {\bf(ii)}
The evolution of the system follows general rules but in any case some special environment conditions, modelled as states in $K_B$ interrupt the dynamics of the queueing system, and the server is stalled as long as these conditions hold on.
{\bf(iii)} The environment is non-autonomous with respect to the queueing system, i.e.~its dynamics depend on the status of the associated queue.  
This bi-directional interaction is different from most work in the literature but reflects many real situations, as shown in examples above.
{\bf(iv)} Concurrent jumps of the queue and the environment occur with positive probability.\\
Special emphasis will be put on
\begin{itemize}
	\item separable models with  product form steady state:
	Asymptotically and in equilibrium  the queue and the environment as components (in space) of the 1-dimensional marginals in time decouple.
	(In steady state the queue and the environment seem to behave independently at fixed times.) This usually allows to determine ergodicity conditions directly, and
	\item  non-separable models where ergodicity and exponential ergodicity conditions via construction of Lyapunov functions will be  proven. In this case the infinite system of global balance equations of the joint queueing-environment process are usually not explicitly solvable.
\end{itemize}

The environment is always modelled by a discrete state space. In Sections \ref{sect:model}, \ref{sect:Separable}, \ref{sect:FiNModel} and \ref{sect:ergodicity-necessary} the environment is allowed to be infinite, from Section \ref{sect:ergodicity-Lyapunov} on we assume the environment to be finite.

\subsection{Main results and techniques}\label{sect:MainResults}
\begin{enumerate}
	\item[\textbf{(i)}] We identify a large class of separable queues
	with finite or infinite waiting room in a non-autonomous random environment. We compute explicitly the stationary distribution which opens the path for performance evaluation of these systems. The main part of the proof is to substitute the two-dimensional set of global balance equations by a set of independent one-dimensional equations with the same solution.
	This is combined with the stationary distribution of the isolated queue which is well known.
	\item[\textbf{(ii)}] For non-separable systems we provide necessary conditions for ergodicity which reveal a hidden geometrical structure of the (unknown) stationary distribution. Sufficient criteria for ergodicity and exponential ergodicity are proved using a Lyapunov function approach.
	The main technique is to start with an ergodic version of the queue in isolation (which usually can be easily characterized)  and an associated Lyapunov function for the queue only. Taking this function as partial function of the two-dimensional target function, 
	a second partial (environment) function is attached to obtain the final function.
	\item[\textbf{(iii)}] We take the explicit results of \textbf{(i)} to approximate performance measures of the ergodic (proved via \textbf{(ii)}) non-separable system (no stationary distribution at hand) using lower and upper bounds
	of modified versions of the target system which are separable.
	The main technique is to construct suitable reward processes which generate the respective performance measures.
\end{enumerate}

\subsubsection{Structure of the article.}
In Section \ref{sect:model} we describe the general model of a queueing system in a non-autonomous random environment.
In Section \ref{sect:Separable} we characterize separability of the queueing-environment process and derive ergodicity conditions and the stationary distribution. The case of a finite waiting room is investigated in Section \ref{sect:FiNModel}.
In Section \ref{sect:Non-separable} we investigate non-separable queueing-environment systems.
In Section \ref{sect:ergodicity-necessary} we prove a necessary condition for ergodicity of non-separable queueing systems in a random environment, which is trivially valid in the separable case.
In Section \ref{sect:ergodicity-Lyapunov} we provide sufficient conditions for ergodicity by constructing a Lyapunov function which indicates negative drift of the queueing-environment process. The section ends with a non-separable modification of the introductory Example \ref{ex:QueueInventory1}. In Section \ref{sect:expo-ergodicity-Lyapunov} we prove conditions for exponential ergodicity.
In Section \ref{sect:BoundingNPF} we combine our findings on separable and non-separable systems by showing that in many cases it is possible to find for a non-separable system related separable partner systems such that performance indices of the former (which are not explicitly computable) can be bounded by the respective indices of the separable partner.


\subsubsection{Notations, definitions and conventions.}
\begin{itemize}
	\item $\mathbb{R}_{0}^{+}:= [0,\infty)$, $\mathbb{R}^{+}:= (0,\infty)$, $\mathbb{N}$:= \{1,2,3,\dots\}, 
	$\mathbb{N}_{0}:= \{0\}\cup\mathbb{N}$
	\item Empty sums are 0, and empty products are 1.
	\item $1_{\left\{ expression\right\} }$ is the indicator function which
	is $1$ if $expression$ is true and $0$ otherwise.
	\item We write $C:=A + B$ to
	emphasize that $C$ is the union of disjoint sets $A$ and $B$.
	\item For $a\in \mathbb{R}$ we set  $a_{+}:=\max\{0,a\}$.
	\item All random variables and processes occurring henceforth are defined
	on a common underlying probability space $(\Omega,{\cal F},P)$.
\end{itemize}

Queueing systems in a random environment 
are described in this article by homogeneous Markov processes with countable (discrete) state space. All processes occurring henceforth have the properties summarized in the following definition.

\begin{definition}\label{def:MarkovProperties}
	A Markov process $X$ with state space $E$ and transition rate matrix $\mathbf{Q}$ is regular if
	(i)	all states are stable, i.e.~all diagonal elements of $\mathbf{Q}$ are finite, and (ii) $\mathbf{Q}$ is conservative, i.e.~its row sums are zero, and (iii) the process is non-explosive, i.e.~the sequence of jump times of the process diverges almost surely.  $X$ has cadlag paths, i.e.~each path of  $X$ is right-continuous and has left limits everywhere.
	
	The Markov process $X$ is ergodic if there exists a probability measure $\pi:=(\pi(z):z\in E)$ such that
	$P(X(t)=z|X(0)=\hat z)\longrightarrow \pi(z)$ for $t\to\infty$ holds for all $z\in E$ independent of the initial state $X(0)=\hat z$, $\hat z \in E$. $\pi$ is the asymptotic and stationary distribution of $X$.
	
	An ergodic Markov process $X$ with asymptotic distribution $\pi$
	is exponentially ergodic if there exists some $\alpha>0$
	and constants $C_{z,\hat{z}}>0, {z,\hat{z}}\in E,$
	such that $|P(X(t)=z|X(0)=\hat{z})-\pi(z)|\leq 
	C_{z,\hat{z}}\cdot e^{-\alpha\cdot t}$ for all ${z,\hat{z}}\in E$ 	holds.
\end{definition}   

We use the following Foster-Lyapunov criteria for ergodicity (see \cite[Proposition D.3]{kelly;yudovina:14}) and exponential ergodicity (see \cite[Theorem 6.5]{anderson:91}).
\begin{proposition}\label{prop:KellyYudovina14}
	Let $X:=(X(t):t\geq 0)$ be an irreducible  regular Markov process with countable state space $E$ and transition rate matrix $\mathbf{Q}:=(q(x;y):x,y\in E)$.
	Suppose that ${\cal L}:E\to [0,\infty)$ is a function such that for constants $\varepsilon>0$ and $b\in \mathbb{R}$, and some finite exception set $F\subset E$ and all $x\in E$ it holds
	\begin{equation}
	\sum_{y\in E\setminus\{x\}}q(x;y)\left[{\cal L}(y)-{\cal L}(x)\right]\leq
	\begin{cases}
	-\varepsilon &x\notin F,\\
	b-\varepsilon & x\in F.
	\end{cases}
	\end{equation}
	Then $X$ is ergodic.
\end{proposition}	

\begin{proposition}\label{prop:Anderson91}
	Let $X:=(X(t):t\geq 0)$ be an ergodic Markov process with countable state space $E$ and transition rate matrix $\mathbf{Q}:=(q(x;y):x,y\in E)$.
	$X$ is exponentially ergodic if and only if there exists a function
	${\cal M}:E\to [0,\infty)$, a finite exception set $F\subset E$, and some $\omega\in(0, \inf_{x\in E\setminus F}(-q(x;x)))$ such that the following holds:
	\begin{align}
	{\cal M}(x) &=0	,~ x\in F, \label{eq:ExpoErg2}\\
	\sum_{y\in E}q(x;y) {\cal M}(y) &<\infty, ~ x\in F,
	\label{eq:ExpoErg3}\\
	\sum_{y\in E\setminus\{x\}}q(x;y)\left[{\cal M}(y)-{\cal M}(x)\right]&\leq -\omega {\cal M}(x) - 1,~~~
	x\in E\setminus F. \label{eq:ExpoErg1}
	\end{align}
\end{proposition}	
The functions ${\cal L}$ and  ${\cal M}$ in the above propositions are called drift functions or \emph{Lyapunov functions}. In the literature,  Lyapunov functions are utilized as test functions to prove other properties of Markov processes (explosion, non-explosion, absorption) as well.
Usually the processes' behaviour, especially the drift, is characterized via transformation of the test functions by the infinitesimal generator. 


\section{The model: Queue in an interactive random environment}\label{sect:model}

Our starting point is the classical $M/M/1/\infty$-queue with queue-length-dependent rates under first-come-first-served regime (FCFS).
Customers are indistinguishable. If the queue length (i.e.~number of customers either waiting or in service) is $n\geq 0$, customers arrive at the system with rate $\lambda(n)>0,$  
and if $n\geq 1$, service is provided to the customer at the head of the line with rate $\mu(n)>0$. We set formally $\mu(0):= 0$.

Setting in force the usual (conditional) independence assumptions, the queue length process  $X:= (X(t):t\geq 0)$ with state space $\mathbb{N}_0$ is Markov. It is the simplest example of a birth-death process, which in case of ergodicity has stationary distribution
$\xi:= (\xi(n):n\in\mathbb{N}_0)$ with
\begin{equation}\label{eq:queue-xi}
\xi(n) :=C^{-1} \cdot \prod_{k=0}^{n-1}\frac{\lambda(k)}{\mu(k+1)},\quad n\in\mathbb{N}_0,
\end{equation}
where $C$ is the normalization constant.

We consider the situation where the service system  and the arrival stream are subject to external random influences which disturb  the queueing process, see Figure \ref{fig:queue-env}.
\begin{figure}[h]
	\centering{}\includegraphics[width=0.6\textwidth]{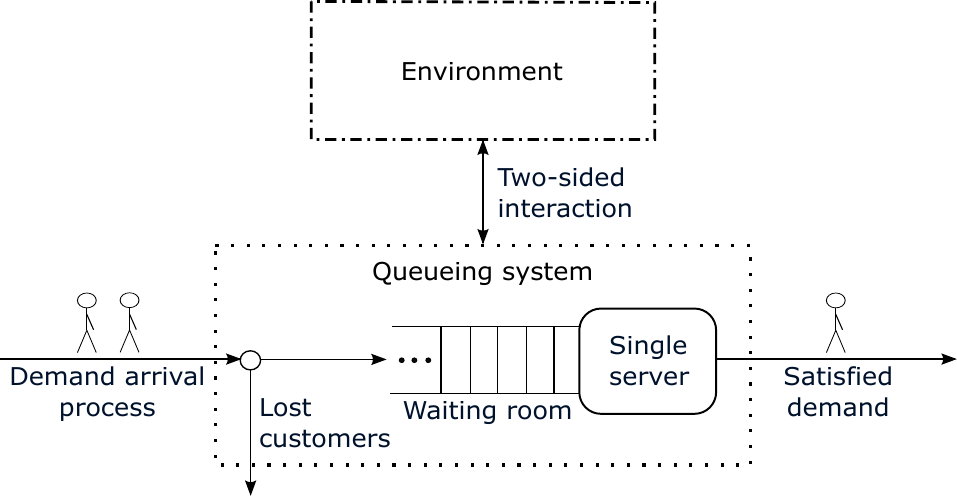}\caption{\label{fig:queue-env}Queue in a random environment}
\end{figure}
The states of the environment are summarized as
countable environment state space $K\neq \varnothing$ and the environment process is denoted by
$Y:= (Y(t):t\geq 0)$. The joint queueing-environment process is
$Z:=(X,Y)=((X(t),Y(t)):t\geq 0)$ on state space $E:= \mathbb{N}_0\times K$, i.e.~$Z(t):= (X(t),Y(t))=(n,k)$ indicates that at time $t$  the queue length is $n$ and the environment's state is $k$.

In most investigations found in the literature (e.g.~\cite{foss;shneer;tyurlikov:12}, \cite{zhu:94},  \cite{economou:05})  an \emph{autonomous environment} is considered:
This means that the environment process $Y=(Y(t):t\geq 0)$  on $K$ is Markov of its own, and the
state of the environment influences arrival and service rates of the queue. Consequently, in this
situation  there is only a "one-way interaction".

We consider in this note mainly the case of \emph{non-autonomous environments}:
Then the environment process $Y$  on $K$ is not Markov of its own.
The state of the environment and state of the queue influence transitions of the other component of the system
vice-versa, which results in a "two-way interaction".
In case of birth-death processes in a random environment this is often termed "feedback property of the environment",  see e.g.~\cite{cornez:87}.

In all applications mentioned in the introduction, for the  joint queueing-environment process $Z=(X,Y)$ on  $E=\mathbb{N}_0\times K$ we observe that
the environment space  $K$ is partitioned  as a disjoint union $K = K_W + K_B$ with the  following meaning and consequences:
\begin{itemize}
	\item$Y(t)\in K_{B}$ $ :\Longleftrightarrow$ at time  $t$ no service is provided and no arrivals occur,
	i.e.~the queue length process $X$ is frozen (= {server is {\bf B}locked}).
	\item$Y(t)\in K_{W}$ $:\Longleftrightarrow$  at time  $t$ the server is functioning, new arrivals
	are admitted (= {server  {\bf W}orks}).
\end{itemize}
The dynamics of the environment are as follows: If the queue length is $n\in \mathbb{N}_0$, then
\begin{itemize}
	\item a generator $V_n:= (v_n(k,m):k,m\in K)$ governs continuous changes 	of  the environment for $n\geq 0,$ and
	\item a stochastic matrix $R_n:= (r_n{(k,m)}:k,m\in K)$ governs instantaneous jumps of the environment 
	triggered by service  completions (downward jumps of the queue)
	for $n\geq1$.
\end{itemize}

We henceforth assume that
$Z=(X,Y)=((X(t),Y(t)):t\geq 0)$  is a Markov process on $E = \mathbb{N}_{0}\times K$.
The characterizing data for the system's development are  $\lambda(n)$, $\mu(n)$, the
countable  environment {$K=K_W + K_B$}, and the driving components for the environment $R_n$ and $V_n$.
The generator $\mathbf{Q}:= ({q}(z,\tilde z):z, \tilde z\in E)$ of 
$Z$ is with generic state $(n,k)\in E$:
\begin{align}
q\left((n,k);(n+1,k)\right) & := \lambda(n)\cdot{1_{\left\{ k\in K_{W}\right\} }}, \label{generator1}\\
q\left((n,k);(n-1,\ell)\right) & := \mu(n)\cdot {r_{n}(k,\ell)}\cdot1_{\left\{ n>0\right\} }\cdot
{1_{\left\{ k\in K_{W}\right\} }},& \ell\in K,\nonumber\\
q\left((n,k);(n,\ell)\right) & := {v_{n}(k,\ell)}, & \ell\in K,\ k\neq\ell,\nonumber
\end{align}
and $q(z;\tilde{z}):= 0$ for  other $z\neq\tilde{z}$,
and
$q\left(z;z\right):= -\sum_{\substack{\tilde{z}\in E,\\
		z\neq\tilde{z}} }q\left(z;\tilde{z}\right)$ for all $z\in E.$

We note that such processes can be considered as quasi-birth-death  processes with level-dependent phase-dynamics.

As pointed out in Section 1 the class of Markov processes $Z=(X,Y)$ on $E=\mathbb{N}_0\times K$ with generator given by \eqref{generator1}
encompasses a rich class of examples from different application areas. The typical examples lead to Markov processes $Z$ which are irreducible on $E$. On the other side, the general construction \eqref{generator1}
allows examples of reducible systems  which seem to be of no specific interest. We therefore set in force the following assumption.
\begin{assumption}\label{ass:irreducible}
	We assume throughout that $Z=(X,Y)$, resp. $\mathbf{Q}$ is irreducible on $E$.
\end{assumption}
%
\begin{remark}{} \label{rem:NotErgodic} Although $Z=(X,Y)$ is a Markov process, neither $X$ nor $Y$ is in general Markov. Moreover, even under Assumption \ref{ass:irreducible}
	neither the $R_n$ nor the $V_n$ need to be irreducible  nor determines an ergodic Markov chain, respectively an ergodic Markov process.
\end{remark}


\section{Separable queueing-environment systems}\label{sect:Separable}

Separability means roughly that the state vector of a multi-dimensional Markov process in equilibrium has for
a fixed moment independent coordinates. The classical examples are Jackson networks of queues
(\cite{jackson:57}) and their generalizations as BCMP networks
(\cite{baskett;chandy;muntz;palacios:75}), resp.~Kelly networks (\cite{kelly:76}).
We are interested in conditions which guarantee this asymptotic (resp. equilibrium) independence for the pair $Z=(X,Y)$.
A characterization theorem for the case when the  dynamics of the environment are independent
of queue lengths, i.e.~$V_n=V$, $R_n=R$ for all $n$, is proved in \cite[Theorem 2]{krenzler;daduna:12}.
We provide a characterization for the general case here.
%
%
\begin{theorem}\fakephantomsection
	\label{theorem:Separable1}
	\begin{enumerate}
		\item[{\bf (a)}] For $n\in \mathbb N_0$ define ``reduced generators''  ${\mathbf{Q}}_{red}^{(n)}:= ({q}_{red}^{(n)}(k,m):k,m\in K)$
		on the environment space $K$ via $\mathbf{Q}$ by
		\begin{align*}
			{q}_{red}^{(n)}(k,m) & :=  \lambda{(n)}\cdot r_{n+1}(k,m)\cdot 1_{\{k\in K_{W}\}}+v_n(k,m),\qquad k\neq m,\\
			{q}_{red}^{(n)}(k,k) & :=  -[1_{\{k\in K_{W}\}}\cdot \lambda{(n)}\cdot 
			(1-r_{n+1}(k,k))+\sum_{m\in K\backslash\{k\}}v_n(k,m)].
		\end{align*}
		The ``reduced generators'' ${\mathbf{Q}}_{red}^{(n)}$ are generators for  Markov processes on $K$.
		\item[{\bf (b)}] The following properties are equivalent:
		\begin{enumerate}
			\item[{\bf (i)}] $Z=(X,Y)$ is ergodic with product form steady state $\pi = \xi \cdot \theta$, with $\xi:= (\xi(n):n\in\mathbb{N}_0)$ from \eqref{eq:queue-xi}, i.e.
			\begin{equation}\label{eq:LS-product-form}
			\pi(n,k)= C^{-1}\cdot \left(\prod_{i=0}^{n-1}\frac{\lambda(i)}{\mu{(i+1)}}\right) \cdot \theta(k),\quad
			(n,k)\in E,
			\end{equation}
			where $\theta:= (\theta(k):k\in K)$ is a probability distribution on $K$.
			\item[{\bf (ii)}] The summability condition
			$ C  := \sum_{n=0}^{\infty}\prod_{i=0}^{n-1}\frac{\lambda{(i)}}{\mu{(i+1)}} <\infty$
			holds, and  the equation
			\begin{equation}\label{eq:LS-q-tilde-0-conditon}
			\theta\cdot{\mathbf{Q}}_{red}^{(0)}=0
			\end{equation}
			admits a strictly positive stochastic solution $\theta := (\theta(k):k\in K)$ which solves also
			\begin{equation}\label{eq:LS-q-tilde-n-conditon}
			\forall n\in \mathbb N:\theta\cdot{\mathbf{Q}}_{red}^{(n)}=0 .
			\end{equation}
		\end{enumerate}
	\end{enumerate}
\end{theorem}
%
%
\proof
\textbf{(a)} Let $k\in K$. By definition we have   ${q}_{red}^{(n)}(k,m)\geq 0$ for all $m\in K\setminus{k}$ and ${q}_{red}^{(n)}(k,k)\leq 0.$ It holds
\begin{align*}
	&\phantomeq\sum_{m\in K\setminus \{k\}} {q}_{red}^{(n)}(k,m)
	=\sum_{m\in K\setminus \{k\}} \left( \lambda{(n)}\cdot r_{n+1}(k,m)\cdot 1_{\{k\in K_{W}\}}+v_n(k,m)\right)\\
	&= \lambda{(n)}\cdot 1_{\{k\in K_{W}\}}
	\underbrace{\sum_{m\in K\setminus \{k\}}r_{n+1}(k,m)}_{= 1- r_{n+1}(k,k)} +
	\sum_{m\in K\setminus \{k\}}v_n(k,m) = 
	- {q}_{red}^{(n)}(k,k).
\end{align*}	
So  the row sums of all
${\mathbf{Q}}_{red}^{(n)}, n\in\mathbb{N}_0,$ are zero.

\textbf{(b)} \textbf{(ii)} $\mathbf{\Rightarrow}$ \textbf{(i):}
By assumption (\ref{eq:LS-q-tilde-0-conditon}) there exists a stochastic
solution to $\theta\cdot {\mathbf{Q}}_{red}^{(0)}=0$, which according to requirement
(\ref{eq:LS-q-tilde-n-conditon}) is a solution of $\theta\cdot {\mathbf{Q}}_{red}^{(n)}=0$ too.
Due to the summability condition $C<\infty$ we can use our $\theta$ and $C$ to define $\pi(n,k)$ by the right-hand side of (\ref{eq:LS-product-form}). Next, we show that $\pi$ fulfils the global balance equation $\pi\cdot \mathbf{Q}=\mathbf{0}$ of the Markov process $(X,Y)$, which are for
$(n,k)\in E$ 
\begin{align}
&  \phantomeq \pi(n,k)\cdot \Big(1_{\{k\in K_{W}\}}\cdot \lambda(n)+\sum_{m\in K\backslash\{k\}}\myv_n(k,m)+1_{\{k\in K_{W}\}}\cdot 1_{\{n>0\}}\cdot \mu(n)\Big)\nonumber \\
& =  \pi(n-1,k)\cdot 1_{\{k\in K_{W}\}}\cdot 1_{\{n>0\}}\cdot \lambda(n-1)+\sum_{m\in K_{W}}\pi(n+1,m)\cdot \Rentry mk{n+1}\cdot \mu(n+1)\nonumber \\
&  \phantomeq +\sum_{m\in K\backslash\{k\}}\pi(n,m)\cdot  \myv_n(m,k).\label{eq:LS-lba-general-2-2}
\end{align}

Inserting the proposed product form solution (\ref{eq:LS-product-form})
for the stationary distribution into the global balance
equations (\ref{eq:LS-lba-general-2-2}), cancelling $C^{-1}$
and multiplying with $\prod_{i=0}^{n-1}\left(\frac{\lambda(i)}{\mu(i+1)}\right)^{-1}$
yields
\begin{align*}
	&  \phantomeq \theta(k)\cdot \Big(1_{\{k\in K_{W}\}}\cdot \lambda(n)+\sum_{m\in K\backslash\{k\}}\myv_n(k,m)+
	\underbrace{1_{\{k\in K_{W}\}}\cdot 1_{\{n>0\}}\cdot \mu(n)}_{(*)}\Big)\\
	& =  \underbrace{\theta(k)\cdot \frac{\mu(n)}{\lambda(n-1)}\cdot 1_{\{k\in K_{W}\}}\cdot 1_{\{n>0\}}\cdot \lambda(n-1)}_{(**)}
	+\sum_{m\in K_{W}}\theta(m)\cdot \frac{\lambda(n)}{\mu(n+1)}\cdot \Rentry mk{n+1}\cdot \mu(n+1)\\
	&  \phantomeq +\sum_{m\in K\backslash\{k\}}\theta(m)\cdot \myv_n(m,k).
\end{align*}
Cancelling for $n\geq 1$ the expressions $(\theta(k)\cdot (*))$ and $(**)$ yields for all $n\geq0$
\begin{align*}
	&  \phantomeq \theta(k)\cdot \Big(1_{\{k\in K_{W}\}}\cdot \lambda(n)+\sum_{m\in K\backslash\{k\}}\myv_n(k,m)\Big)\\
	& =  \sum_{m\in K_{W}}\theta(m)\cdot {\lambda(n)}\cdot \Rentry mk{n+1}
	+\sum_{m\in K\backslash\{k\}}\theta(m)\cdot \myv_n(m,k).
\end{align*}
This implies
\begin{align}
&  \phantomeq \theta(k)\cdot \underbrace{\Big(1_{\{k\in K_{W}\}}\cdot \lambda(n)\cdot (1- \Rentry kk{n+1})
	+\sum_{m\in K\backslash\{k\}}\myv_n(k,m)\Big)}_{= - q_{red}^{(n)}(k,k)}\nonumber\\
& =  \sum_{m\in K\setminus\{k\}}\theta(m)\cdot \underbrace{\left(
	1_{\{m\in K_{W}\}}\cdot {\lambda(n)}\cdot \Rentry mk{n+1} +\myv_n(m,k)\right)}_{= q_{red}^{(n)}(m,k)},
\label{eq:LS-equal-solution}
\end{align}
which is for all $n\geq 0$ the condition \eqref{eq:LS-q-tilde-0-conditon},
respectively \eqref{eq:LS-q-tilde-n-conditon}, i.e.~$\theta\cdot {\mathbf{Q}}_{red}^{(n)}=0$.
Hence, we conclude  that \eqref{eq:LS-q-tilde-0-conditon} and \eqref{eq:LS-q-tilde-n-conditon} guarantee that (\ref{eq:LS-product-form}) solves (\ref{eq:LS-lba-general-2-2}).
Therefore, the global balance equations of $Z$ admit the  strictly positive stochastic solution
$\pi= \xi\cdot \theta$, which is unique by Assumption \ref{ass:irreducible},
and so $Z$ is ergodic.

\textbf{(b)} \textbf{(i)} $\mathbf{\Rightarrow}$ \textbf{(ii):}
Because $\pi$ is stochastic and of product form, summability 
holds. Insert the stochastic vector of product form \eqref{eq:LS-product-form}
into \eqref{eq:LS-lba-general-2-2}. As shown in the part \textbf{(ii)}
$\mathbf{\Rightarrow}$ \textbf{(i)} of the proof, this leads to \eqref{eq:LS-equal-solution}
and we have found a strictly positive stochastic solution $\theta$
which solves \eqref{eq:LS-q-tilde-0-conditon} and
\eqref{eq:LS-q-tilde-n-conditon} for all $n\in\mathbb{N}$.	\endproof

%
\begin{remark}\label{rem:Tn}
	The reduced generators ${\mathbf{Q}}_{red}^{(n)}=({q}_{red}^{(n)}(k,m):k,m\in K)$ can be considered as generalizations of the generators
	$\mathbf{T}_n, n\in \mathbb{N}_0,$ in the construction of the Markovian jump generators in \cite[Section 2]{pang;sarantsev;belopolskaya;suhov:20}.
	The reduced generators ${\mathbf{Q}}_{red}^{(n)}$
	enable concurrent jumps in two dimensions for  the original generator $\mathbf{Q}$.
	The property that $\theta$ is the common solution of the equations 
	\eqref{eq:LS-q-tilde-0-conditon} and
	\eqref{eq:LS-q-tilde-n-conditon}   
	is  parallel to Assumption 2.1 in \cite{pang;sarantsev;belopolskaya;suhov:20} required for the $\mathbf{T}_n, n\in \mathbb{N}_0,$ there.  
\end{remark}

%
\begin{example}[Production-inventory system, see 
	Example \ref{ex:QueueInventory1} with Figure~\ref{fig:BS-PER-figure-base-stock-model}]\label{ex:QueueInventory11}
	As described in the introduction this production-inventory system  fits into the definition of the queueing system in a random environment described by a Markov process $Z=(X,Y)=\text{(queue-length,\ inventory size)}$ on state space $E=\mathbb{N}_0\times K$ with $K:= \{0,1,\dots,b\}$ and
	\[
	K_{B}:= \left\{0\right\} ,\qquad K_{W}:= \left\{ 1,\ldots,b\right\}.
	\]
	$Y(t)\in K_{W}$ indicates for the inventory that there is 
	stock on hand for production, and $Y(t)\in K_{B}$ indicates stock-out. 
	
	The inventory is controlled according to the base stock policy, i.e.~each item taken from the inventory triggers  an immediate order for one item of raw material at the supplier. The base stock level
	$b\geq1$ is the maximal size of the inventory. The supplier consists of a single server  with waiting room of size $b-1$
	under FCFS. Service times to produce one item of raw material at the supplier are exponentially distributed with parameter $\nu>0$.
	A finished item of raw material departs immediately from the supplier and is added to the inventory.
	
	Note that the physical environment of the production system includes the replenishment system. The status of the replenishment server at time $t$ is uniquely determined by the size of the inventory as $b-Y(t)$.
	
	The production system is a single server queue with state-dependent rates. If the queue length is $n\geq 0$, service is provided with rate $\mu(n)>0$, $n\geq 1$, to the customer at the head of the line (if any) and the arrival stream has rate $\lambda(n)>0$.
	The dynamics of $Z$ are determined by the infinitesimal generator $\mathbf{Q}=\left(q(z;\widetilde{z}):z,\widetilde{z}\in E\right)$ with the following transition rates for $(n,k)\in E =\mathbb{N}_{0}\times K$:
	\begin{align*}
	q\left((n,k);(n+1,k)\right) & := \lambda(n)\cdot{1_{\left\{k\in\left\{ 1,\ldots,b\right\}\right\}}},\\
	q\left((n,k);(n-1,k-1)\right) & := \mu(n) \cdot 1_{\left\{ n>0\right\} }\cdot
	{1_{\left\{ k\in\left\{ 1,\ldots,b\right\}\right\} }},\\
	q\left((n,k);(n,k+1)\right) & := \nu \cdot 1_{\left\{ k\in\left\{ 0,1,\ldots,b-1\right\}\right\}},
	\end{align*}
	and $q(z;\tilde{z}):= 0$ for  other $z\neq\tilde{z}$,
	and
	$q\left(z;z\right):= -\sum_{\substack{\tilde{z}\in E,\\
			z\neq\tilde{z}} }q\left(z;\tilde{z}\right)$ for all $z\in E.$
	
	The queue-length-dependent dynamics of the inventory process are determined by 
	\[
	r_{n}(0,0):= 1,\qquad r_{n}(k,k-1):= 1,\quad k\in\left\{ 1,\ldots,b\right\},\ n\in\mathbb{N},
	\]
	and $r_n(k,\ell):=0$ for other $k,\ell\in K,\ n\in\mathbb{N}$,
	\begin{align*}
	v_{n}(k,\ell) & := \begin{cases}
	\nu, & \text{if }k\in\left\{ 0,1,\ldots,b-1\right\},\ \ell=k+1,\\
	0, & \text{otherwise for }k\neq\ell,
	\end{cases}\quad n\geq0,
	\end{align*}
	and $v_n\left(k,k\right):= -\sum_{\substack{\ell\in K,\\
			k\neq\ell} }v_n\left(k,\ell\right)$ for all $k\in K$ and $n\geq 0$.
	
	If $\lambda(n)=\lambda$ for all $n\in \mathbb{N}$ and some $\lambda >0$ and the production-inventory process is ergodic, then the stationary distribution $\pi$ is given by
	\begin{equation}\label{eq:QueueingInventoryPF}
	\pi(n,k):=  \xi(n)\cdot \theta (k), \quad (n,k)\in E,
	\end{equation}
	\begin{align*}
	\text{with}\qquad \xi(n):= C^{-1} \cdot \prod_{i=0}^{n-1}\frac{\lambda}{\mu{(i+1)}},  \  n\in \mathbb{N}_{0},\qquad
	\theta(k):= C_{\theta}^{-1}\cdot \left(\frac{\nu}{\lambda}\right)^{k}, \  k\in K,
	\end{align*}
	and normalization constants $C$ and $C_{\theta}$. 
	Here $\theta$ is the steady state of the  classical inventory (without service system) with demand rate $\lambda$.
\end{example}
\medskip

In  Section \ref{sect:BoundingNPF} we present further examples of  separable queueing-environment systems. But as will be seen there, queue-length-dependent environment dynamics often lead to models which do not satisfy the conditions of Theorem \ref{theorem:Separable1}. We therefore discuss in the next example systems with simple queue-length-dependent
environment dynamics which  either \textbf{(a)} prevent common jumps of the queue and the environment or  \textbf{(b)} enforce concurrent jumps of the queue and the environment.
We complement these separable examples in \textbf{(c)} with a slight modification of  \textbf{(b)} which destroys separability.

%

\begin{example}[Queue-length dependent environment dynamics: Availability]
	\label{ex:QueueInventory1A}
	We consider an ergodic $M/M/1/\infty$-queue with arrival rate $\lambda(n)$ and service rate $\mu(n)$ where the server and the arrival stream are randomly interrupted by breakdown of the server. The server is under repair for a random time after each breakdown.
	Different interruption schemes will be investigated.
	
	In any case the availability of the server is subject to interruption and restart by an alternating exponential renewal process (on-off process) with queue-length-dependent rates.
	The state space of the environment is $K:= \{0,1\}$, where $0 =$ off and  $1 =$ on. In terms of our general model $K_B=\{0\}$ and $K_W=\{1\}$. For queue length $n$ the mean off-times are $\eta(n)^{-1}$ and the mean on-times are $\gamma(n)^{-1}$.
	We specify the dynamics of the environment in three examples by generator matrices $V_n$ and jump matrices $R_n$ and obtain separable and non separable systems.
	\begin{enumerate}
		\item[\textbf{(a)}] Service system with breakdown and repair: The queueing system and the environment have \emph{no simultaneous jumps} (which is the general situation of \cite{pang;sarantsev;belopolskaya;suhov:20}).
		The generators $V_n$ are given with $\eta(n):=\eta\cdot (n+1)$ and $\gamma(n):=\gamma\cdot (n+1)$ for some $\eta,\gamma>0$ by
		\begin{equation}\label{eq:Vn}
		v_n(0,0):=-\eta\cdot (n+1),\ \
		v_n(0,1):= \eta\cdot (n+1),\ \
		v_n(1,0):= \gamma\cdot (n+1), \ \
		v_n(1,1):= -\gamma\cdot (n+1),
		\ n\geq 0.
		\end{equation}
		To exclude jumps of the environment when a customer departs, the jump matrices are taken as identity 
		$R_n:=  [1_{\{i=j\}}:i,j\in\{0,1\}]$.
		It follows ${\mathbf Q}^{(n)}_{red} = V_n$
		and the common probability solution $\theta$ of 
		$
		\theta\cdot {\mathbf{Q}}^{(n)}_{red} =0,~ n\geq 0,
		$
		is independent of the arrival and service rates,
		\begin{equation}\label{eq:UnrelA}
		\theta =\left(\frac{\gamma}{\gamma +\eta},\frac{\eta}{\gamma +\eta} \right).
		\end{equation}
		Consequently, the system is \textbf{separable} by Theorem
		\ref{theorem:Separable1}.
		
		The next examples are modifications of \textbf{(a)} which \emph{allow the environment to jump concurrently with the queue}. Both systems occur as ``vacation queues'' in the literature.
		\item[\textbf{(b)}] In a service system with breakdown and repair the server takes a vacation whenever a service is completed. The generators $V_n$ from  \eqref{eq:Vn} control the continuous changes of the environment. 
		The jump matrices are given by
		$r_{n+1}(1,0):=1$, $r_{n+1}(0,0):=1$ for $n\geq 0$, and zero otherwise, which says that whenever a service is completed the server is not available for a random time (takes a vacation).
		With linear arrival rates 
		$\lambda(n):=\lambda\cdot(n+1)$ for $n\geq 0$ and some $\lambda>0$ and any service rate function $\mu(\cdot)$
		the common probability solution $\theta$ of 
		$
		\theta\cdot {\mathbf{Q}}^{(n)}_{red} =0, ~ n\geq 0,
		$
		is
		\begin{equation}\label{eq:UnrelB}
		\theta =\left(\frac{\lambda+\gamma}{\lambda+\gamma +\eta},\frac{\eta}{\lambda+\gamma +\eta} \right).
		\end{equation}
		Consequently, the system is \textbf{separable} by Theorem
		\ref{theorem:Separable1}.
		This specific vacation policy
		occurs in investigations of polling systems. If only one queue of a multi-queue polling system is investigated, the time when the server polls and serves the other queue is modelled as a vacation. If the queue of interest is controlled by the so-called 1-limited policy, then after each service the server takes a vacation, see  \cite{boon;boxma;winands:11} for a short introduction, for more details see \cite{takagi:90a}.
		\item[\textbf{(c)}] In a service system with breakdown and repair the server takes a vacation whenever the queue is empty after a service is completed. The generators $V_n$ from  \eqref{eq:Vn} control the continuous changes of the environment. 
		The jump matrices are 
		$r_{n+1}(1,1):=1,\ r_{n+1}(0,0):=1$ for $n\geq 1$, while $r_{1}(1,0):=1,\ r_{1}(0,0):=1$  and zero otherwise.
		So, whenever a departing customer leaves behind an empty queue, then for a random time the server takes a vacation because it serves somewhere else.
		Direct computation shows that 
		$
		\theta\cdot {\mathbf{Q}}^{(0)}_{red} =0$
		is solved by \eqref{eq:UnrelB} and 
		$
		\theta\cdot {\mathbf{Q}}^{(n)}_{red} =0, ~ n\geq 1,
		$
		is solved by \eqref{eq:UnrelA}. Consequently, the system is \textbf{not separable}. This control policy 
		is the standard vacation policy, which is applied to reduce idle times of servers (\cite{doshi:90}).
	\end{enumerate}
	
\end{example}

\newpage
\section{Separable queue with finite waiting room in a random environment}\label{sect:FiNModel}  

In this section we consider  the queueing-environment system described in Section  \ref{sect:model} with the restriction that the waiting room of the queue has finite capacity $N\geq 0$. So at most $N+1$ customers can reside in the system, either in service or waiting. Customers which arrive when the waiting room is full are lost for the system. We use the same notation as in Section \ref{sect:Separable} to make comparison easy.

The queue length process  $X:= (X(t):t\geq 0)$ (with states $\{0,1,\dots,N+1\}$) of the  $M/M/1/N$-queue with queue-length-dependent rates is an ergodic Markov process. Its stationary distribution is
$\xi:= (\xi(n):n\in\mathbb{N}_0)$ with
\begin{equation}\label{eq:FiNqueue-xi}
\xi(n) :=C^{-1} \cdot \prod_{k=0}^{n-1}\frac{\lambda(k)}{\mu(k+1)},\quad n\in\{0,1,\dots,N+1\},
\end{equation}
where $C$ is the normalization constant.
The stationary distribution 
\eqref{eq:FiNqueue-xi} can be obtained from the stationary distribution \eqref{eq:queue-xi} of a stationary $M/M/1/\infty$-queue by conditioning on the event ``queue length $\leq N+1$''. Shortly:
Truncation of the waiting room yields a stationary distribution obtained by conditioning.
This fact results from reversibility of the queue length process of a stationary $M/M/1/\infty$-queue.
We will show in Remark \ref{rem:truncation} below that  due to  problems arising at the boundary $\{N+1\}\times K$ of the state space a similar truncation-conditioning principle does not apply in general for the case of ergodic queueing-environment processes.

With  environment space $K = K_W + K_B$ as in Section  \ref{sect:model} we consider the joint queueing-environment process
$Z=(X,Y)=((X(t),Y(t)):t\geq 0)$ on state space $E:= \{0,1,\dots,N+1\}\times K$.
$Z(t)= (X(t),Y(t))=(n,k)$ indicates that at time $t$  the queue length is $n$ and the environment state is $k$.
The dynamics of the environment are similar to those described in
Section  \ref{sect:model}. For queue length  $n\in\{0,1,\dots,N+1\}$ 
\begin{itemize}
	\item a generator matrix $V_n:= (v_n(k,m):k,m\in K)$ governs continuous changes
	of  the environment, and
	\item a stochastic matrix $R_n:= (r_n{(k,m)}:k,m\in K)$ governs instantaneous jumps of the environment 
	triggered by service  completions (downward jumps of the queue).
\end{itemize}

We assume that the queueing-environment process
$Z$  is an irreducible Markov process on the state space $E =\{0,1,\dots,N+1\}\times K$ and note that Remark \ref{rem:NotErgodic} applies here as well.
The generator $\mathbf{Q}:= ({q}(z,\tilde z):z, \tilde z\in E)$ of the Markov process
$Z$ is with generic state $(n,k)\in E$:
\begin{align}
q\left((n,k);(n+1,k)\right) & := \lambda(n)\cdot{1_{\left\{ k\in K_{W}\right\} }} \cdot{1_{\left\{n\leq N\right\} }}, \label{FiNgenerator1}\\
q\left((n,k);(n-1,\ell)\right) & := \mu(n)\cdot {r_{n}(k,\ell)}\cdot1_{\left\{ n>0\right\} }\cdot
{1_{\left\{ k\in K_{W}\right\} }},& \ell\in K,\nonumber\\
q\left((n,k);(n,\ell)\right) & := {v_{n}(k,\ell)}, & \ell\in K,\ k\neq\ell,\nonumber
\end{align}
and $q(z;\tilde{z}):= 0$ for  other $z\neq\tilde{z}$,
and
$q\left(z;z\right):= -\sum_{\substack{\tilde{z}\in E,\\
		z\neq\tilde{z}} }q\left(z;\tilde{z}\right)$ for all $z\in E.$

We are looking for conditions which guarantee separability, i.e.~asymptotic independence for the pair $Z(t)=(X(t),Y(t))$.
A characterization theorem for the case when the  dynamic of the environment is independent
of queue length, i.e.~$V_n=V$, $R_n=R$ for all $n$, is proved in 
\cite[Section 3]{krenzler;daduna:15}  and \cite[Section 2.1.2]{krenzler:16}.
Although the general case investigated here is similar to Theorem  \ref{theorem:Separable1} we  encounter additional problems.
%
%
\begin{theorem}\fakephantomsection
	\label{theorem:FiNSeparable1}
	\begin{enumerate}
		\item[{\bf (a)}] For $n\in\{0,1,\dots,N+1\}$ define ``reduced generators''  ${\mathbf{Q}}_{red}^{(n)}:= ({q}_{red}^{(n)}(k,m):k,m\in K)$
		on the environment space $K$ via $\mathbf{Q}$ by
		\begin{align*}
			{q}_{red}^{(n)}(k,m) & :=  \lambda{(n)}\cdot r_{n+1}(k,m)\cdot 1_{\{k\in K_{W}\}}
			\cdot 1_{\{n\leq N\}}+v_n(k,m),\qquad k\neq m,\\
			{q}_{red}^{(n)}(k,k) & :=  -[1_{\{k\in K_{W}\}}\cdot 1_{\{n\leq N\}}\cdot \lambda{(n)}\cdot 
			(1-r_{n+1}(k,k))+\sum_{m\in K\backslash\{k\}}v_n(k,m)].
		\end{align*}
		The ``reduced generators'' ${\mathbf{Q}}_{red}^{(n)}$ are generators for  Markov processes on $K$.
		\item[{\bf (b)}] The following properties are equivalent:
		\begin{enumerate}
			\item[{\bf (i)}] $Z=(X,Y)$ is ergodic with product form steady state $\pi = \xi \cdot \theta$, with $\xi:= (\xi(n):n\in\mathbb{N}_0)$ from \eqref{eq:FiNqueue-xi}, i.e.
			\begin{equation} \label{eq:FiNLS-product-form}
			\pi(n,k)= C^{-1}\cdot \left(\prod_{i=0}^{n-1}\frac{\lambda(i)}{\mu{(i+1)}}\right) \cdot \theta(k),\quad
			(n,k)\in E,
			\end{equation}
			where $\theta:= (\theta(k):k\in K)$ is a probability distribution on $K$.
			\item[{\bf (ii)}]The equation
			\begin{equation}\label{eq:FiNLS-q-tilde-0-conditon}
			\theta\cdot{\mathbf{Q}}_{red}^{(0)}=0
			\end{equation}
			admits a strictly positive stochastic solution $\theta:= (\theta(k):k\in K)$ which solves also
			\begin{equation}\label{eq:FiNLS-q-tilde-n-conditon}
			\forall n\in \{1,\dots,N+1\}:\theta\cdot{\mathbf{Q}}_{red}^{(n)}=0 .
			\end{equation}
		\end{enumerate}
	\end{enumerate}
	
\end{theorem}
%
%
\proof
\textbf{(a)} By definition  ${q}_{red}^{(n)}(k,m)\geq 0$,
and ${q}_{red}^{(n)}(k,k)\leq 0$ for $k,m\in K$, $k\neq m$. Direct summation shows that the row sums of 
${\mathbf{Q}}_{red}^{(n)}$ are zero for all $n\in\{0,1,\dots,N+1\}$.

\textbf{(b)} \textbf{(ii)} $\mathbf{\Rightarrow}$ \textbf{(i):}
By assumption (\ref{eq:FiNLS-q-tilde-0-conditon}) there exists a strictly positive stochastic
solution to $\theta\cdot {\mathbf{Q}}_{red}^{(0)}=0$, which according to requirement
(\ref{eq:FiNLS-q-tilde-n-conditon}) is a solution of $\theta\cdot {\mathbf{Q}}_{red}^{(n)}=0$ for all $n\in\{1,\dots,N+1\}$ as well.
We define $\pi(n,k)$ by the right-hand side of
\eqref{eq:FiNLS-product-form}
and show that this $\pi$ fulfils the global balance  equations $\pi\cdot \mathbf{Q}=\mathbf{0}$ of the Markov process $(X,Y)$ which are for
$(n,k)\in E$
\begin{align}
&  \phantomeq \pi(n,k)\cdot \Big(1_{\{k\in K_{W}\}}\cdot \lambda(n) \cdot 1_{\{n\leq N\}}+\sum_{m\in K\backslash\{k\}}\myv_n(k,m)+1_{\{k\in K_{W}\}}\cdot 1_{\{n>0\}}\cdot \mu(n)\Big)\nonumber \\
& =  \pi(n-1,k)\cdot 1_{\{k\in K_{W}\}}\cdot 1_{\{n>0\}}\cdot \lambda(n-1)+\sum_{m\in K_{W}}\pi(n+1,m)\cdot \Rentry mk{n+1}\cdot \mu(n+1)\cdot 1_{\{n\leq N\}}\nonumber \\
&  \phantomeq +\sum_{m\in K\backslash\{k\}}\pi(n,m)\cdot \myv_n(m,k).\label{eq:FiNLS-lba-general-2-2}
\end{align}

Inserting the proposed product form solution
(\ref{eq:FiNLS-product-form}) 
for the stationary distribution into the global balance
equations (\ref{eq:FiNLS-lba-general-2-2}), cancelling $C^{-1}$
and multiplying with $\prod_{i=0}^{n-1}\left(\frac{\lambda(i)}{\mu(i+1)}\right)^{-1}$
yields
\begin{align*}
	&  \phantomeq \theta(k)\cdot \Big(1_{\{k\in K_{W}\}}\cdot \lambda(n)\cdot 1_{\{n\leq N\}}+\sum_{m\in K\backslash\{k\}}\myv_n(k,m)+
	\underbrace{1_{\{k\in K_{W}\}}\cdot 1_{\{n>0\}}\cdot \mu(n)}_{(*)}\Big)\\
	& =  \underbrace{\theta(k)\cdot \frac{\mu(n)}{\lambda(n-1)}\cdot 1_{\{k\in K_{W}\}}\cdot 1_{\{n>0\}}\cdot \lambda(n-1)}_{(**)}
	+\sum_{m\in K_{W}}\theta(m)\cdot \frac{\lambda(n)}{\mu(n+1)}\cdot \Rentry mk{n+1}\cdot \mu(n+1)\cdot 1_{\{n\leq N\}}\\
	&  \phantomeq +\sum_{m\in K\backslash\{k\}}\theta(m)\cdot \myv_n(m,k).
\end{align*}
Cancelling 	the expressions $(\theta(k)\cdot (*))$ and $(**)$ yields for all $n$ with $N+1\geq n\geq0$
\begin{align*}
	&  \phantomeq \theta(k)\cdot \Big(1_{\{k\in K_{W}\}}\cdot \lambda(n)\cdot 1_{\{n\leq N\}}+\sum_{m\in K\backslash\{k\}}\myv_n(k,m)\Big)\\
	& =  \sum_{m\in K_{W}}\theta(m)\cdot {\lambda(n)}\cdot \Rentry mk{n+1}\cdot 1_{\{n\leq N\}}
	+\sum_{m\in K\backslash\{k\}}\theta(m)\cdot \myv_n(m,k).
\end{align*}
This implies
\begin{align}
&  \phantomeq \theta(k)\cdot \underbrace{\Big(1_{\{k\in K_{W}\}}\cdot \lambda(n)\cdot 1_{\{n\leq N\}}\cdot (1- \Rentry kk{n+1})
	+\sum_{m\in K\backslash\{k\}}\myv_n(k,m)\Big)}_{= -  q_{red}^{(n)}(k,k)}\nonumber\\
& =  \sum_{m\in K\setminus\{k\}}\theta(m)\cdot \underbrace{\left(
	1_{\{m\in K_{W}\}}\cdot {\lambda(n)}\cdot \Rentry mk{n+1}\cdot 1_{\{n\leq N\}} +\myv_n(m,k)\right)}_{= q_{red}^{(n)}(m,k)},
\label{eq:FiNLS-equal-solution}
\end{align}
which is for all $n$ with  $N+1\geq n\geq 0$ the condition \eqref{eq:FiNLS-q-tilde-0-conditon},
respectively \eqref{eq:FiNLS-q-tilde-n-conditon}, i.e.~$\theta\cdot {\mathbf{Q}}_{red}^{(n)}=0$.
Hence, we conclude  that \eqref{eq:FiNLS-q-tilde-0-conditon} and \eqref{eq:FiNLS-q-tilde-n-conditon} guarantee that
(\ref{eq:FiNLS-product-form}) 	
solves (\ref{eq:FiNLS-lba-general-2-2}).
Therefore, the global balance equations of $Z$ admit the strictly positive stochastic solution
$\pi= \xi\cdot \theta$, which is unique by irreducibility
and so $Z$ is ergodic.

\textbf{(b)} \textbf{(i)} $\mathbf{\Rightarrow}$ \textbf{(ii):}
Take the stochastic vector $\pi$  of product form from
(\ref{eq:FiNLS-product-form}) and insert it
into \eqref{eq:FiNLS-lba-general-2-2}. As shown in the part \textbf{(ii)}
$\mathbf{\Rightarrow}$ \textbf{(i)} of the proof, this leads to \eqref{eq:FiNLS-equal-solution} and we have found a strictly positive stochastic solution $\theta$ which solves \eqref{eq:FiNLS-q-tilde-0-conditon} and
\eqref{eq:FiNLS-q-tilde-n-conditon} for all $n = 1, \dots, N+1$.	\endproof

\begin{remark}\label{rem:truncation}
	The statement of Theorem \ref{theorem:FiNSeparable1} is at a first glance
	(up to the size of the waiting room) almost identical to that of Theorem \ref{theorem:Separable1}. But there are subtleties which result from the boundary of the state space at finite height $N+1$.
	Consider for $n:= N+1$ condition \eqref{eq:FiNLS-q-tilde-n-conditon},
	which is in full detail \eqref{eq:FiNLS-equal-solution}. We conclude that  
	$\theta\cdot {\mathbf{Q}}_{red}^{(N+1)}=0$
	is just $\theta\cdot V_{N+1}=0$.
	Consequently, if we have a stationary separable queueing-environment system with infinite waiting room as discussed in Theorem
	\ref{theorem:Separable1}, a queueing-environment system with finite queue (say, of length $N$) derived by truncation of the waiting room, has a stationary distribution obtained by conditioning on the reduced state space if and only if the distribution $\theta$ on $K$ fulfils the conditions of Theorem \ref{theorem:Separable1} \emph{and} satisfies $\theta\cdot V_{N+1}=0$ or equivalently 
	\begin{equation}
	\theta(k)\cdot 1_{\{k\in K_{W}\}}\cdot  (1- \Rentry kk{N+2})
	= \sum_{m\in K\setminus\{k\}}\theta(m)\cdot 
	1_{\{m\in K_{W}\}}\cdot \Rentry mk{N+2}.
	\label{eq:FiNLS-equal-solutionConditioning1}
	\end{equation}
	Taking any $k\in K_B$ in \eqref{eq:FiNLS-equal-solutionConditioning1}, we see that for all $m\in K_W$ it must hold $\Rentry mk{N+2}=0$, i.e.
	the set $K_W$ is closed under $\Rentry \cdot\cdot{N+2}$. 
	This means that in the system with unbounded waiting room
	the subset $\{0,1,\dots,N+1\}\times K$ can be entered from above only by a customer's departure (which is only possible if the environment state is in $K_W$) without a jump out of $K_W$ into the set $K_B$.
\end{remark}

\begin{example}[Production-inventory system with finite capacity]  \label{ex:QueueInventory1Finite}
	Production-inventory systems with finite capacity of the waiting room have been considered in the literature, see e.g.~\cite{melikov;molchanov:92}, \cite{yadavalli;sivakumar;arivarignan:07,yadavalli;sivakumar;arivarignan;adetunji:12} and \cite[Section 6]{schwarz;sauer;daduna;kulik;szekli:06}.
	We consider a variant of the production-inventory system which is part of a ``transportation-storage system'' in \cite{melikov;molchanov:92}.
	
	In the model of Figure~\ref{fig:BS-PER-figure-base-stock-model} we assume that the single production server has a restricted waiting room of capacity $N\geq 0.$
	The maximal inventory size (for items of raw material needed for production)  is $Q>0.$ 
	An order for new raw material is placed immediately when the stock size drops down to $0$. 
	The time for delivering the order from the replenishment server is 
	exponentially distributed with parameter $\nu>0$, the interarrival time of customers is
	exponentially distributed with parameter $\lambda>0$, the service time is exponentially distributed with parameter $\mu>0$.
	The order size is random with
	distribution  $\kappa_n:=(\kappa_n(i):i=1,\dots,Q)$,
	when the queue length is $n$ at the moment of ordering.

	In \cite{melikov;molchanov:92}
	it is assumed that newly arriving requests are admitted to enter the system and are backordered as long as the waiting room is not full.
	So the arrival stream at the server is not interrupted when the stock reaches $0$. Therefore due to  backordering this policy does not fit into our scheme of environment behaviour.
	
	We therefore consider the companion lost sales inventory policy:
	Whenever the stock size drops down to $0$, newly arriving requests are rejected (similar to \cite{schwarz;sauer;daduna;kulik;szekli:06}). 
	The set $K_B:=\{0\}$ 
	is a ``blocking set'' in the sense of Section \ref{sect:model} and the environment space is with $K_W:= \{1,\dots,Q\}$  partitioned as $K=K_W+K_B$.
	
	The joint production-inventory process $Z=(X,Y)$ is (with the usual independence assumptions) Markov and we assume that the order size distributions guarantee that it is irreducible on $E:=
	\Big(\{0,1,\dots,N\}\times\{0,\dots,Q\}\Big)\cup\Big(\{N+1\}\times\{1,\dots,Q\}\Big).$ The ``state'' $(N+1,0)$ cannot be attained because stock size $0$ can only be entered when a customer departs concurrently.
	
	The dynamics of $Z$ are determined by the infinitesimal generator $\mathbf{Q}:=\left(q(z;\widetilde{z}):z,\widetilde{z}\in E\right)$ with the following transition rates for generic states $(n,k)\in E$:
	\begin{align*}
	q\left((n,k);(n+1,k)\right) & := \lambda\cdot{1_{\left\{k\in\left\{ 1,\ldots,Q\right\}\right\}}}\cdot{1_{\left\{n\in\left\{ 0,1,\ldots,N\right\}\right\}}},\\
	q\left((n,k);(n-1,k-1)\right) & := \mu \cdot 1_{\left\{ n>0\right\} }\cdot
	{1_{\left\{ k\in\left\{1,\ldots,Q\right\}\right\} }},\\
	q\left((n,0);(n,i)\right) & := \nu \cdot\kappa_n(i),
	\qquad i\in\{1,\dots,Q\},
	\end{align*}
	and $q(z;\tilde{z}):= 0$ for  other $z\neq\tilde{z}$,
	and
	$q\left(z;z\right):= -\sum_{\substack{\tilde{z}\in E,\\
			z\neq\tilde{z}} }q\left(z;\tilde{z}\right)$ for all $z\in E.$
	
	The queue-length-dependent dynamics of the inventory process are determined by 
	\begin{align*} 
	r_{n}(0,0):= 1,&\qquad r_{n}(k,k-1):= 1,\quad k\in\left\{ 1,\ldots,Q\right\},\ n\in \{1,\dots,N\},\\
	v_{n}(k,\ell) & := \begin{cases}
	\nu\cdot\kappa_n(i), & \text{if }k=0, \ell =i,
	~~ i\in \{1,\dots,Q\},\\
	0, & \text{otherwise for }k\neq\ell,
	\end{cases}\quad n\geq0.
	\end{align*}
	$Z$ is irreducible on $E$ and ergodic by $|E|<\infty$. Therefore, a unique stationary distribution $\pi=(\pi(n,k):(n,k)\in E)$ exists.
	\cite{melikov;molchanov:92} realized that there is no directly accessible solution $\pi$, which seems to be due to the dependence
	of the order size distribution on the queue length. Fortunately enough, for the case of state-independent order size distribution $\kappa:=(\kappa(i):i=1,\dots,Q)$
	the stationary distribution in case of lost sales is obtained in Theorem 6.2 in \cite{schwarz;sauer;daduna;kulik;szekli:06}.
	It holds
	with normalization constant $K$
	\begin{align*}
	\pi(n,0)&=K^{-1} \cdot \left(\frac{\lambda}{\mu}\right)^n\cdot\frac{\lambda}{\nu},
	& 0\leq n\leq N,\\
	\pi(n,k)&=K^{-1} \cdot\left(\frac{\lambda}{\mu}\right)^n\cdot
	\left(\sum_{h=k}^{Q} \kappa(h)\right),& 0\leq n\leq N+1,~~~
	1\leq k\leq Q.
	\end{align*}
	We remark that in the production-inventory system of \cite{melikov;molchanov:92}
	the reorder level is $q\in \{0,1,\dots,Q-1\}.$ When this inventory level is attained, no service is performed until the replenishment  arrived. Therefore, the number $q$ may be interpreted as safety stock which has to be maintained in any case and inventory levels below $q$ need not be incorporated into the process description under the lost sales regime. 
\end{example}	

The stationary distribution in Example \ref{ex:QueueInventory1Finite} has a ``product form'' because $\pi(n,k)$ is composed of two factors. This does not indicate separability of the model because the state space is not a product space. This is
required for a two-dimensional distribution with independent coordinates.

%
%

\section{Non-separable queueing-environment systems}\label{sect:Non-separable}
Ergodicity in case of separable queueing-environment systems is in most cases easy to detect because of the product  structure for the solution of the global balance equations
of the process.
This is in general not the case for non-separable systems which are dealt with in this section.
We provide an exception from this general statement at the end of Subsection  \ref{sect:ergodicity-Lyapunov} in Example \ref{ex:QueueInventory2} and Corollary \ref{cor:QueueInventoryPerishable-PF} by considering a slight modification of Example  \ref{ex:QueueInventory1} and compute  in Corollary \ref{cor:QueueInventoryPerishable-PF} the solution of the global balance equations which is available  for this example but not of product form.
Nevertheless, from the explicit solution ergodicity can be proved directly.
In this section the focus is on systems where this is not possible because an explicit expression for the solution of the global balance equations of $Z=(X,Y)$ seems to be out of reach. So the criterion of summability of that solution is not applicable for proving (exponential) ergodicity.

Instead we construct Lyapunov functions (drift functions) for verifying (exponential) ergodicity.
Usually, such a construction is not an easy task but we succeeded in both cases with constructing Lyapunov functions to apply the relevant Propositions \ref{prop:KellyYudovina14} and \ref{prop:Anderson91}.

Our guiding principle in the construction is the following.
For $Z$ to be ergodic the queueing component $X$ in isolation, i.e.~a birth-death process with associated rates $\lambda(n),\mu(n)$ should be ergodic with some suitable Lyapunov function $\widetilde{\mathcal{L}}: \mathbb{N}_0\to [0,\infty)$.
Then we construct a 2-dimensional Lyapunov function 
${\mathcal{L}}: \mathbb{N}_0\times K\to [0,\infty)$ where the first coordinate is (roughly) a modified version of $\widetilde{\mathcal{L}}$ and attach a queue-length-dependent second coordinate function.
Clearly, this is the main difficulty because we cannot expect to find Lyapunov functions for the second (environment) components in isolation because the generators $V_n$ and the jump transition matrices $R_n$ are in general neither irreducible nor ergodic.
For exponential ergodicity we proceed in an analogous way.

We  start this section with a necessary condition for ergodicity in Proposition \ref{prop:random-abh-notw}
which strongly supports our guiding principle described above.
In Section \ref{sect:ergodicity-Lyapunov} for the case of finite environment space   sufficient conditions for positive recurrence are proved in \prettyref{thm:BS-PER-GER-pos-rec} and Corollary \ref{cor:BS-PER-GER-pos-rec-2}. 
Exponential ergodicity is investigated in Section \ref{sect:expo-ergodicity-Lyapunov}.\\

\subsection{A necessary condition for ergodicity}\label{sect:ergodicity-necessary}

We start with a proposition which is of independent interest because it underpins the importance of
the environment's structure and its stalling feature for the queueing system. We emphasize that in this subsection the environment space is allowed to be countably infinite.
%
%
\begin{proposition}
	\label{prop:random-prop-1}If the queueing-environment process $Z$
	is ergodic, then the solution $\mathbf{x}:= \left(x\left(z\right):z\in E\right)$
	of the global balance equations $\mathbf{x}\mathbf{\cdot Q=0}$ fulfils
	for all $n\in\mathbb{N}_{0}$
	\begin{equation}
	\sum_{k\in K_{W}}x(n,k)=\sum_{k\in K_{W}}x(n+1,k)\cdot\frac{\mu(n+1)}{\lambda(n)}\label{eq:BS-PER-GEN-PXnYgreater0(2)-1a}
	\end{equation}
	and consequently, we have a geometrical structure
	inherent in the stationary distribution
	\begin{equation}
	\sum_{k\in K_{W}}x(n,k)=
	\left(\sum_{k\in K_{W}} x(0,k)\right)\cdot
	\prod_{m=1}^{n}\frac{\lambda(m-1)}{\mu(m)},
	\quad n\in\mathbb{N}_{0}.\label{eq:BS-PER-GEN-PXnYgreater0(1)-1a}
	\end{equation}
\end{proposition}
%
\proof
By ergodicity  there is a unique
strictly positive stationary probability distribution
$\mathbf{x}:= \left(x\left(z\right):z\in E\right)$  as  solution
of the global balance equations $\mathbf{x}\mathbf{\cdot Q=0}$, see e.g.~\cite[Theorem 4.2,  p. 51]{asmussen:03}.
We apply the cut-criterion to $\mathbf{x}$ \cite[Lemma 1.4]{kelly:79}. This criterion states that for the stationary distribution $\mathbf{x}$ and complementary sets $A, A^c\subset E$   
the probability flows between these sets balance. 
We apply the criterion to 
\[
A:=\Big\{(m,k):m\in\{0,1,\ldots,n\},\: k\in K\Big\}
~\text{and}~
A^c=	\left\{ (\widetilde{m},\widetilde{k}):\widetilde{m}\in\mathbb{N}_{0}\setminus\{0,1,\ldots,n\},\:\widetilde{k}\in K\right\} ,~~ n\in\mathbb{N}_{0}.
\]
Balancing the probability flows between these sets yields
\[
\sum_{m=0}^{n}\sum_{k\in K}\ \sum_{\widetilde{m}=n+1}^{\infty}\sum_{\widetilde{k}\in K}x(m,k)\cdot q((m,k);(\widetilde{m},\widetilde{k}))
=\sum_{\widetilde{m}=n+1}^{\infty}\sum_{\widetilde{k}\in K}\ \sum_{m=0}^{n}\sum_{k\in K}x(\widetilde{m},\widetilde{k})\cdot q((\widetilde{m},\widetilde{k});(m,k)),
\]
which simplifies to
\begin{align*}
& \sum_{k\in K_{W}}x(n,k)\cdot\lambda(n)=
\sum_{\widetilde{k}\in K_{W}}x(n+1,\widetilde{k})\cdot\mu(n+1)\cdot
\underbrace{\sum_{\ell\in K}r_{n+1}(\widetilde{k},\ell)}_{=1}.
\end{align*}
This is \eqref{eq:BS-PER-GEN-PXnYgreater0(2)-1a}. Equation \eqref{eq:BS-PER-GEN-PXnYgreater0(1)-1a} follows directly. 
\endproof

The next result shows that it is impossible to stabilize a non-ergodic (isolated) queue by embedding it into a suitably constructed environment. 
Additionally, the result demonstrates that in an ergodic queueing-environment process the up and down rates for the queueing component constitute necessarily an ergodic birth-death process.
%
\begin{proposition}
	\label{prop:random-abh-notw}If the queueing-environment process $Z$
	is ergodic, it holds
	\begin{equation}\label{eq:birth-death-constant-finite}
	\sum_{n=0}^{\infty}\prod_{m=1}^{n}\frac{\lambda(m-1)}{\mu(m)}<\infty.
	\end{equation}
\end{proposition}
%
\proof
Ergodicity implies that any solution $\mathbf{x}:= \left(x\left(z\right):z\in E\right)$
of the global balance equations 
fulfils $\sum_{n=0}^{\infty}\sum_{k\in K}x(n,k)<\infty.$
It holds
\begin{align*}
\sum_{n=0}^{\infty}\sum_{k\in K}x(n,k) &{=}\sum_{n=0}^{\infty}\sum_{k\in K_{W}}x(n,k)+\sum_{n=0}^{\infty}\sum_{k\in K_{B}}x(n,k)\\
&{=}\sum_{n=0}^{\infty}\underbrace{\sum_{k\in K_{W}}x(0,k)}_{=:\widetilde{W}}\cdot\prod_{m=1}^{n}\frac{\lambda(m-1)}{\mu(m)}+\sum_{n=0}^{\infty}\sum_{k\in K_{B}}x(n,k)\\
&{=}\widetilde{W}\cdot\sum_{n=0}^{\infty}\prod_{m=1}^{n}\frac{\lambda(m-1)}{\mu(m)}+\sum_{n=0}^{\infty}\sum_{k\in K_{B}}x(n,k).
\end{align*}
By ergodicity, it holds $\widetilde{W}\in\left(0,\infty\right)$ and
$\sum_{n=0}^{\infty}\sum_{k\in K_{B}}x(n,k)<\infty$.
Hence, $\sum_{n=0}^{\infty}\sum_{k\in K}x(n,k)<\infty$ implies
$\sum_{n=0}^{\infty}\prod_{m=1}^{n}\frac{\lambda(m-1)}{\mu(m)}<\infty$.
\endproof

\subsection{Ergodicity via Lyapunov functions}\label{sect:ergodicity-Lyapunov}
We follow a standard approach
constructing Lyapunov functions to apply Foster-Lyapunov criterion, see Proposition \ref{prop:KellyYudovina14}.
\begin{assumption}\label{ass:K-finite}
	Henceforth we assume  that the environment space $K$ is finite. 
\end{assumption}
We start with three preparatory lemmas. The proof of the first one is by direct computation.
%
%
\begin{lemma}
	\label{lem:BS-PER-GEN-Lemma1}Consider an $M/M/1/\infty$-queue with queue-length-dependent
	arrival rates $\lambda(n)>0$ and service rates $\mu(n)>0$.
	If
	$\widetilde{\mathcal{L}}:\mathbb{N}_{0}\rightarrow\mathbb{R}_{0}^{+}$ is a Lyapunov function for the queue length process  with finite exception set $\widetilde{F}$ and constant $\widetilde{\varepsilon}>0$ which satisfies the Foster-Lyapunov stability criterion from Proposition \ref{prop:KellyYudovina14}, the following inequalities are satisfied:
	\begin{align}
	& \lambda(0)\cdot\left(\widetilde{\mathcal{L}}(1)-\widetilde{\mathcal{L}}(0)\right)\leq-\widetilde{\varepsilon}, &  & 0\notin\widetilde{F},\label{eq:random-mm1-eq1}\\
	& \lambda(n)\cdot\left(\widetilde{\mathcal{L}}(n+1)-\widetilde{\mathcal{L}}(n)\right)+\mu(n)\cdot\left(\widetilde{\mathcal{L}}(n-1)-\widetilde{\mathcal{L}}(n)\right)\leq -\widetilde\varepsilon, &  & n\notin\widetilde{F},\ n>0.\label{eq:random-mm1-eq2}
	\end{align}
\end{lemma}

For the Markovian queueing-environment process $Z=(X,Y)$ of Section \ref{sect:model} with generator 
$\mathbf{Q}=({q}(z,\tilde z):z, \tilde z\in E)$ from \eqref{generator1} we define for every 
$n\in \mathbb{N}_{0}$ an artificial Markov process
$Y^{(n)}:= (Y^{(n)}(t):t\geq 0)$ on $K$ with generator $V_n=(v_n(k,\ell):k,\ell\in K)$. 
The processes $Y^{(n)}$ are in general neither irreducible, nor recurrent.
\label{page:YnDef}

By definition the stochastic behaviour of $Y^{(n)}$ when started in $Y^{(n)}(0)=  k\in K_B$ and observed  until the first entrance into $K_W$ 
is identical to the behaviour of  the $Y$-component of  $Z= (X,Y)$ on $\{n\}\times K$ when started in
$(X(0),Y(0))= (n,k)\in E$ and observed until the first entrance into $\{n\}\times K_W$. Especially, until this first entrance of $(X,Y)$ into $\{n\}\times K_W$ the first coordinate of $(X,Y)$ is constant $n$.

Denote by $T_n$ the first-entrance time of $Y^{(n)}$ into $K_W$, which is a $[0,\infty]$-valued random variable.
The function  $\tau_{n}:K\rightarrow \mathbb{R}_{0}^{+}$ with  $\tau_{n}(k):= \mathbb{E}[T_n|Y^{(n)}(0)=k]$
for $n\in\mathbb{N}_{0}$
is the mean first-entrance time of $Y^{(n)}$ into $K_{W}$ when starting in $k\in K$ (conditional mean absorption time in $K_W$).
For $\ell\in K_{W}$ we have $\tau_{n}(\ell)= 0$, indicating that absorption has already happened.

%
%
\begin{lemma}\label{lem:BS-PER-GEN-Lemma2}
	For all $n\in \mathbb N_0$ it holds $0< \tau_{n}(k) < \infty$ for $k\in K_{B}$ and we have a set of first-entrance equations
	\[
	\tau_{n}(k)=\frac{1}{-v_{n}(k,k)}+\sum_{\ell\in K_{B}\setminus\{k\}}\frac{v_{n}(k,\ell)}{-v_{n}(k,k)}\cdot\tau_{n}(\ell),
	\]
	which are equivalent to
	\begin{equation}
	-1=\sum_{\ell\in K_{B}\setminus\{k\}}v_{n}(k,\ell)\cdot\left(\tau_{n}(\ell)-\tau_{n}(k)\right)-\sum_{\ell\in K_{W}}v_{n}(k,\ell)\cdot\tau_{n}(k),\qquad n\in\mathbb{N}_{0}.\label{eq:BS-PER-GEN-eins}
	\end{equation}
\end{lemma}
%
\proof
Positivity of $\tau_{n}$ on $K_B$ is due to the regularity of the $Y^{(n)}$, which follows from regularity of $(X,Y)$. Irreducibility of $(X,Y)$ implies that for any $(n,k)\in \{n\}\times K_B$ there exists some
state $(n,\ell)\in \{n\}\times K_W$ which can be reached in a finite number of jumps. This implies that  $Y^{(n)}$
until absorption in $K_W$	 can be considered as a finite state process with attached single absorbing state $a$. Consequently, time to absorption  of $Y^{(n)}$ in $a$ has  finite mean for any initial state.

The set $\left(\tau_{n}(k):k\in K_{B}\right)$ satisfies the following
set of first-entrance equations:
\begin{align}
\tau_{n}(k) & \overset{\phantom{(\star)}}{=}\sum_{\ell\in K_{B}\setminus\{k\}}\frac{v_{n}(k,\ell)}{-v_{n}(k,k)}\cdot\left(\frac{1}{-v_{n}(k,k)}
+\tau_{n}(\ell)\right) +\sum_{\ell\in K_{W}}\frac{v_{n}(k,\ell)}{-v_{n}(k,k)}\cdot\left(\frac{1}{-v_{n}(k,k)}+\tau_{n}(\ell)\right)\nonumber \\
& \overset{\phantom{(\star)}}{=}\sum_{\ell\in K\setminus\{k\}}\frac{v_{n}(k,\ell)}{-v_{n}(k,k)}\cdot\frac{1}{-v_{n}(k,k)}\nonumber
+\sum_{\ell\in K_{B}\setminus\{k\}}\frac{v_{n}(k,\ell)}{-v_{n}(k,k)}\cdot\tau_{n}(\ell)+\sum_{\ell\in K_{W}}\frac{v_{n}(k,\ell)}{-v_{n}(k,k)}\cdot\underbrace{\tau_{n}(\ell)}_{=0}\nonumber \\
& \overset{(\star)}{=}\frac{1}{-v_{n}(k,k)}+\sum_{\ell\in K_{B}\setminus\{k\}}\frac{v_{n}(k,\ell)}{-v_{n}(k,k)}\cdot\tau_{n}(\ell),\qquad\qquad\qquad n\in\mathbb{N}_{0}.\label{eq:rand-lemma-taunk}
\end{align}
$(\star)$ holds because $V_{n}=\left(v_{n}(k,\ell):k,\ell\in K\right)$ is conservative.
Equation \eqref{eq:rand-lemma-taunk} is equivalent to
\begin{align*}
-1 & =\left(\sum_{\ell\in K_{B}\setminus\{k\}}v_{n}(k,\ell)\cdot\tau_{n}(\ell)\right)+v_{n}(k,k)\cdot\tau_{n}(k)
=\sum_{\ell\in K_{B}\setminus\{k\}}v_{n}(k,\ell)\cdot\tau_{n}(\ell)-\sum_{\ell\in K\setminus\{k\}}v_{n}(k,\ell)\cdot\tau_{n}(k)\\
& =\sum_{\ell\in K_{B}\setminus\{k\}}v_{n}(k,\ell)\cdot\left(\tau_{n}(\ell)-\tau_{n}(k)\right)-\sum_{\ell\in K_{W}}v_{n}(k,\ell)\cdot\tau_{n}(k),\qquad n\in\mathbb{N}_{0}.\qquad\qquad\qquad
\end{align*} 
\endproof
%

\begin{lemma}
	{\label{lem:random-cn}}We define for $n\in\mathbb{N}_{0}$
	\begin{align}
	{c}_{n} & := \min\left\{ \frac{1}{\max_{k\in K_{W}}\left\{ \mu(n+1)\cdot\sum_{\ell\in K_{B}}r_{n+1}(k,\ell)\cdot\tau_{n}(\ell)\right\} }, \:\frac{1}{\max_{k\in K_{W}}\left\{ \sum_{\ell\in K_{B}}v_{n}(k,\ell)\cdot\tau_{n}(\ell)\right\} }\right\} .\label{eq:BS-PER-GEN-cn}
	\end{align}
	The ${c}_{n}$ are well-defined, i.e.~it holds $0 < {c}_{n}<\infty$ for
	$n\in\mathbb{N}_{0}$.
\end{lemma}
\proof \label{rem:cn-kleiner}
\textbf{(i)} In \prettyref{lem:BS-PER-GEN-Lemma2} we have shown that the mean
first-entrance times $\tau_{n}(k), k\in K_{B},$ are positive and finite. All  other quantities which occur are 
positive and finite by definition. So ${c}_{n}>0$.\\
\textbf{(ii)} 
From $\mu(n+1)>0$ and $\tau_{n}(\ell)>0, \ell\in K_{B},$ follows
\begin{equation}\label{eq:cnLess}
{c}_{n}<\infty \quad\Leftrightarrow\quad\exists\ k\in K_{W}, \ell\in K_{B}: r_{n+1}(k,\ell)\neq 0 \vee v_{n}(k,\ell)\neq0.	
\end{equation}	
Take any pairs $(\widetilde{n},k)\in \mathbb{N}_0\times K_{W}$ and $(n,\ell)\in \mathbb{N}_0\times K_{B}$. By irreducibility of $Z$ there exists a finite path (sequence of jumps with positive probability) from $(\widetilde{n},k)$ to $(n,\ell)$. 
Without loss of generality we can assume that $(n,\ell)$ is the first state of that path which is in $\mathbb{N}_0\times K_{B}$.

Since $\ell\in K_B$, we note that for any $m\in\mathbb{N}_0$  in $(m,\ell)$ arrival and service processes are stalled and that arrivals in state  $(n-1,k)$ cannot trigger a change of the environment to $\ell$. So the only possible transitions out of environment state $k$ which lead to $(n,\ell)$ are of the form
\begin{align}
&\mathbb{N}\times K_W\ni (n+1,k)\stackrel{\mu(n+1)\cdot r_{n+1}(k,\ell)}{\longrightarrow}(n,\ell)\in \mathbb{N}_0\times K_B,\label{eq:trans1}\\
&\mathbb{N}_0\times K_W\ni (n,k)
\stackrel{v_{n}(k,\ell)}{\longrightarrow}(n,\ell)\in \mathbb{N}_0\times K_B.\label{eq:trans2}
\end{align}
Consequently, either $r_{n+1}(k,\ell)$ in \eqref{eq:trans1} or 
$v_{n}(k,\ell)$ in \eqref{eq:trans2} must be strictly positive
to terminate the path from $(\widetilde{n},k)$ to $(n,\ell)$.

\endproof

\begin{theorem}\label{thm:BS-PER-GER-pos-rec}
	Consider the queueing-environment process $Z$ with finite environment set $K$. Assume that 
	\[
	\widetilde{\mathcal{L}}:\mathbb{N}_{0}\rightarrow\mathbb{R}_{0}^{+}
	\]
	is a Lyapunov function with  finite exception set $\widetilde{F}$ and constant $\widetilde{\varepsilon}>0$ for the $M/M/1/\infty$-queue with queue-length-dependent arrival rates $\lambda(n)\in (0,\infty)$ and service rates $\mu(n)\in (0,\infty)$.
	So, the queue length process in isolation of the system is ergodic.
	Assume further that 
	\[
	\inf_{n\in\mathbb{N}_{0}}{c}_{n}>0, 
	\]
	holds, where ${c}_{n}$ is defined in  Lemma \ref{lem:random-cn}.
	Then 
	\begin{equation}
	\mathcal{L}:E\rightarrow\mathbb{R}_{0}^{+}~~\text{with}~~ \mathcal{L}(n,k):= \widetilde{\mathcal{L}}(n)+1_{\left\{ k\in K_{B}\right\} }\cdot c_{n}\cdot\frac{\widetilde{\varepsilon}}{4}\cdot\tau_{n}(k) \nonumber 
	\end{equation}
	is a Lyapunov function for $Z$ with
	finite exception set $F:= \widetilde{F}\times K$
	and constant
	\[
	\varepsilon:= \min\left(\frac{\widetilde{\varepsilon}}{2}\ ,\ \frac{\widetilde{\varepsilon}}{4}\cdot\inf_{n\in\mathbb{N}_{0}} c_{n}\right)>0
	\]
	and $Z$ is  ergodic.
\end{theorem}
%
\proof
We apply Proposition \ref{prop:KellyYudovina14} and show that $\mathcal{L}$ is a Lyapunov function for $Z$ with finite exception set $F$ and constant $\varepsilon$.

First, because  the jumps of $X$ are of height $1$, and because $|K|<\infty$, to check $\left(\mathbf{Q\cdot\mathcal{L}}\right)\left(n,k\right)<\infty$
for $(n,k)\in F=\widetilde{F}\times K$ is by direct computation.

Secondly, we will check $\left(\mathbf{Q\cdot\mathcal{L}}\right)\left(n,k\right)\leq-\varepsilon$ for $(n,k)\notin F=\widetilde{F}\times K$:\\
$\blacktriangleright$ For $k\in K_{B}$ and $n\notin \widetilde{F}$, $n\geq0$, it holds
\begin{align*}
& \left(\mathbf{Q\cdot\mathcal{L}}\right)\left(n,k\right)
=\sum_{\ell\in K\setminus\{k\}}v_{n}(k,\ell)\cdot\left(\mathcal{L}(n,\ell)-\mathcal{L}(n,k)\right)\\
& =\sum_{\ell\in K\setminus\{k\}}v_{n}(k,\ell)\cdot\left(\left(\widetilde{\mathcal{L}}(n)+1_{\left\{ \ell\in K_{B}\right\} }\cdot c_{n}\cdot\frac{\widetilde{\varepsilon}}{4}\cdot\tau_{n}(\ell)\right)-\left(\widetilde{\mathcal{L}}(n)+1_{\left\{ k\in K_{B}\right\} }\cdot c_{n}\cdot\frac{\widetilde{\varepsilon}}{4}\cdot\tau_{n}(k)\right)\right)\\
& =\sum_{\ell\in K\setminus\{k\}}v_{n}(k,\ell)\cdot1_{\left\{ \ell\in K_{B}\right\} }\cdot c_{n}\cdot\frac{\widetilde{\varepsilon}}{4}\cdot\tau_{n}(\ell)-\sum_{\ell\in K\setminus\{k\}}v_{n}(k,\ell)\cdot\underbrace{1_{\left\{ k\in K_{B}\right\} }}_{=1}c_{n}\cdot\frac{\widetilde{\varepsilon}}{4}\cdot\tau_{n}(k)\\
& =\sum_{\ell\in K_{B}\setminus\{k\}}v_{n}(k,\ell)\cdot c_{n}\cdot\frac{\widetilde{\varepsilon}}{4}\cdot\tau_{n}(\ell)-\sum_{\ell\in K\setminus\{k\}}v_{n}(k,\ell)\cdot c_{n}\cdot\frac{\widetilde{\varepsilon}}{4}\cdot\tau_{n}(k)\\
& =\sum_{\ell\in K_{B}\setminus\{k\}}v_{n}(k,\ell)\cdot\left(c_{n}\cdot\frac{\widetilde{\varepsilon}}{4}\cdot\tau_{n}(\ell)-c_{n}\cdot\frac{\widetilde{\varepsilon}}{4}\cdot\tau_{n}(k)\right)-\sum_{\ell\in K_{W}}v_{n}(k,\ell)
\cdot\frac{\widetilde{\varepsilon}}{4}\cdot\tau_{n}(k)\\
& =c_{n}\cdot\frac{\widetilde{\varepsilon}}{4}\cdot\underbrace{\Bigg[\sum_{\ell\in K_{B}\setminus\{k\}}v_{n}(k,\ell)\cdot\left(\tau_{n}(\ell)-\tau_{n}(k)\right)-\sum_{\ell\in    K_{W}}v_{n}(k,\ell)\cdot\tau_{n}(k)\Bigg]}_{{=}-1, {~~\text{by}~ \eqref{eq:BS-PER-GEN-eins}}}\\
& =-c_{n}\cdot\frac{\widetilde{\varepsilon}}{4}\leq -\min\left(\frac{\widetilde{\varepsilon}}{2}\ ,\ \frac{\widetilde{\varepsilon}}{4} \cdot \inf_{m\in\mathbb{N}_{0}} c_{m}\right)
=-\varepsilon.\end{align*}
$\blacktriangleright$ For $k\in K_{W}$ and $n=0\notin \widetilde{F}$ it holds
\begin{align*}
& \phantom{\leq}\left(\mathbf{Q\cdot\mathcal{L}}\right)\left(n,k\right)
=\lambda(0)\cdot\left(\mathcal{L}(1,k)-\mathcal{L}(0,k)\right)
+\sum_{\ell\in K\setminus\{k\}}v_{0}(k,\ell)
\cdot\left(\mathcal{L}(0,\ell)-\mathcal{L}(0,k)\right)\\
& =\lambda(0)\cdot\Bigg(\Bigg(\widetilde{\mathcal{L}}(1)+\underbrace{1_{\left\{ k\in K_{B}\right\} }}_{=0}\cdot c_{1}\cdot\frac{\widetilde{\varepsilon}}{4}\cdot\tau_{1}(k)\Bigg)-\Bigg(\widetilde{\mathcal{L}}(0)+\underbrace{1_{\left\{ k\in K_{B}\right\} }}_{=0}\cdot c_{0}\cdot\frac{\widetilde{\varepsilon}}{4}\cdot\tau_{0}(k)\Bigg)\Bigg)\\
& \phantom{\leq}+\sum_{\ell\in K\setminus\{k\}}v_{0}(k,\ell)\cdot\Bigg(\Bigg(\widetilde{\mathcal{L}}(0)+1_{\left\{ \ell\in K_{B}\right\} }\cdot c_{0}\cdot\frac{\widetilde{\varepsilon}}{4}\cdot\tau_{0}(\ell)\Bigg)-\Bigg(\widetilde{\mathcal{L}}(0)+\underbrace{1_{\left\{ k\in K_{B}\right\} }}_{=0}\cdot c_{0}\cdot\frac{\widetilde{\varepsilon}}{4}\cdot\tau_{0}(k)\Bigg)\Bigg)\\
& =\lambda(0)\cdot\left(\widetilde{\mathcal{L}}(1)-\widetilde{\mathcal{L}}(0)\right)+\sum_{\ell\in K\setminus\{k\}}v_{0}(k,\ell)\cdot 1_{\left\{ \ell\in K_{B}\right\} }\cdot c_{0}\cdot\frac{\widetilde{\varepsilon}}{4}\cdot\tau_{0}(\ell)\\
& =\underbrace{\lambda(0)\cdot\left(\widetilde{\mathcal{L}}(1)-
	\widetilde{\mathcal{L}}(0)\right)}_{{\leq}-\widetilde{\varepsilon}~~{(\diamondsuit)}}+\sum_{\ell\in K_{B}}v_{0}(k,\ell)\cdot c_{0}\cdot\frac{\widetilde{\varepsilon}}{4}\cdot\tau_{0}(\ell)\\
& \leq-\widetilde{\varepsilon}+\underbrace{\sum_{\ell\in K_{B}}v_{0}(k,\ell)\cdot c_{0}\cdot\frac{\widetilde{\varepsilon}}{4}\cdot\tau_{0}(\ell)}_{\leq\frac{\widetilde{\varepsilon}}{4},~\text{by } \eqref{eq:BS-PER-GEN-cn}}\leq-\frac{3}{4}\cdot\widetilde{\varepsilon}\leq-\min\left(\frac{\widetilde{\varepsilon}}{2}\ ,\ \frac{\widetilde{\varepsilon}}{4} \cdot \inf_{m\in\mathbb{N}_{0}} c_{m}\right) = -\varepsilon.
\end{align*}
$(\diamondsuit)$ holds because of \eqref{eq:random-mm1-eq1}
since $\widetilde{\mathcal{L}}:E\rightarrow\mathbb{R}_{0}^{+}$ is
a Lyapunov function for the $M/M/1/\infty$-queue with queue-length-dependent
arrival and service rates with constant $\widetilde{\varepsilon}$.\\
$\blacktriangleright$ For $k\in K_{W}$ and $n\notin \widetilde{F}$, $n>0$, it holds
\begin{align*}
& \phantom{\leq}\left(\mathbf{Q\cdot\mathcal{L}}\right)\left(n,k\right)
=\lambda(n)\cdot\left(\mathcal{L}(n+1,k)-\mathcal{L}(n,k)\right)\\
& \phantom{\leq}+\sum_{\ell\in K}\mu(n)\cdot r_{n}(k,\ell)\cdot\left(\mathcal{L}(n-1,\ell)-\mathcal{L}(n,k)\right)
+\sum_{\ell\in K\setminus\{k\}}v_{n}(k,\ell)\cdot\left(\mathcal{L}(n,\ell)-\mathcal{L}(n,k)\right)\\
& =\lambda(n)\cdot\Bigg(\Bigg(\widetilde{\mathcal{L}}(n+1)+\underbrace{1_{\left\{ k\in K_{B}\right\} }}_{=0}\cdot c_{n+1}\cdot\frac{\widetilde{\varepsilon}}{4}\cdot\tau_{n+1}(k)\Bigg)-\Bigg(\widetilde{\mathcal{L}}(n)+\underbrace{1_{\left\{ k\in K_{B}\right\} }}_{=0}\cdot c_{n}\cdot\frac{\widetilde{\varepsilon}}{4}\cdot\tau_{n}(k)\Bigg)\Bigg)\\
& \phantom{\leq}+\sum_{\ell\in K}\mu(n)\cdot r_{n}(k,\ell)\cdot\Bigg(\Bigg(\widetilde{\mathcal{L}}(n-1)+1_{\left\{ \ell\in K_{B}\right\} }\cdot c_{n-1}\cdot\frac{\widetilde{\varepsilon}}{4}\cdot\tau_{n-1}(\ell)\Bigg)\\
& \hphantom{+\sum_{\ell\in K}\mu(n)\cdot R_{n}(k,\ell)\cdot\Bigg(\quad}-\Bigg(\widetilde{\mathcal{L}}(n)+\underbrace{1_{\left\{ k\in K_{B}\right\} }}_{=0}\cdot c_{n}\cdot\frac{\widetilde{\varepsilon}}{4}\cdot\tau_{n}(k)\Bigg)\Bigg)\\
& \phantom{\leq}+\sum_{\ell\in K\setminus\{k\}}v_{n}(k,\ell)\cdot\Bigg(\Bigg(\widetilde{\mathcal{L}}(n)+1_{\left\{ \ell\in K_{B}\right\} }\cdot c_{n}\cdot\frac{\widetilde{\varepsilon}}{4}\cdot\tau_{n}(\ell)\Bigg)-\Bigg(\widetilde{\mathcal{L}}(n)+\underbrace{1_{\left\{ k\in K_{B}\right\} }}_{=0}\cdot c_{n}\cdot\frac{\widetilde{\varepsilon}}{4}\cdot\tau_{n}(k)\Bigg)\Bigg)\\
& =\lambda(n)\cdot\left(\widetilde{\mathcal{L}}(n+1)-
\widetilde{\mathcal{L}}(n)\right)+\mu(n)\cdot\underbrace{\sum_{\ell\in K}r_{n}(k,\ell)}_{=1}\cdot\left(\widetilde{\mathcal{L}}(n-1)-\widetilde{\mathcal{L}}(n)\right)\\
& \phantom{\leq}+\mu(n)\cdot\sum_{\ell\in K}r_{n}(k,\ell)\cdot1_{\left\{ \ell\in K_{B}\right\} }\cdot c_{n-1}\cdot\frac{\widetilde{\varepsilon}}{4}\cdot\tau_{n-1}(\ell)+\sum_{\ell\in K\setminus\{k\}}v_{n}(k,\ell)\cdot 1_{\left\{ \ell\in K_{B}\right\} }\cdot c_{n}\cdot\frac{\widetilde{\varepsilon}}{4}\cdot\tau_{n}(\ell)\\
& =\underbrace{\lambda(n)\cdot\left(\widetilde{\mathcal{L}}(n+1)-\widetilde{\mathcal{L}}(n)\right)+\mu(n)
	\cdot\left(\widetilde{\mathcal{L}}(n-1)-
	\widetilde{\mathcal{L}}(n)\right)}_{{\leq}-\widetilde{\varepsilon}~~{(\triangle)}}\\
& \phantom{\leq}+\mu(n)\cdot\sum_{\ell\in K_{B}}r_{n}(k,\ell)\cdot c_{n-1}\cdot\frac{\widetilde{\varepsilon}}{4}\cdot\tau_{n-1}(\ell)+\sum_{\ell\in K_{B}}v_{n}(k,\ell)\cdot c_{n}\cdot\frac{\widetilde{\varepsilon}}{4}\cdot\tau_{n}(\ell)\\
& \leq-\widetilde{\varepsilon}+\underbrace{\mu(n)\cdot\sum_{\ell\in K_{B}}r_{n}(k,\ell)\cdot c_{n-1}\cdot\frac{\widetilde{\varepsilon}}{4}\cdot\tau_{n-1}(\ell)}_{{\leq}
	\frac{\widetilde{\varepsilon}}{4},~\text{by }{\eqref{eq:BS-PER-GEN-cn}}}+\underbrace{\sum_{\ell\in K_{B}}v_{n}(k,\ell)\cdot c_{n}\cdot\frac{\widetilde{\varepsilon}}{4}\cdot\tau_{n}(\ell)}_{{\leq}\frac{\widetilde{\varepsilon}}{4},~\text{by }{\eqref{eq:BS-PER-GEN-cn}}}\\
& \leq-\frac{\widetilde{\varepsilon}}{2}\leq-\min\left(\frac{\widetilde{\varepsilon}}{2}\ ,\ \frac{\widetilde{\varepsilon}}{4} \cdot \inf_{m\in\mathbb{N}_{0}} c_{m}\right) = -\varepsilon.
\end{align*}
$(\triangle)$ holds because of \eqref{eq:random-mm1-eq2} since
$\widetilde{\mathcal{L}}:E\rightarrow\mathbb{R}_{0}^{+}$ is a Lyapunov
function for the $M/M/1/\infty$-queue with queue-length-dependent
arrival and service rates with constant $\widetilde{\varepsilon}$. 
\endproof
%

%
\begin{remark}\label{rem:inf-c-n}
	The positivity condition in Theorem \ref{thm:BS-PER-GER-pos-rec},
	\[
	\inf_{n\in\mathbb{N}_{0}}{c}_{n}>0,
	\]
	says roughly that the mean passage times through $K_B$ are uniformly (over all queue lengths $n\geq 0$) bounded.
	Any such passage through $K_B$, initiated when the environment is on leave  from some $k\in K_W$ by entering some $\ell\in K_B$,
	originates either from a jump movement (service completion which triggers an immediate jump)  $k\to\ell$
	governed by $r_{n+1}(k,\cdot)$, or from a continuous movement 
	from some $k\in K_W$ to some $\ell\in K_B$ 
	driven by $v_{n}(k,\cdot)$.
	In both cases the resulting queue length during the passage is $n$. Rewriting \eqref{eq:BS-PER-GEN-cn} as
	\begin{align*}
	{c}_{n} & =\min\left\{ \frac{1}{\max_{k\in K_{W}}\left\{ \mu(n+1)\cdot\sum_{\ell\in K_{B}}r_{n+1}(k,\ell)\cdot\tau_{n}(\ell)\right\} },\right.\nonumber \\
	& \hphantom{=\left\{ \min\ \ \right.}\left. \frac{1}{\max_{k\in K_{W}}
		\left\{ - v_{n}(k,k)\sum_{\ell\in K_{B}}\frac{v_{n}(k,\ell)}{- v_{n}(k,k)}\cdot\tau_{n}(\ell)\right\} }\right\} ,
	\end{align*}
	we observe that the mean passage times through $K_B$ are in both cases weighted by the entrance probabilities into $K_B$.
	Including the departure rates $ \mu(n+1)$ resp.~$- v_{n}(k,k)$ out of $k\in K_W$, ensures that the queueing-environment system resides in $K_W$ sufficiently long. 
\end{remark}
\begin{remark}\label{rem:LyaFctIndependence}
	The construction of the Lyapunov function with
	\[
	\mathcal{L}(n,k):=
	\widetilde{\mathcal{L}}(n)+1_{\left\{ k\in K_{B}\right\} }\cdot c_{n}\cdot\frac{\widetilde{\varepsilon}}{4}\cdot\tau_{n}(k)
	\]	
	shows that the terms
	$1_{\left\{ k\in K_{B}\right\} }\cdot c_{n}\cdot\frac{\widetilde{\varepsilon}}{4}\cdot\tau_{n}(k)$
	(relevant for the environment) are additively separated from the terms $\widetilde{\mathcal{L}}$ (relevant for the queue). 
	There is only an indirect coupling of the queue and the environment because the $c_n$ are functions which depend on the $\mu(n)$.
	Moreover, the only additional condition $\inf_{n\in \mathbb{N}_0} c_n>0$ also does not explicitly refer to the 
	Lyapunov function $\widetilde{\mathcal{L}}$ of the queue in isolation. 
\end{remark}

\begin{corollary}\label{cor:BS-PER-GER-pos-rec-2}
	Consider the queueing-environment process $Z$ with finite environment set $K$ and queue-length-dependent
	arrival rates $\lambda(n)\in(0,\infty)$ and service rates $\mu(n)\in(0,\infty)$.
	If \begin{equation} \label{eq:birth-death-constant-finite-1}
	\sum_{n=0}^{\infty}\prod_{m=1}^{n}\frac{\lambda(m-1)}{\mu(m)}<\infty
	\end{equation}
	holds, and if with ${c}_{n}$ from  Lemma \ref{lem:random-cn}, it holds
	\[
	\inf_{n\in\mathbb{N}_{0}}{c}_{n}>0,
	\]
	then  $Z$ is ergodic.
\end{corollary}
\proof
The isolated Markovian queue length process of the $M/M/1/\infty$-queue with rates  $\lambda(n), \mu(n)\in(0,\infty)$ is ergodic under \eqref{eq:birth-death-constant-finite-1}. With slightly abusing notation we denote this process by $X$ as well.
A Lyapunov function  for $X$ can be constructed as follows:
Take $\widetilde{F} := \{0\}$ as exception set and for $n\geq 1$
define $\widetilde{\mathcal{L}}(n)$ as the mean first-entrance time of $X$ into $\{0\}$ given $X(0)=n$, and set formally 
$\widetilde{\mathcal{L}}(0):= 0$. For $n\geq 1$ we have the mean first-entrance equations
\[
\widetilde{\mathcal{L}}(n)= \frac{1}{\lambda(n)+\mu(n)} +
\left(\frac{\lambda(n)}{\lambda(n)+\mu(n)}
\widetilde{\mathcal{L}}(n+1)+
\frac{\mu(n)}{\lambda(n)+\mu(n)}
\widetilde{\mathcal{L}}(n-1)\right),
\]
which is 
\begin{equation}\label{eq:LyapunovGuT-1} 
\sum_{m\in\{n-1,n+1\}}q(n;m)\left[\widetilde{\mathcal{L}}(m) -
\widetilde{\mathcal{L}}(n)
\right] = -1, \qquad n \notin \widetilde{F},
\end{equation}
and it holds furthermore for some $b> 0$
\begin{equation}\label{eq:LyapunovGuT-2}
q(n;n+1)\left[\widetilde{\mathcal{L}}(n+1) -
\widetilde{\mathcal{L}}(n)
\right] = \lambda(0) \widetilde{\mathcal{L}}(1) \leq b-1 , \qquad n=0\in \widetilde{F}.
\end{equation}
This says that 
$\widetilde{\mathcal{L}}$ constitutes by \eqref{eq:LyapunovGuT-1}  a Lyapunov function for $X$ in the sense of Proposition \ref{prop:KellyYudovina14}.
Because of \eqref{eq:birth-death-constant-finite-1},
$\widetilde{\mathcal{L}}(n)\in (0,\infty)$  holds for all $n\geq 1$,
see \cite[Corollary of Theorem 2, p. 214]{chung:67}.
We therefore can apply Theorem \ref{thm:BS-PER-GER-pos-rec}
to finish the proof. 
\endproof

\begin{remark}\label{rem:FossEtal}
	It is interesting to compare the result of Corollary \ref{cor:BS-PER-GER-pos-rec-2}, especially the condition
	\eqref{eq:birth-death-constant-finite-1}, with the structure of the results in \cite{foss;shneer;tyurlikov:12}, although the two-dimensional Markov processes are discrete time models on a general state space. For simplicity we concentrate on Section 2 of \cite{foss;shneer;tyurlikov:12}. The process 
	$Z=(X,Y)$ there constitutes a non-Markovian chain $Y$ in a random environment $X$ which is a Markov chain.  $X$ is therefore an autonomous environment, which is ergodic due to a given Lyapunov function. The evolution of $Y$ depends on $X$ via the $X$-dependent transition probabilities.
	
	If we consider in our running example of the production-inventory system (see Example \ref{ex:QueueInventory1})  the 
	\emph{queue, i.e.~$X$,  as the environment of the inventory, i.e.~$Y$}, then the condition 
	\eqref{eq:birth-death-constant-finite-1} seems to fix ergodicity for $X$ via the Lyapunov function $\widetilde{\mathcal{L}}$ as in the model of \cite{foss;shneer;tyurlikov:12}.
	
	The point is that although $\widetilde{\mathcal{L}}$ guarantees in some sense a drift condition for $X$, the standard drift approach of Markov theory is not applicable because this environment $X$ is not Markov as required in \cite{foss;shneer;tyurlikov:12}. So, our theorems deal with a completely different situation.
\end{remark}
\medskip

The following corollary and example shed some light on the range of possible classes of models subsumed in Theorem \ref{thm:BS-PER-GER-pos-rec} and Corollary \ref{cor:BS-PER-GER-pos-rec-2}.

\begin{corollary}
	\label{cor:random-cor-mm1-bsp}The queueing-environment process $Z$
	is ergodic if there exists $N\in\mathbb{N}_{0}$ such that
	\[
	 \inf_{n\geq N}\left(\mu(n)-\lambda(n)\right)>0\quad
	\text{and}\quad\inf_{n\in\mathbb{N}_{0}}{c}_{n}>0,
	\]
	where ${c}_{n}$ is defined in \prettyref{lem:random-cn}.
\end{corollary}
\proof
For the $M/M/1/\infty$-queue with queue-length-dependent
arrival rates $\lambda(n)>0$ and service rates $\mu(n)>0$ let there be $N\in\mathbb{N}_{0}$ such that $\inf_{n\geq N}\left(\mu(n)-\lambda(n)\right)>0$.
Then $\widetilde{\mathcal{L}}:\mathbb{N}_{0}\rightarrow\mathbb{R}_{0}^{+}$ with $\widetilde{\mathcal{L}}(n):= n$, finite exception set $\widetilde{F}:= \{0,1,\ldots,N-1\}$ and constant $\widetilde{\varepsilon}:= \inf_{n\geq N}\left(\mu(n)-\lambda(n)\right)>0$ is
a Lyapunov function, which satisfies the Foster-Lyapunov stability
criterion (Proposition \ref{prop:KellyYudovina14}). Hence, we can apply \prettyref{thm:BS-PER-GER-pos-rec}.  \\
\endproof
%

\begin{example}\label{ex:random-bsp-perish} If $\sup_{n\in \mathbb{N}_{0}}\mu(n)<\infty$, then in the following examples it holds  $\inf_{n\in\mathbb{N}_{0}}{c}_{n}>0$. It should be noted that ${c}_{n}$ is only defined by $\mu(n+1)$, the generator $V_n$ and the stochastic matrix $R_n$ (since $\tau_n$ is determined by the generator $V_n$).
	\begin{itemize}
		\item [\textbf{(a)}]For the generator   $V_{n}=\left(v_{n}(k,\ell):k,\ell\in K\right)$
		it	holds 
		\[
		V_{2n-1}:= V_{1}\quad\text{and}\quad V_{2n}:= V_{2},\qquad n\in\mathbb{N},
		\]
		and for the stochastic matrix $R_{n}=(r_{n}(k,\ell):k,\ell\in K$)
		it	holds
		\[
		R_{2n-1}:= R_{1}\quad\text{and}\quad R_{2n}:= R_{2},\qquad n\in\mathbb{N}.
		\]
		A similar structure is found in birth-death processes with alternating
		rates which are considered e.g.~by \cite{dicrescenzo:12, dicrescenzo:14}. 
		\item [\textbf{(b)}]Let $N_{0}\in \mathbb{N}$. For the generator $V_{n}=\left(v_{n}(k,\ell):k,\ell\in K\right)$
		it	holds
		\[
		V_{n}:= V,\quad n\geq N_{0},
		\]
		and for the stochastic matrix $R_{n}=(r_{n}(k,\ell):k,\ell\in K$)
		it	holds
		\[
		R_{n}:= R,\quad n\geq N_{0}.
		\]
		Then, ${c}_{n}$ can be arbitrarily for $n\leq N_{0}-1$ and
		from $N_{0}$ it is bounded below by a ${c}_{\min}>0$. 
		Similar  structures are found
		\begin{itemize}
			\item[$\bullet$] in multi-server models ($M/M/s$-queues)
			which are studied e.g.\ by \cite[Section 6.2, Section 6.5]{neuts:81} ,
			\item[$\bullet$] in a queue with $N$ servers subject
			to breakdowns and repairs, studied by~\cite{neuts;lucantoni:79},
			\item[$\bullet$] in the study of complex multi-server retrial models by~\cite{neuts;rao:90} who introduced simplifying approximations
			to obtain a system with 
			an infinitesimal generator with a modified matrix-geometric
			steady state vector. This could be  computed efficiently.
		\end{itemize}
		\item [\textbf{(c)}]The production-inventory system with perishable items in the following Example \ref{ex:QueueInventory2} is a special case of \textbf{(b)}  with $N_{0}=1$.
	\end{itemize}
\end{example}

%
\begin{example}[Production-inventory system with  perishable items]\label{ex:QueueInventory2}
	Often products like foodstuffs, human blood, chemicals, etc.~have a maximum lifetime, i.e.~when they are hold in  inventories, they may either perish,
	deteriorate, are subject to ageing, or become obsolete. 
	We consider the production-inventory system of Example \ref{ex:QueueInventory1} and Example \ref{ex:QueueInventory11}, respectively,  with the additional restriction that
	the lifetime of raw material in the inventory is exponentially
	distributed with ``ageing rate'' $\gamma>0$.
	In the literature, it is often assumed that an item of raw material, being already in the production process does not perish any longer
	(e.g.~\cite{manuel;sivakumar;arivarignan:07,manuel;sivakumar;arivarignan:08},~\cite{jeganathan:14},~\cite{yadavalli;anbazhagan;jeganathan:15a}). More complex systems with additional features are found in \cite{koroliuk;melikov;ponomarenko;rustamov:17,koroliuk;melikov;ponomarenko;rustamov:18}, 
	where ``stock that is already at distribution stage cannot perish''.
	We  incorporate in  a standard production-inventory system  this form of perishing, which implies:
	\begin{itemize}
		\item [$\bullet$]If $n>0$ and there are $k>0$ items of raw material  in the inventory, 
		then  one piece of raw material is in production and does not perish. Consequently, the total loss rate of inventory due to perishing is $\gamma\cdot\left(k-1\right)\cdot1_{\left\{ n>0\right\} }$.
		\item [$\bullet$]If $n=0$ and there are $k>0$ items of raw material in the inventory, then the total loss rate
		of inventory due to perishing is $\gamma\cdot k\cdot1_{\left\{ n=0\right\} }$.
	\end{itemize}
	
	We call the functions $k\mapsto\gamma\cdot k\cdot1_{\left\{ n=0\right\} }$ and\ $k\mapsto\gamma\cdot(k-1)_{+}\cdot1_{\left\{ n>0\right\} }$
	ageing regimes which determine the  queue-length-dependent overall loss rates of inventory due to perishing.
	
	This production-inventory system  fits into the definition of the queueing system in a random environment described by a Markov process $Z=(X,Y)=$(queue-length, inventory size) on state space $E=\mathbb{N}_0\times K$ with $K:= \{0,1,\dots,b\}$ and
	\[
	K_{B}=\left\{0\right\} ,\qquad K_{W}=\left\{ 1,\ldots,b\right\} .
	\]
	$Y(t)\in K_{W}$  indicates for the inventory that there is stock on hand for production, and $Y(t)\in K_{B}$ indicates stock-out. Note that the physical environment of the production system includes the replenishment system
	which consists of a single server with exponential-$\nu$ service times. The status of the replenishment server at time $t$ is uniquely determined by the size of the inventory as $b-Y(t)$.
	The dynamics of $Z$ are determined by the infinitesimal generator $\mathbf{Q}:=\left(q(z;\widetilde{z}):z,\widetilde{z}\in E\right)$ with the following transition rates for $(n,k)\in E =\mathbb{N}_{0}\times K$:
	\begin{align*}
	q\left((n,k);(n+1,k)\right) & := \lambda(n) \cdot1_{\left\{ k\in\left\{ 1,\ldots,b\right\}\right\} },\\
	q\left((n,k);(n,k-1)\right) & := \left(\gamma\cdot(k-1)\cdot1_{\left\{ n>0\right\} }+\gamma\cdot k\cdot1_{\left\{ n=0\right\} }\right)\cdot1_{\left\{ k\in\left\{ 1,\ldots,b\right\}\right\} },\\
	q\left((n,k);(n-1,k-1)\right) & := \mu(n) \cdot1_{\left\{ n>0\right\} }\cdot1_{\left\{ k\in\left\{ 1,\ldots,b\right\}\right\} },\\
	q\left((n,k);(n,k+1)\right) & := \nu\cdot1_{\left\{k \in\left\{ 0,1,\ldots,b-1\right\}\right\} },
	\end{align*}
	and $q(z;\tilde{z}):= 0$ for  other $z\neq\tilde{z}$,
	and
	$q\left(z;z\right):= -\sum_{\substack{\tilde{z}\in E,\\
			z\neq\tilde{z}} }q\left(z;\tilde{z}\right)$ for all $z\in E.$
	
	The queue-length-dependent dynamics of the inventory process are determined by 
	\[
	r_{n}(0,0):= 1,\qquad r_{n}(k,k-1):= 1,\quad k\in\left\{ 1,\ldots,b\right\},\ n\in\mathbb{N},
	\]
	and $r_n(k,\ell):=0$ for other $k,\ell\in K,\ n\in\mathbb{N}$,
	%
	%
	\begin{align*}
	v_{0}(k,\ell) & := \begin{cases}
	\nu, & \text{if }k\in\left\{ 0,1,\ldots,b-1\right\},\ \ell=k+1,\\
	\gamma\cdot k, & \text{if }k\in\left\{ 1,\ldots,b\right\},\ \ell=k-1\\
	0, & \text{otherwise for }k\neq\ell,
	\end{cases}\\
	\\
	v_{n}(k,\ell) & := \begin{cases}
	\nu, & \text{if } k\in\left\{ 0,1,\ldots,b-1\right\},\ \ell=k+1,\\
	\gamma\cdot(k-1), & \text{if }k\in\left\{ 1,\ldots,b\right\},\ \ell=k-1,\\
	0, & \text{otherwise for }k\neq\ell,
	\end{cases}\quad  n\in \mathbb{N},
	\end{align*}
	and $v_n\left(k,k\right):= -\sum_{\substack{\ell\in K,\\
			k\neq\ell} }v_n\left(k,\ell\right)$ for all $k\in K$ and $n\geq 0$.	
\end{example}	
\medskip

We are able to show by an explicit example in Corollary \ref{cor:QueueInventoryPerishable-PF} below that in the production-inventory system of Example \ref{ex:QueueInventory2} 
the stationary distribution is in general not of product form. 
The proof is  direct: (i) Insert the proposed stationary distribution into the steady-state 
(global balance) equations. (ii) Ergodicity of the production-inventory
process follows from summability of the obtained solution under 
$\lambda<\mu$. (iii) Verify that the stationary distribution is not the product of its marginal distributions. This leads to

\begin{corollary}\label{cor:QueueInventoryPerishable-PF}
	Consider the production-inventory system of Example \ref{ex:QueueInventory2} with state-independent rates
	$\lambda(n)=\lambda$ and $\mu(n)=\mu$ for all $n$ and some $\lambda,\ \mu >0$
	and base stock level $b=1$.
	If $\lambda<\mu$, the production-inventory process is ergodic, and the stationary distribution $\pi$ is given by
	\begin{align*}
	&\pi(0,0) := C^{-1} \cdot \frac{\lambda + \gamma }{\nu },\quad \\
	&\pi(n,0) := C^{-1} \cdot \left(\frac{\lambda}{\mu }\right)^{n}\cdot \frac{\lambda}{\nu },\   n>0,\qquad 
	\pi(n,1) := C^{-1} \cdot \left(\frac{\lambda}{\mu }\right)^{n}, \  n\geq 0,
	\end{align*}
	with normalization constant
	\begin{align*}
	C := \frac {\mu}{\mu - \lambda}\cdot \left(1+\frac{\lambda}{\nu}\right)+\frac{\gamma}{\nu}.
	\end{align*}
\end{corollary}
The case of general base stock level $b$ in Example \ref{ex:QueueInventory2} can be proved using case \textbf{(b)} of Example \ref{ex:random-bsp-perish}. This is the result of 
\begin{corollary}
	\label{cor:QueueInventoryPerishable-NonPF}
	Consider the production-inventory system of Example \ref{ex:QueueInventory2} with state-dependent rates
	$\lambda(n), \mu(n)\in (0,\infty)$ and general base stock level $b$.
	If the isolated production system (without inventory-replenishment system) with rates $\lambda(n), \mu(n) \in (0,\infty)$ is ergodic, then the production-inventory system (with replenishment system) is ergodic.
\end{corollary}

\begin{remark}\label{rem:AllErgodic}
	The construction of Lyapunov functions under the assumptions of Theorem \ref{thm:BS-PER-GER-pos-rec} can be done using the same procedure and conditions as above for systems which are separable. This technique is usually not needed if we have found (as in case of separability in Section \ref{sect:Separable}) a stationary measure which must be summable to obtain  a stationary distribution. This proves ergodicity.
\end{remark}

\subsection{Exponential ergodicity via Lyapunov functions}\label{sect:expo-ergodicity-Lyapunov}

The Lyapunov  function for ergodicity in Theorem \ref{thm:BS-PER-GER-pos-rec} is of ``additive separable'' structure, see Remark \ref{rem:LyaFctIndependence}. 
A path to exponential ergodicity via some similar
``additive separability'' seems to be not possible.
Our approach will be multiplicative, i.e.~the Lyapunov function developed below is of product form. 
The factors are (as the sums in Theorem \ref{thm:BS-PER-GER-pos-rec}) a term which stems from exponential ergodicity of the queueing component in isolation and a term which is responsible for sufficiently fast return of the environment to $K_W$ whenever it enters $K_B$.
Recall that we assume that $K$ is finite.
We start with a short remark on the criteria for exponential ergodicity.
\begin{remark}\fakephantomsection\label{rem:ExpErgVariant}
	\begin{enumerate}
		\item [\textbf{(i)}]  Due to \eqref{eq:ExpoErg2} the condition \eqref{eq:ExpoErg1} can be stated as \cite[Theorem 6.5 (6.8)]{anderson:91}
		\begin{equation}
		\sum_{y\in (E\setminus F) \setminus\{x\}}q(x;y){\cal M}(y)
		\leq(-q(x;x)-\sigma){\cal M}(x) - 1,~~~
		x\in E\setminus F.
		\end{equation}
		\item[\textbf{(ii)}] A necessary condition for exponential ergodicity according to Proposition \ref{prop:Anderson91} is 
		$\inf_{x\in E}(-q(x;x))>0$.	
	\end{enumerate}
\end{remark}	

Recall the definition of the  Markov processes 
$Y^{(n)}:= (Y^{(n)}(t):t\geq 0)$ on $K$ with generator $V_n=(v_n(k,\ell):k,\ell\in K)$ (for $n\in \mathbb{N}_0$)
before Lemma \ref{lem:BS-PER-GEN-Lemma2}, and that
$T_n$ denotes the $[0,\infty]$-valued first-entrance time of $Y^{(n)}$ into $K_W$.
The proof of the following lemma is inspired by \cite[Chapter 6, Lemma 1.5]{anderson:91}.

\begin{lemma}\label{lem:Thetas}
	Define for $k\in K_B$ and $\sigma_n\in (0,\min_{\ell\in K_B}(-  v_n(\ell,\ell)))$
	\begin{equation}\label{eq:Theta1}
	\theta_n(k) := \int_{0}^{\infty} P(T_n>t|Y^{(n)}(0)=k)
	e^{\sigma_n t} dt.
	\end{equation}
	For $k\in K_W$ define $\theta_n(k):=0.$ Then it holds
	\begin{equation}\label{eq:Theta2}
	\sum_{\ell\in K\setminus\{k\}} v_n(k,\ell)[\theta_n(\ell)-\theta_n(k)] =
	-\sigma_n\cdot \theta_n(k) - 1,\quad k\in K_B. 
	\end{equation}
	For all $k\in K_B$ and all $\sigma_n\in (0,\min_{\ell\in K_B}(- v_n(\ell,\ell)))$ it holds
	$0< \theta_n(k)  < \infty.$
\end{lemma}	
Note that $\min_{\ell\in K_B}(-  v_n(\ell,\ell)) >0$ because $K_B$ has no absorbing states and \eqref{eq:Theta2} can be written as
\begin{equation}\label{eq:Theta3}
\sum_{\ell\in K_B\setminus\{k\}} v_n(k,\ell)\theta_n(\ell)=
(-v_n(k,k)-\sigma_n)\cdot \theta_n(k) - 1,\quad   k\in K_B. 
\end{equation}
\proof
We  abbreviate $v_n(k):=-v_n(k,k)$ for $k\in K$.
Then for  $k\in K_B$ it holds
\begin{flalign*}
&\phantomeq \theta_n(k) \\
&= \int_{0}^{\infty} P(T_n>t|Y^{(n)}(0)=k)e^{\sigma_n t}dt\\
&\stackrel{(*)}{=}
\int_{0}^{\infty}  \left( \sum_{\ell\in K_{B}}P(T_n>t, Y^{(n)}(t) =\ell|Y^{(n)}(0)=k)\right) e^{\sigma_n t}dt\\
&=
\sum_{\ell\in K_{B}}\int_{0}^{\infty}  \left(P(T_n>t, Y^{(n)}(t) =\ell|Y^{(n)}(0)=k)\right) e^{\sigma_n t} dt \\
&= \sum_{\ell\in K_{B}}\int_{0}^{\infty} \Big( P(T_n>t, Y^{(n)}(t) =\ell, ~\text{no jump until}~ t|Y^{(n)}(0)=k)\\
&\phantomeq+P(T_n>t, Y^{(n)}(t) =\ell, ~\text{jump before}~ t|Y^{(n)}(0)=k) \Big) e^{\sigma_n t}dt\\
&= \sum_{\ell\in K_{B}}\int_{0}^{\infty}  1_{\{\ell=k\}} e^{-v_n(k)t}  e^{\sigma_n t}dt\\
&\phantomeq+ \sum_{\ell\in K_{B}}\int_{0}^{\infty} \Big(
\int_{0}^{t}  v_n(k) e^{-v_n(k)s} \sum_{\substack{h\in K_{B}\\h\neq k}}
\frac{v_n(k,h)}{v_n(k)}
P(T_n>t-s, Y^{(n)}(t-s) =\ell|Y^{(n)}(0)=h)ds\Big)e^{\sigma_n t}dt\\
&=\int_{0}^{\infty}  e^{-v_n(k)t}  e^{\sigma_n t}dt\\
&\phantomeq+ \sum_{\ell\in K_{B}}\int_{0}^{\infty} \Big(
\int_{0}^{t}   e^{-(v_n(k)-\sigma_n)s} \sum_{\substack{h\in K_{B}\\h\neq k}} v_n(k,h)
P(T_n>t-s, Y^{(n)}(t-s) =\ell|Y^{(n)}(0)=h)\Big)e^{\sigma_n(t-s)}dsdt\\
&=\int_{0}^{\infty}  e^{-v_n(k)t}  e^{\sigma_n t}dt\\
&\phantomeq+\int_{0}^{\infty}  e^{-(v_n(k)-\sigma_n)s} \sum_{\substack{h\in K_{B}\\h\neq k}} v_n(k,h)  \sum_{\ell\in K_{B}}\int_{s}^{\infty}
e^{\sigma_n(t-s)} 
P(T_n>t-s, Y^{(n)}(t-s) =\ell|Y^{(n)}(0)=h)dtds\\
&=\int_{0}^{\infty}  e^{-v_n(k)t}  e^{\sigma_n t} dt\\
&\phantomeq+\int_{0}^{\infty}  e^{-(v_n(k)-\sigma_n)s} ds \sum_{\substack{h\in K_{B}\\h\neq k}} v_n(k,h)  \sum_{\ell\in K_{B}}\int_{0}^{\infty} 
e^{\sigma_n(t)}
P(T_n>t, Y^{(n)}(t) =\ell|Y^{(n)}(0)=h)dt\\
&=\int_{0}^{\infty}  e^{-v_n(k)t}  e^{\sigma_n t}dt
+\int_{0}^{\infty}  e^{-(v_n(k)-\sigma_n)s} ds~ \sum_{\substack{h\in K_{B}\\h\neq k}} v_n(k,h) 
\underbrace{\int_{0}^{\infty}  e^{\sigma_n(t)}
	P(T_n>t|Y^{(n)}(0)=h)}_{=\theta_n(h)} dt\\
&=\frac{1}{v_n(k)-\sigma_n} + 
\frac{1}{v_n(k)-\sigma_n}\sum_{\substack{h\in K_{B}\\h\neq k}} v_n(k,h)\theta_n(h).
\end{flalign*}	
Here $(*)$ follows from
$P(T_n>t, Y^{(n)}(t) =\ell|Y^{(n)}(0)=k)=0$ 
for $\ell\in K_W$.
Rearranging terms yields the proposed formula. Positivity of the 
$\theta_n(k)$ follows from $\int_{0}^{\infty} P(T_n(k)>t|Y^{(n)}(0)=k)
e^{\sigma_n t}dt > \int_{0}^{\infty} P(T_n(k)>t|Y^{(n)}(0)=k)dt =\tau_{n}(k)>0.$
Finiteness of the $\theta_n(k)$ follows from the observation (integration by parts)
\begin{equation*}
\theta_n(k)=\frac{\mathbb{E}[e^{\sigma_n T_n}|Y^{(n)}(0)=k]-1}{\sigma_n}.
\end{equation*}
$\theta_n(k)<\infty$ holds because we can extend $Y^{(n)}$ to a Markov process $\hat{Y}^{(n)}$ with  state space $K_B \cup \{a\}$ as follows: On $K_B$ both processes move identically. Whenever $Y^{(n)}$ leaves $K_B$ the extended process enters $a$, dwells there for an exponential time and then enters all states in $K_B$ which 
the environment  can reach from $K_W$ with equal probability. Then  $\hat{Y}^{(n)}$ moves again according to the law of ${Y}^{(n)}$, and so on.
$\hat{Y}^{(n)}$ (with finite state space) is exponentially ergodic and the return time $\rho_n$ to $a$ when starting in $a$ dominates $T_n$. According to \cite[Theorem 6.5 (b)]{anderson:91} the moment generating function of $\rho_n$ exists and therefore that of $T_n$.

\endproof
A direct consequence of Proposition \ref{prop:Anderson91} is the following criterion for birth-death processes.
\begin{lemma}
	\label{lem:ExpoErgodicBD}
	Consider an $M/M/1/\infty$-queue with queue-length-dependent
	arrival rates $\lambda(n)\in (0,\infty)$ and service rates $\mu(n)\in (0,\infty)$ which is exponentially ergodic.
	Then there exists a function
	$\widetilde{\mathcal{M}}:\mathbb{N}_{0}\rightarrow\mathbb{R}_{0}^{+}$ (Lyapunov function) for the queue length process  with finite exception set $\widetilde{F}\neq \emptyset$ and constant $\widetilde{\omega}\in (0,\lambda(0)\wedge \inf_{n\in \mathbb{N}}(\lambda(n)+\mu(n)))$ which satisfies the 
	criterion from Proposition \ref{prop:Anderson91}, {in particular} the following inequalities are satisfied with
	$\widetilde{\mathcal{M}}(n)=0$ for $n\in \widetilde{F}$.
	\begin{align}
	& \lambda(0)\cdot\left(\widetilde{\mathcal{M}}(1)-
	\widetilde{\mathcal{M}}(0)\right) \leq -1 - \widetilde{\omega}\cdot \widetilde{\mathcal{M}}(0),
	&  & \text{if }0\notin\widetilde{F},\label{eq:exp-mm1-eq1}\\
	& \lambda(n)\cdot\left(\widetilde{\mathcal{M} }(n+1)-\widetilde{\mathcal{M}}(n)\right)+\mu(n)\cdot\left(\widetilde{\mathcal{M}}(n-1)-\widetilde{\mathcal{M}}(n)\right)\leq -1 -\widetilde{\omega}\cdot \widetilde{\mathcal{M}}(n), &  & \text{if }n\notin\widetilde{F},\ n>0.\label{eq:exp-mm1-eq2}
	\end{align}	              
\end{lemma}	
\begin{corollary}\label{cor:ExpoErgodicBD} 
	Consider an exponentially ergodic $M/M/1/\infty$-queue with queue-length-dependent arrival rates $\lambda(n)\in (0,\infty)$ and service rates $\mu(n)\in (0,\infty)$ as in Lemma \ref{lem:ExpoErgodicBD}.
	Denote the associated queue length process by
	$\widetilde{X}:=(\widetilde{X}(t):t\geq 0)$.
	For any $F\subset\mathbb{N}_0$
	denote by $\widetilde{S}_F$ the first-entrance time of 
	$\widetilde{X}$ into  $F,$ and by 
	$\widetilde{S}_{mF}$ a random variable which is distributed according to 
	$P(\widetilde{S}_{F} \leq \cdot|\widetilde{X}(0)=m)$ for
	$m\in\mathbb{N}_0.$ (So $\widetilde{S}_{mF}$ has 1-point distribution in $0$ if $m\in F$.) 
	If $F=\{n\}$, we abbreviate 
	$\widetilde{S}_F$ by $\widetilde{S}_n$ and 
	$\widetilde{S}_{mF}$ by $\widetilde{S}_{mn}$.
	The following properties 
	of $\widetilde{X}$ hold.
	\begin{enumerate}
		\item[\textbf{(a)}] For any finite $\widetilde{F}\neq \emptyset$ and suitable  $\widetilde{\omega}=\widetilde{\omega}(\widetilde{F})\in (0,\lambda(0)\wedge \inf_{n\in \mathbb{N}}(\lambda(n)+\mu(n)))$
		the conditions \eqref{eq:ExpoErg2}--\eqref{eq:ExpoErg1} are satisfied (\eqref{eq:ExpoErg1} with equality) by
		\begin{align}
		&\widetilde{\mathcal{M}}(n) =0, \quad n\in \widetilde{F}
		\label{eq:M-BD-1}\\
		&\widetilde{\mathcal{M}}(n) = \int_{0}^{\infty} P(\widetilde{S}_{n\widetilde{F}}>t) e^{\widetilde{\omega} t} \,dt,
		~~~ n\not\in \widetilde{F}. \label{eq:M-BD-2}
		\end{align}	
		If $\widetilde{\mathcal{M}}_0(n), n \in \mathbb{N}_0,$ is (for the same $\widetilde{F}$
		and  $\widetilde{\omega}$) another solution of 
		\eqref{eq:ExpoErg2}--\eqref{eq:ExpoErg1}, then it holds $\widetilde{\mathcal{M}}_0(n)\geq \widetilde{\mathcal{M}}(n)$ for all $n \in \mathbb{N}_0$, i.e.~$\widetilde{\mathcal{M}}$ is the minimal solution of 
		\eqref{eq:ExpoErg2}--\eqref{eq:ExpoErg1}.
		\\	
		(A proof is given in  \cite[Theorem 6.5 and Lemma 1.5 in Chapter 6]{anderson:91}.)
		\item[\textbf{(b)}] For any $m\geq 0$ and $n\geq m$ it holds
		$\widetilde{S}_{nm} \sim \widetilde{S}_{n n-1}+
		\widetilde{S}_{n-1 n-2}+\dots+\widetilde{S}_{m+1 m}$ and the random variables on the right-hand side are independent. So the sequence $(\widetilde{S}_{nm}:n\geq m)$ is stochastically increasing in $n$ for any $m\geq 0$.
		Consequently, if $\widetilde{F}=\{0,1,\dots,m\}$, the sequence $(\widetilde{\mathcal{M}}(n):n> m)$ is strictly increasing in $n\geq m$.
		\item[\textbf{(c)}] $\widetilde{S}_{10}$ is distributed according to the busy period of the $M/M/1/\infty$-queue. 
		\item[\textbf{(d)}] For the queueing system with state-independent rates  $\lambda(n)=\lambda, \mu(n)=\mu$ with $\lambda<\mu$
		it holds, see \cite[p.~105]{asmussen:03},
		\begin{equation}
		\mathbb{E}[\widetilde{S}_{10}] = \frac{1}{\mu - \lambda}.
		\end{equation}
		In this system the random variables $ \widetilde{S}_{n n-1},
		\widetilde{S}_{n-1 n-2},\dots,\widetilde{S}_{10}$ in \textbf{(b)} are independent and identically distributed.
	\end{enumerate}
\end{corollary}
We now combine Lemma \ref{lem:Thetas} and Lemma \ref{lem:ExpoErgodicBD}  to characterize the behaviour of $Z$.

\begin{theorem}\label{thm:ExpErgodic}
	Consider the ergodic queueing-environment process $Z$ with finite environment set $K$. Assume that the $M/M/1/\infty$-queue
	 with queue-length-dependent arrival rates $\lambda(n)\in (0,\infty)$ and service rates $\mu(n)\in (0,\infty)$ in isolation is exponentially ergodic and that
	\begin{equation}\label{eq:MtildeNull}
	\widetilde{\mathcal{M}_0}:\mathbb{N}_{0}\rightarrow [0,\infty)
	\end{equation}
	is a  Lyapunov function for this process in the sense of Proposition \ref{prop:Anderson91} (for ergodic Markov processes) with finite exception set $\widetilde{F}$ and constant ${\widetilde{\omega}}\in (0,\lambda(0)\wedge \inf_{n\in \mathbb{N}}(\lambda(n)+\mu(n)))$ according to Lemma \ref{lem:ExpoErgodicBD}.
	For all $n\in \mathbb{N}_0\setminus\widetilde{F}$ define
	with  $\sigma_n\in (0,\inf_{\ell\in K_B}(-  v_n(\ell,\ell)))$
	as in \eqref{eq:Theta1} (recall that $\theta_n(k):=0$
	for all $k\in K_W$):
	\[
	\theta_n(k) := \int_{0}^{\infty} P(T_n>t|Y^{(n)}(0)=k)
	e^{\sigma_n t}dt, \quad k\in K_B,
	\]
	and according to Corollary \ref{cor:ExpoErgodicBD}\textbf{(a)}
	\begin{equation}\label{eq:Mtilde}
	\widetilde{\mathcal{M}}:\mathbb{N}_{0}\rightarrow [0,\infty)
	\end{equation}
	with $\widetilde{\mathcal{M}}(n)=0$ for $n\in \widetilde{F}$ and
	\begin{equation}
	\widetilde{\mathcal{M}}(n) = \int_{0}^{\infty} P(\widetilde{S}_{n\widetilde{F}}>t) e^{\widetilde{\omega} t} \,dt,
	~~~ n\not\in \widetilde{F}. \label{eq:M-BD-2NEU}
	\end{equation}	
	Assume that it holds:
	\begin{enumerate}
		\item[\textbf{(i)}] 
		$\sigma:=\inf_{n\in \mathbb{N}_0\setminus\widetilde{F}}\inf_{\ell\in K_B}(-  v_n(\ell,\ell))>0$, and
		\item[\textbf{(ii)}]
		there exists $m_0\in \mathbb{N}_0$ with $m_0\geq \sup \widetilde{F}$  and a  sequence of positive  numbers 
		$(d_n:n>m_0)$ and constants $\alpha, \beta\in (0,1)$ 
		such that (with $d_{m_0} := 0$)
		it holds
		\begin{align}
		& \sum_{\ell\in K_B} \mu(n)\cdot r_{n}(k,\ell)\cdot 
		\Big(\theta_{n-1}(\ell)\cdot d_{n-1} -1\Big)\cdot \frac{\widetilde{\mathcal{M}}(n-1)}{\widetilde{\mathcal{M}}(n)}\nonumber\\
		&\qquad\qquad\qquad\qquad + \sum_{\ell\in K_B} v_{n}(k,\ell)\cdot
		\Big(\theta_{n}(\ell)\cdot d_{n} -1\Big)
		\leq  \alpha\cdot \widetilde{\omega},~~ k\in K_W,~~ n> m_0,
		\label{eq:ExpECond-1}\\
		&\sum_{\ell\in K_W} v_{n}(k,\ell) +\frac{1}{\widetilde{\mathcal{M}}(n)} - d_n
		\leq \beta\cdot \sigma_n\cdot\theta_{n}(k)\cdot d_n,~~	
		k\in K_B,~~ n> m_0.\label{eq:ExpECond-2}
		\end{align}
	\end{enumerate}
	Then
	\begin{equation}\label{eq:LFctExpErg}
	{\mathcal{M}}\colon E \to [0,\infty), 
	~~(n,k)\mapsto \begin{cases}
	0                           & n\leq  m_0,\\
	\widetilde{\mathcal{M}}(n), &k\in K_W, n>  m_0,\\
	\widetilde{\mathcal{M}}(n)\cdot \theta_{n}(k)\cdot d_n,
	&k\in K_B,  n>  m_0,
	\end{cases}
	\end{equation}	
	is a Lyapunov function for exponential ergodicity, as defined in Proposition \ref{prop:Anderson91}, with exception set 
	{$F:= \{0,1,\dots,m_0\}\times K$}  and constant 
	${\omega}:= \left((1-\alpha)\widetilde{\omega}\wedge
	(1-\beta)\sigma\right)$ and $Z$ is exponentially ergodic. 
\end{theorem}
Before proving the theorem a short remark is in order. Introducing the lower boundary $m_0$ for the relevant queue lengths in the criterion enables us to neglect in applications possible extreme behaviour of the environment (with respect to conditions \eqref{eq:ExpECond-1} and \eqref{eq:ExpECond-2}) for a finite set of queue lengths. In Example \ref{ex:QueueInventory4}
below setting an (artificial) boundary $m_0$ supports to prove exponential ergodicity.
%
\proof
We apply Proposition \ref{prop:Anderson91} and show that $\mathcal{M}$ is a Lyapunov function for $Z$ with the proposed finite exception set {$F:= \{0,1,\dots,m_0\}\times K$}
and constant $\omega$.

We first note that $\widetilde{\mathcal{M}}$ is a Lyapunov function for the isolated $M/M/1/\infty$-queue with exception set $\widetilde{F}$ and constant $\widetilde{\omega}$
according to Corollary \ref{cor:ExpoErgodicBD}\textbf{(a)}
and therefore satisfies the system 	\eqref{eq:ExpoErg2}--\eqref{eq:ExpoErg1} of Proposition \ref{prop:Anderson91} (with equality).

To check
$
\sum_{(m,\ell)\in E} q((n,k);(m,\ell)) \mathcal{M}(m,\ell)<\infty 
$
for $(n,k)\in F$ 
is direct 
because  jumps of $X$ are of distance $1$, and $|K|<\infty$.\\
$\blacktriangleright$ For $k\in K_{W}$ and  
$n > m_0+1$ it holds
\begin{align}
& \phantomeq
\left(\mathbf{Q\cdot\mathcal{M}}\right)\left(n,k\right)\nonumber \\
&=\lambda(n)\cdot\left(\mathcal{M}(n+1,k)-\mathcal{M}(n,k)\right) 
+\sum_{\ell\in K}\mu(n)\cdot r_{n}(k,\ell)\cdot\left(\mathcal{M}(n-1,\ell)
-\mathcal{M}(n,k)\right)\nonumber\\
&
\phantomeq +\sum_{\ell\in K\setminus\{k\}}v_{n}(k,\ell)\cdot\left(\mathcal{M}(n,\ell)-\mathcal{M}(n,k)\right)\nonumber\\
&=\lambda(n)\cdot\left(\widetilde{\mathcal{M}}(n+1)-
\widetilde{\mathcal{M}}(n)\right)
+\sum_{\ell\in K_W}\mu(n)\cdot r_{n}(k,\ell)\cdot\left(\widetilde{\mathcal{M}}(n-1)
-\widetilde{\mathcal{M}}(n)\right)\nonumber\\
& \phantomeq
+\sum_{\ell\in K_B}\mu(n)\cdot r_{n}(k,\ell)\cdot\left(\widetilde{\mathcal{M}}(n-1)\cdot\theta_{n-1}(\ell)\cdot d_{n-1}
-\widetilde{\mathcal{M}}(n)\right)\nonumber\\
&\phantomeq +\sum_{\ell\in K_B}v_{n}(k,\ell)\cdot\left(\widetilde{\mathcal{M}}(n)\cdot\theta_n(\ell)\cdot d_n-\widetilde{\mathcal{M}}(n)\right)\nonumber\\
&=\left\langle\lambda(n)\cdot\left(\widetilde{\mathcal{M}}(n+1)- \widetilde{\mathcal{M}}(n)\right)
+\mu(n) \cdot\left(\widetilde{\mathcal{M}}(n-1)
-\widetilde{\mathcal{M}}(n)\right)\right\rangle\nonumber\\
&\phantomeq-\sum_{\ell\in K_B}\mu(n)\cdot r_{n}(k,\ell)\cdot\left(\widetilde{\mathcal{M}}(n-1)
-\widetilde{\mathcal{M}}(n)\right)\nonumber\\
&\phantomeq+\sum_{\ell\in K_B}\mu(n)\cdot r_{n}(k,\ell)\cdot\left(\widetilde{\mathcal{M}}(n-1)\cdot\theta_{n-1}(\ell)\cdot d_{n-1}
-\widetilde{\mathcal{M}}(n)\right)\nonumber\\
&\phantomeq +\sum_{\ell\in K_B}v_{n}(k,\ell)\cdot\left(\widetilde{\mathcal{M}}(n)\cdot\theta_n(\ell)\cdot d_n-\widetilde{\mathcal{M}}(n)\right)\nonumber\\
&\stackrel{(*)}{=}\left\langle
-1-\widetilde{\omega}\cdot\widetilde{\mathcal{M}}(n)
\right\rangle
+ \widetilde{\mathcal{M}}(n)
\left[\mu(n)\cdot \sum_{\ell\in K_B} r_{n}(k,\ell)\cdot \frac{\widetilde{\mathcal{M}}(n-1)}{\widetilde{\mathcal{M}}(n)}
\Big(\theta_{n-1}(\ell)\cdot d_{n-1} -1\Big)\right.\nonumber\\
&\phantomeq \left. + \sum_{\ell\in K_B} v_{n}(k,\ell)\cdot 
\Big(\theta_{n}(\ell)\cdot d_{n} -1\Big)\right]\label{eq:ForEx-QueueInventory3-1}\\
&\stackrel{\eqref{eq:ExpECond-1}}{\leq}\left\langle
-1-\widetilde{\omega}\cdot\widetilde{\mathcal{M}}(n)
\right\rangle
+\widetilde{\mathcal{M}}(n)\cdot\alpha\cdot \widetilde{\omega}
=\left\langle
-1-(1-\alpha)\widetilde{\omega}\cdot\widetilde{\mathcal{M}}(n)
\right\rangle \nonumber \\
&\leq \left\langle
-1-{\omega}\cdot\widetilde{\mathcal{M}}(n)
\right\rangle
= 
-1-{\omega}\cdot\mathcal{M}(n,k)
. \nonumber
\end{align}
Here $(*)$ follows from Lemma \ref{lem:ExpoErgodicBD} and Corollary \ref{cor:ExpoErgodicBD}\textbf{(a)}.\\ 	
$\blacktriangleright$ The case $k\in K_W$ and $n=m_0+1$  leads to  similar computations with some slight simplifications. \\
$\blacktriangleright$ For $k\in K_{B}$ and 
$n> m_0$ 
it holds
\begin{align}
& \phantomeq \left(\mathbf{Q\cdot\mathcal{M}}\right)\left(n,k\right)\nonumber \\
&=\sum_{\ell\in K\setminus\{k\}}v_{n}(k,\ell)\cdot\left(\mathcal{M}(n,\ell)-\mathcal{M}(n,k)\right)\nonumber\\
& =\sum_{\ell\in K_B\setminus\{k\}}
v_{n}(k,\ell)\cdot
\left(
\widetilde{\mathcal{M}}(n)\cdot\theta_n(\ell)\cdot d_n  
- \widetilde{\mathcal{M}}(n)\cdot\theta_n(k)\cdot d_n 
\right)\nonumber\\
&\phantomeq +
\sum_{\ell\in K_W}
v_{n}(k,\ell)\cdot
\left(
\widetilde{\mathcal{M}}(n) 
- \widetilde{\mathcal{M}}(n)\cdot\theta_n(k)\cdot d_n 
\right)\nonumber\\
& =\widetilde{\mathcal{M}}(n)\cdot d_n
\sum_{\ell\in K_B\setminus\{k\}}
v_{n}(k,\ell) \left(\theta_n(\ell)  - \theta_n(k)
\right)\nonumber
+
\widetilde{\mathcal{M}}(n) \sum_{\ell\in K_W}
v_{n}(k,\ell)\cdot
\left(1 - \theta_n(k)\cdot d_n 
\right)\nonumber\\
& =\widetilde{\mathcal{M}}(n)\cdot d_n
\left\langle\sum_{\ell\in K\setminus\{k\}}
v_{n}(k,\ell) \left(\theta_n(\ell)  - \theta_n(k)
\right)\right\rangle\nonumber
-\widetilde{\mathcal{M}}(n)\cdot d_n
\sum_{\ell\in K_W}
v_{n}(k,\ell) \left(0  - \theta_n(k)
\right)\nonumber\\
&\phantomeq+
\widetilde{\mathcal{M}}(n) \sum_{\ell\in K_W}
v_{n}(k,\ell)\cdot
\left(1 - \theta_n(k)\cdot d_n 
\right)\nonumber\\
& \stackrel{\eqref{eq:Theta2}}{=}\widetilde{\mathcal{M}}(n)\cdot d_n
\Big\langle
-1 -\sigma_n\cdot \theta_{n}(k)
\Big\rangle +
\widetilde{\mathcal{M}}(n) \sum_{\ell\in K_W}
v_{n}(k,\ell)\nonumber\\
& = -\widetilde{\mathcal{M}}(n)\cdot d_n
+
\widetilde{\mathcal{M}}(n) \sum_{\ell\in K_W}
v_{n}(k,\ell)
-\sigma_n\cdot 
\widetilde{\mathcal{M}}(n)\cdot \theta_{n}(k)\cdot d_n
\nonumber\\
&=\left\langle -1 -\sigma_n \cdot \widetilde{\mathcal{M}}(n)\cdot \theta_{n}(k)\cdot d_n
\right\rangle + \widetilde{\mathcal{M}}(n) \sum_{\ell\in K_W}
v_{n}(k,\ell)
-\left(\widetilde{\mathcal{M}}(n)\cdot d_n -1\right)
\label{eq:ForEx-QueueInventory3-2}
\\
&\stackrel{\eqref{eq:ExpECond-2}}{\leq}
\left\langle -1 -\sigma_n \cdot \widetilde{\mathcal{M}}(n)\cdot \theta_{n}(k)\cdot d_n
\right\rangle 
+ \widetilde{\mathcal{M}}(n) \theta_{n}(k) d_n \sigma_n \beta
\nonumber\\
&\stackrel{\eqref{eq:LFctExpErg}}{=}
\left\langle -1 -(1-\beta)\sigma_n \cdot {\mathcal{M}}(n,k)
\right\rangle 
\leq -1 -\omega \cdot {\mathcal{M}}(n,k).
\nonumber \qquad\qquad\qquad\qquad\qquad\qquad
\end{align}
\endproof
\begin{remark}\label{rem:LyaFunctExpErg}
	The construction of the Lyapunov function for exponential ergodicity with
	\[{\mathcal{M}}(n,k)=
	\widetilde{\mathcal{M}}(n)\cdot \exp\left[ 1_{\{k\in K_B\}}     \log \left( \theta_{n}(k)\cdot d_n \right)\right],
	\quad k\in K,\  n\geq  m_0,
	\]
	shows that the terms
	\[
	\exp\left[ 1_{\{k\in K_B\}}     \log \left( \theta_{n}(k)\cdot d_n \right)\right]
	\]
	(relevant for the environment) are multiplicatively separated from the terms $\widetilde{\mathcal{M}}(n)$ (relevant for the queue), i.e.~we obtained a construction in product form.
	
	Different from the situation with standard ergodicity in Theorem \ref{thm:BS-PER-GER-pos-rec} in the  conditions \eqref{eq:ExpECond-1} and \eqref{eq:ExpECond-2} 
	the terms for the queue and the environment are intertwined.
	This indicates that a stronger coupling of the dynamics of queue and environment is needed to obtain the  faster convergence of the system to stationarity.
	
	The  product form criterion in Theorem \ref{thm:ExpErgodic} is in line with the results in 
	\cite{spieksma;tweedie:94}. 
	For Markov chains (in discrete time) with Lyapunov function $V$ (similar to Proposition \ref{prop:KellyYudovina14})
	the authors develop conditions which ensure that $V^* := e^{\delta \cdot V(\cdot)}$ is for some $\delta>0$ a Lyapunov function which detects exponential ergodicity of the Markov chain (similar to Proposition \ref{prop:Anderson91}).		
	The technique developed there is not applicable in our problem setting, but we note that the procedure would turn an additive $V$ on a $2$-dimensional state space 
	into a multiplicative $V^*$.
	
	Nevertheless, there is no such direct progress from ergodicity to exponential ergodicity in our setting because the $\theta_{n}$ are not the exponentials of the $\tau_{n}$. 
\end{remark}

\begin{example}\label{ex:QueueInventory4}
	We consider the production-inventory system with perishable items from Example \ref{ex:QueueInventory2}	 with state-independent arrival and service rates $\lambda,\mu$. According to Corollary \ref{cor:QueueInventoryPerishable-NonPF} the state process $Z$ is ergodic if $\lambda<\mu$. Recall that $K_W=\{1,\dots,b\}$ and  $K_B=\{0\}$ {and $E=\mathbb{N}_0\times \{0,1,\ldots, b\}$}.\\		
	For $k=0$ (stock-out) we obtain for the first-entrance times $T_n$, $n\geq 1$, into $K_W$ (which occurs if a replenishment arrives at the empty inventory) with $\sigma_n<\nu$
	\begin{equation}\label{eq:theta-Perish-3}	
	\theta_{n}(0) = \int_{0}^{\infty} P(T_n>t|Y^{n}(0)=0)
	e^{\sigma_n  t} dt=
	\int_{0}^{\infty} e^{-\nu t}e^{\sigma_n  t}dt =
	\frac{1}{\nu - \sigma_n}.
	\end{equation}
	For $k=1,\dots,b$ it holds $\theta_{n}(0) = 0$.
	Note that $\sigma_n$ and therefore $\theta_{n}(0)$ can be taken independent of $n\geq 1$. We set
	\[
	0<\sigma_n=:\sigma<\nu ~~~\text{and}~~~ \theta(0):=\theta_{n}(0),~~~n\geq 1.
	\]
	Because the queue length process of the $M/M/1/\infty$-queue with 
	$\lambda<\mu$ in isolation is exponentially ergodic, a
	Lyapunov function $\widetilde{\mathcal{M}}$ according to
	Corollary \ref{cor:ExpoErgodicBD}\textbf{(a)} exists with $\widetilde{F}=\{0\}$ and suitable $0<\widetilde{\omega}<\lambda$ as given in \eqref{eq:M-BD-1} and \eqref{eq:M-BD-2}.		 
	We shall prove according to Proposition \ref{prop:Anderson91}  the existence of a
	Lyapunov function ${\mathcal{M}}$ for $Z$. For this
	we shall show that suitable values $d_n$ and $m_0$ and $\alpha,\beta\in(0,1)$ exist to apply Theorem \ref{thm:ExpErgodic}.
	So $Z$ is exponentially ergodic.
	
	For $k=0$ we have ${\sum_{\ell\in K_W} v_{n}(k,\ell)=}v_n(0,1)=\nu$ for all $n\geq m_0+1$  and consequently, for validity of 
	\eqref{eq:ExpECond-2} we have  to satisfy the condition 
	\begin{equation}\label{eq:ExpECond-2-Ex1}
	\nu+\frac{1}{\widetilde{\mathcal{M}}(n)} - d_n
	\leq \beta\cdot\sigma\cdot\theta(0) \cdot d_n.
	\end{equation}
	Inserting $\theta(0)={1}/{(\nu -\sigma)}$, this is equivalent to 
	\begin{equation}\label{eq:ExpECond-2-Ex2}
	\left(\nu + \frac{1}{\widetilde{\mathcal{M}}(n)}\right)
	\cdot\left(\frac{\nu - \sigma}{\nu - \sigma(1-\beta)}\right)\leq d_n.	
	\end{equation}
	For $k=1$ we have   $r_n(1,0)=1$ and $v_n(1,0)=0$ for all $n\geq m_0+1$. Consequently, for validity of \eqref{eq:ExpECond-1}  we have to satisfy 
	the condition 
	\begin{equation}\label{eq:ExpECond-1-Ex1-prev}
	\mu \cdot\big(\theta(0)\cdot d_{n-1} - 1\big) \cdot\frac{\widetilde{\mathcal{M}}(n-1)}{\widetilde{\mathcal{M}}(n)}	 \leq
	\alpha\cdot \widetilde{\omega}.
	\end{equation}
	Because 
	\[\frac{\widetilde{\mathcal{M}}(n-1)}{\widetilde{\mathcal{M}}(n)}\leq 1,\]
	it is sufficient to have 
	\begin{equation}\label{eq:ExpECond-1-Ex1}
	\mu \cdot \big(\theta(0)\cdot d_{n-1} - 1\big) \leq
	\alpha\cdot \widetilde{\omega}.
	\end{equation}
	Inserting then $\theta(0)={1}/{(\nu -\sigma)}$, the condition 
	\eqref{eq:ExpECond-1-Ex1} is equivalent to 
	\begin{equation}\label{eq:ExpECond-1-Ex2}
	d_{n-1} \leq \left(
	\frac{\alpha \cdot \widetilde{\omega} + \mu}{\mu}\right)
	\cdot({\nu - \sigma}).		
	\end{equation}
	To show exponential ergodicity we have to find parameters
	$d_{n-1}, d_n, \alpha, \beta$ such that with
	$\nu-\sigma\in(0,\nu)$ the inequalities 
	\eqref{eq:ExpECond-1-Ex2} and \eqref{eq:ExpECond-2-Ex2} are jointly fulfilled.
	
	Tentatively, we set $d_n:=d$ for all $n\geq m_0+1$ for some suitable $d\geq 0$ to be chosen below
	and $d_{m_0}:=0$. We have to find suitable parameters for the inequalities
	\begin{equation}\label{eq:ExpECond-1+2-Ex1}
	\left(\nu + \frac{1}{\widetilde{\mathcal{M}}(n)}\right)
	\cdot\left(\frac{\nu - \sigma}{\nu - \sigma\cdot(1-\beta)}\right)
	\leq d \leq
	\left(
	\frac{\alpha \cdot\widetilde{\omega} + \mu}{\mu}\right)
	\cdot({\nu - \sigma})
	\end{equation}
	to hold concurrently. We rewrite this as
	\begin{equation}\label{eq:ExpECond-1+2-Ex2}
	\frac{\nu}{\nu - \sigma\cdot(1-\beta)} +	
	\frac{1}{\widetilde{\mathcal{M}}(n)\cdot(\nu - \sigma\cdot(1-\beta))}
	\leq \frac{d}{{\nu - \sigma}}\leq
	\frac{\alpha \cdot\widetilde{\omega} + \mu}{\mu}.
	\end{equation}
	Because we can take $\alpha$ arbitrarily close to $1$, the right side can be selected arbitrarily close to 
	$1+\frac{\widetilde{\omega}}{\mu} > 1$.		
	Because we can take $\beta$ arbitrarily close to $1$, the first term on the left side (which is greater than $1$) can be selected  arbitrarily close to $1$.
	The second term on the right side is then approximately
	$ {1}/({\widetilde{\mathcal{M}}(n)\cdot\nu})$.	
	From Corollary \ref{cor:ExpoErgodicBD}\textbf{(b)} and \textbf{(d)}  with $\widetilde{\mathcal{M}}(n)\geq n\cdot 1/(\mu-\lambda)$
	the second term can be selected arbitrarily close to $0$. 
	This selection of suitable values of $n$ leads to determining the explicit value of $m_0$, which was up to now free for our disposal.	
	Having fixed values 
	{$\alpha$, $\beta$,} $m_0$ such that the right-hand side is strictly greater than the left-hand side, we can fix $d$ such that 
	$\frac{d}{{\nu - \sigma}}$ lies strictly in between these bounds.
	
	Summarizing, this guarantees the existence of $m_0$ to define $\widetilde{F}$, $\alpha, \beta,$ and $d$
	such that the conditions \eqref{eq:ExpECond-1} and
	\eqref{eq:ExpECond-2} are satisfied in this example.
	Furthermore, \textbf{(i)} from \prettyref{thm:ExpErgodic} is satisfied because $\sigma:=\inf_{n\in \mathbb{N}_0} -v_n(0,0)=\nu >0$. 
\end{example}	

\medskip

\begin{remark}
	In Remark \ref{rem:LyaFunctExpErg} we pointed out that there is a strong coupling necessary between the dynamics of the queue and the environment, expressed in \eqref{eq:ExpECond-1} and \eqref{eq:ExpECond-2}.
	An inspection of the procedure to prove exponential ergodicity in Example \ref{ex:QueueInventory4} reveals that the following  {stronger} conditions 
	would imply \eqref{eq:ExpECond-1} and \eqref{eq:ExpECond-2}: 
	There exist $m_0\in \mathbb{N}_0$ and a  sequence of non-negative numbers 
	$(d_n:n\geq m_0)$ and constants $\alpha, \beta\in (0,1)$ and  $ 1/\widetilde{\mathcal{M}}(m_0)\leq \delta <\infty$ such that 
	\begin{align*}
	& \sum_{\ell\in K_B} \mu(n)\cdot r_{n}(k,\ell)\cdot 
	\Big(\theta_{n-1}(\ell)\cdot d_{n-1} -1\Big)_+
	+ \sum_{\ell\in K_B} v_{n}(k,\ell)\cdot 	\theta_{n}(\ell)\cdot d_{n} -1
	\leq  \alpha\cdot \widetilde{\omega}, \quad k\in K_W,\ n> m_0,\\ 
	&\sum_{\ell\in K_W} v_{n}(k,\ell) +  
	\delta - d_n
	\leq \beta\cdot \sigma_n\cdot\theta_{n}(k)\cdot d_n,~~	
	\quad k\in K_B,\ n> m_0. 
	\end{align*}
\end{remark}


\section{Bounding performance  of non-separable systems}\label{sect:BoundingNPF}
Standard performance metrics (e.g.~throughput, mean delay, mean queue lengths) of complex systems are often directly accessible if the system is separable as  in Section \ref{sect:Separable}. This relies on the fact that the mentioned metrics can be computed when the stationary distribution is explicitly at hand.
On the other hand, computing these performance metrics for non-separable systems as in Section \ref{sect:Non-separable} is often difficult.
Van Dijk reviews methods for bounding performance metrics of stochastic systems when  values of the metric of
interest  are not explicitly available (cf.~\cite[Section 1.7, p.~62.]{dijk:11b},~\cite[p.~311.]{dijk:98},~\cite{dijk;koreuioglu:92},~\cite{dijk;wal:92}).
He developed a principle to bound performance metrics of ``non-product form systems''
by the respective metrics of related ``product form systems'' and provided examples.
Closely related to the topic of Section \ref{sect:BoundingNPF} are the ``queueing systems in random environment'' with unknown stationary distribution in \cite[Theorem 3]{economou:03a}.
There the environment process is Markov of its own.
Its generator is ``perturbed'' in  a way that bi-directional interactions between queue and environment emerge. The modified system has a stationary distribution of a special ``product form''
which is different from that developed here.

Our results in Sections \ref{sect:Separable} and \ref{sect:Non-separable} suggest to develop approximation principles for queues in a random environment when separability does not hold. The main idea is to manipulate the environment  suitably. These principles should provide approximation methods 
applicable to all the examples described in the introduction. We concentrate on bounding throughputs in production-inventory systems, i.e.~the equilibrium mean number of served customers per time unit.

The stationary distribution of the production-inventory system
with perishable items in Example \ref{ex:QueueInventory2} is in general  not of product form,  see Corollary \ref{cor:QueueInventoryPerishable-PF}.
So, in case of base stock levels $b \geq  2$ closed form expressions for performance metrics are not available when the total ageing rate  depends on whether an item from inventory is already in usage for production.
In the next example we derive product form results for  modifications of the system in Example \ref{ex:QueueInventory2}.
For simplicity we consider arrival rates $\lambda$ and service rates $\mu$ independent of the queue lengths.
\medskip
\begin{example}[Separable production-inventory system with perishable items]\label{ex:QueueInventory3}
	We con\-sider  the production-inventory system of Example \ref{ex:QueueInventory1} with perishable items 
	under ageing regimes  different from that in 
	Example \ref{ex:QueueInventory2}.
	Ageing is \emph{independent} of whether an item from the inventory is already in usage by production. If there are $k>0$ items of raw material in the inventory, then the total loss rate
	of inventory due to perishing is $\gamma\cdot d(k)$, for some function $d(\cdot)$ independent of the queue length. We call the function $k\mapsto\gamma\cdot d(k)$ ageing regime.
	The production-inventory process $Z$ has  generator $\mathbf{Q}:=\left(q(z;\tilde{z}):z,\tilde{z}\in E\right)$
	with  transition rates for $(n,k)\in E$
	as follows.
	\begin{align*}
	q\left((n,k);(n+1,k)\right) & := \lambda \cdot1_{\left\{ k\in\left\{ 1,\ldots,b\right\}\right\} },\\
	q\left((n,k);(n,k-1)\right) & := \gamma\cdot d(k) \cdot 1_{\left\{ k\in\left\{ 1,\ldots,b\right\}\right\} },\\
	q\left((n,k);(n-1,k-1)\right) & := \mu \cdot1_{\left\{ n>0\right\} }\cdot1_{\left\{ k\in\left\{ 1,\ldots,b\right\}\right\} },\\
	q\left((n,k);(n,k+1)\right) & := \nu\cdot1_{\left\{k \in\left\{ 0,1,\ldots,b-1\right\}\right\} },
	\end{align*}
	and $q(z;\tilde{z}):= 0$ for other $z\neq\tilde{z}$,
	and
	$q\left(z;z\right):= -\sum_{\substack{\tilde{z}\in E,\\
			z\neq\tilde{z}} }q\left(z;\tilde{z}\right)$ for all $z\in E.$
	
	This production-inventory system with exponentially distributed
	lifetimes of items in the inventory fits
	into the definition of the queueing system in a random environment by setting
	$K:= \left\{0,1,\ldots,b\right\},
	K_{B}:= \left\{0\right\}, K_{W}:= \left\{1,\ldots,b\right\},$
	and
	\begin{align}\label{eq:v-perish-separable}
	r_{n}(0,0)&:= 1,\qquad r_{n}(k,k-1):= 1,\quad1\leq k\leq b,\ n\in\mathbb{N},
	\end{align}
	and $r_n(k,\ell):=0$ for other $k,\ell\in K,\ n\in\mathbb{N}$,
	\begin{align}
	v_{n}(k,\ell) & := \begin{cases}
	\nu, & \text{if } k\in\left\{ 0,1,\ldots,b-1\right\},\ \ell=k+1,\\
	\gamma\cdot d(k), & \text{if }k\in\left\{ 1,\ldots,b\right\},\ \ell=k-1, \qquad\quad n\geq0,\\
	0, & \text{otherwise for }k\neq\ell,
	\end{cases}
	\end{align}
	and $v_n\left(k,k\right):= -\sum_{\substack{\ell\in K,\\
			k\neq\ell} }v_n\left(k,\ell\right)$ for all $k\in K$ and $n\geq 0$.
	
	If $\lambda<\mu$, the production-inventory process is ergodic, and the stationary distribution $\pi$ is
	\begin{equation}\label{eq:perish-separableSteadyState-1}
	\pi(n,k):= \xi(n)\cdot \theta (k), \ (n,k)\in E,
	\end{equation}
	with suitable normalization constant  $C_{\theta}$ for $\theta$ and
	\begin{align}\label{eq:perish-separableSteadyState-2}
	\xi(n):= \left( 1- \frac{\lambda}{\mu}\right) \cdot \left(\frac{\lambda}{\mu}\right)^n,   
	\  n\in \mathbb{N}_{0},\qquad
	\theta(k):= C_{\theta}^{-1} \cdot \prod_{\ell =0}^{k-1} \frac{\nu}{\lambda + d(\ell +1 )}, \  k\in K.
	\end{align}
\end{example}	
\medskip
We will show that with varying function $d(\cdot)$
this product form result can  be used to bound throughputs for the systems with (unknown) non-product form stationary distribution of Example \ref{ex:QueueInventory2}.
We construct bounding systems according to Example \ref{ex:QueueInventory3} with the same rates $\lambda$ and $\mu$, environment $K=\left\{0,1,\ldots,b\right\},
K_{B}=\left\{0\right\}, K_{W}=\left\{1,\ldots,b\right\},$ and 
jump matrices $R_n$ and different specifications of $d(\cdot)$ in the rate matrices \eqref{eq:v-perish-separable}.
These systems are ergodic because of \eqref{eq:perish-separableSteadyState-1} and \eqref{eq:perish-separableSteadyState-2}.

\emph{A lower bound ``$-$''-system:} The  ageing regime in state $(m,k)\in E$ is $k\mapsto d^-(k) := \gamma\cdot k$. This means that all items are perishable - even
the one already reserved for production.

\emph{An upper bound ``$+$''-system:} The  ageing regime in state $(m,k)\in E$ is $k\mapsto d^+(k) := \gamma\cdot(k-1)_{+}$.
This means that one item in the inventory (if there is any) is not subject
to ageing -- even if the server is idling and no item is reserved for production.

\emph{The target  ``$o$''-system} is the production-inventory system with unknown non-product form
stationary distribution from  Example \ref{ex:QueueInventory2}.
Perishing of items depends on whether an item
from the inventory is in usage by production (i.e.~the server
is busy, or if the server is idling and the inventory is not   empty). The ageing regime in
state $(m,k)\in E$ is $k\mapsto d^o(k):=(\gamma\cdot k)\cdot1_{\left\{ n=0\right\} }+(\gamma\cdot(k-1))\cdot1_{\left\{ n>0\right\} }.$
This system is ergodic by Corollary \ref{cor:BS-PER-GER-pos-rec-2} and Example \ref{ex:random-bsp-perish}\textbf{(b)}.

In the following we write ``$\star$''-system for one of the systems specified by ``+'', ``-'', or ``o''.
With stationary distribution $\pi^{\star}:= (\pi^{\star}(m,k):(m,k)\in E)$ in all cases the throughput is
\begin{align}
TH^{\star}:=& \sum_{(m,k)\in E}\pi^{\star}(m,k)\cdot\mu\cdot1_{\left\{ m>0\right\} }\cdot1_{\left\{ k>0\right\} } 
= \sum_{m=1}^{\infty} \sum_{k=1}^{b} \left( 1- \frac{\lambda}{\mu} \right) 
\left(\frac{\lambda}{\mu}\right)^{m}
\cdot C_ {\theta^{\star}}^{-1}
\cdot \left(  \prod_{\ell =0}^{k-1} \frac{\nu}{\lambda + d^{\star}(\ell +1)} \right)
\cdot \mu  \nonumber\\
&= \sum_{k=1}^{b} C_ {\theta^{\star}}^{-1}
\cdot \left(  \prod_{\ell =0}^{k-1} \frac{\nu}{\lambda + d^{\star}(\ell +1)} \right) 
\cdot \lambda 
\cdot \sum_{m=1}^{\infty}   \left(\frac{\lambda}{\mu}\right)^{m-1}
\cdot\left( 1- \frac{\lambda}{\mu} \right)  \nonumber\\
&= \lambda \cdot C_ {\theta^{\star}}^{-1} 
\cdot \sum_{k=1}^{b} 
\left(  \prod_{\ell =0}^{k-1} \frac{\nu}{\lambda + d^{\star}(\ell +1)} \right)  \nonumber\\
&= \lambda \cdot P(Y^{\star}>0) 
=  \lambda \cdot \left(1-  P(Y^{\star}=0) \right) 
= \lambda\cdot(1 - C_ {\theta^{\star}}^{-1}). \label{eq:TH-star}
\end{align}

Following  \cite{vanderwal:89} we define Markov reward processes to determine throughputs.
$w^\star_n(m,k)$ counts the number of departures from the $\star$-system up to the time of the $n$-th jump of the process if the initial state is $(m,k)\in E$.
So $w^\star_n(m,k)/n$ is the finite time throughput up to the time of the $n$-th jump of the process started in  $(m,k)\in E$.	
It holds  \cite[Lemma 2]{vanderwal:89}:
\begin{equation}\forall (m,k)\in E:\quad          
TH^{\star}= \lim\limits_{n\to\infty} \frac{1}{n} w^\star_n(m,k). 
\label{eq:THo-1-1-StarInfty}
\end{equation}

Our starting point is the observation that for all $(n,k)\in E$ the ageing regimes are ordered:
\[
\gamma\cdot(k-1)_{+}\leq(\gamma\cdot k)\cdot1_{\left\{ n=0\right\} }+(\gamma\cdot(k-1))\cdot1_{\left\{ n>0\right\} }\leq\gamma\cdot k.
\]

We expect that the throughputs of the respective
systems are ordered the other way round. 
An intuitive explanation of this throughput ordering is:
If we have either more inventory and at least the same number of customers
in the system, or more customers in the system and at least the same stock
size of the inventory, then the system should be able to produce more output.
This leads to our following conjecture.
\label{page:Conj}
\begin{conjecture}[{Monotonicity of throughputs}]\label{conj:Monotonicity}Consider three ergodic production-inven\-tory systems with the same arrival rate $\lambda$,
	service rate $\mu$, replenishment rate $\nu$, and individual ageing
	rate $\gamma$ for items in the inventory which are subject to ageing.\\
	Then the following monotonicity property for the throughputs holds
	\begin{equation}\label{eq:THOrdering}
	TH^{-}\leq TH^{o}\leq TH^{+}.
	\end{equation}	
\end{conjecture}

\noindent At present we cannot provide a complete proof of the conjecture. In the following propositions we identify and prove special cases of  \eqref{eq:THOrdering} which support the conjecture.

\begin{proposition}\label{prop:R1}
	For the three systems described in Conjecture \ref{conj:Monotonicity} with base stock level $b=1$ the throughput ordering \eqref{eq:THOrdering} 	is true.
	
\end{proposition}
\proof
With stationary distributions from Corollary \ref{cor:QueueInventoryPerishable-PF} (for "o"-system) and 
Example \ref{ex:QueueInventory3} (\eqref{eq:TH-star}  for "$-$"- and "$+$"-system), the throughputs can be computed   explicitly and are obviously ordered
\begin{align}
TH^{-}=\frac{\lambda \cdot \mu}{\lambda  +\nu +\gamma}<
TH^{o}=\frac{\lambda \cdot \mu}{\lambda  +\nu +\gamma \cdot  \left(1- \frac{\lambda}{\mu} \right) }<
TH^{+}=\frac{\lambda \cdot \mu}{\lambda +\nu}. \nonumber
\end{align}
\endproof
\begin{proposition}\label{prop:R2}
	For "$-$"- and "$+$"-systems in Conjecture \ref{conj:Monotonicity} 
	it always holds $TH^{-}<TH^{+}$. 
\end{proposition}	
\proof
From \prettyref{eq:TH-star} it follows
\begin{align}
TH^{+}-TH^{-}&=\lambda \cdot  \left( P(Y^{-}=0) - P(Y^{+}=0) \right)
= \lambda \cdot \left( C_ {\theta^{-}}^{-1} - C_ {\theta^{+}}^{-1} \right)\nonumber\\
&=\lambda \cdot \left[  
\left( \sum_{k=0}^{b} 
\left(  \prod_{\ell =0}^{k-1} \frac{\nu}{\lambda + \gamma \cdot (\ell +1)} \right) \right)^{-1}
-\left( \sum_{k=0}^{b} 
\left(  \prod_{\ell =0}^{k-1} \frac{\nu}{\lambda + \gamma \cdot \ell} \right) \right)^{-1}
\right].\nonumber
\end{align}
For $k=0,1,\ldots ,b$ it holds for $\nu,\lambda,\gamma>0$
\begin{align}
\prod_{\ell =0}^{k-1} \frac{\nu}{\lambda + \gamma \cdot (\ell +1)}
< \prod_{\ell =0}^{k-1} \frac{\nu}{\lambda + \gamma \cdot \ell }~, \nonumber
\end{align}
which implies 
$C_ {\theta^{-}}^{-1} \leq C_ {\theta^{+}}^{-1}$  and $TH^{+} -TH^{-} >0$.
\endproof

\begin{proposition}\label{prop:isotone}
	Consider the ergodic production-inventory systems ``+'', ``$-$'', and ``o'' from Conjecture \ref{conj:Monotonicity} 
	with the same rates $\lambda,\mu,\gamma$,  and $\nu$.
	We say that the reward function $w^\star_n$ is isotone with respect to the product order on $\mathbb{N}_0\times\{0,1,\dots,b\}$ if
	\begin{equation}
	\forall (m,k), (m',k')\in E:[m\leq m' \wedge k\leq k']~\text{implies}~ [w^\star_n(m,k)\leq  w^\star_n(m',k')~\forall n\in\mathbb{N}_0].\nonumber
	\end{equation}	
	
	Then the following statements hold (see \cite[Proposition 4.3.11]{otten:18}):
	\begin{enumerate}
		\item[\textbf{(a)}] If $w^-_n$ is isotone for all $n\in\mathbb{N}$, then $TH^-\leq TH^o$.
		\item[\textbf{(b)}] If $w^+_n$ is isotone for all $n\in\mathbb{N}$, then $TH^o\leq TH^+$.	
		\item[\textbf{(c)}] If $w^o_n$ is isotone for all $n\in\mathbb{N}$, then  $TH^-\leq TH^o\leq TH^+$.
	\end{enumerate}
\end{proposition}

\proof
\textbf{(a)} If $w^-_n$ is isotone, then for all $ (m,k)\in E$ it holds $w^-_n(m,k)\leq  w^o_n(m,k)$ for all $n\in\mathbb{N}$
(see \prettyref{lem:isotone-1} in Appendix). 
From \eqref{eq:THo-1-1-StarInfty} it follows $TH^-\leq TH^o$.\\
\textbf{(b)} If $w^+_n$ is isotone, then 	for all $ (m,k)\in E$ it holds 
$	w^o_n(m,k)\leq  w^+_n(m,k)$ for all $ n\in\mathbb{N}$ (see \prettyref{lem:isotone-1} in Appendix).
From \eqref{eq:THo-1-1-StarInfty} it follows	 $TH^o\leq TH^+$.\\
\textbf{(c)} If $w^o_n$ is isotone, then  
for all $ (m,k)\in E$ it holds 	$ 
w^-_n(m,k)\leq w^o_n(m,k)\leq  w^+_n(m,k)$ for all $n\in\mathbb{N}$ (see \prettyref{lem:isotone-1} in Appendix).
From \eqref{eq:THo-1-1-StarInfty} it follows	 $TH^-\leq TH^o\leq TH^+$.
 
\endproof

\begin{proposition}\label{prop:isotoneCases}
	Consider the ergodic production-inventory systems ``+'', ``$-$'', and ``o'' from Conjecture \ref{conj:Monotonicity} 
	with the same rates $\lambda,\mu,\gamma$,  and $\nu$. Then
	\begin{enumerate}
		\item[\textbf{(a)}]
		$\lambda\leq \gamma$ implies $TH^{-}\leq TH^{o}$, and
		\item[\textbf{(b)}]
		$\mu=\gamma$ implies $TH^{-}\leq TH^{o}\leq TH^{+}$.
	\end{enumerate}
	
\end{proposition}
\proof
Utilizing ideas of  \cite{vanderwal:89},  it can be shown (see \prettyref{lem:isotone-2} in the Appendix):\\
\textbf{(a)} 	
For the "$-$"-system $\lambda\leq \gamma$ implies that $w_n^-$ is isotone. So Proposition \ref{prop:isotone}\textbf{(a)}  yields
$TH^{-}\leq TH^{o}$.\\
\textbf{(b)}
For the "$o$"-system $\mu=\gamma$ implies  that $w_n^o$ is isotone.
So Proposition \ref{prop:isotone}\textbf{(c)}  yields
$TH^{-}\leq TH^{o}\leq TH^{+}$.	
\endproof

In the following we assume that Conjecture \ref{conj:Monotonicity} holds and  derive analytically upper  bounds for the error when using the suggested  approximations $TH^{-}$ or $TH^{+}$  for  $TH^{o}$ 
by a worst-case	analysis.

We consider $TH^{\star}$ as a function of $b\in \mathbb{N}$. Then for the absolute error of the bounds it holds
\begin{align}
\begin{rcases}
|TH^{o}(b)-TH^{-}(b)|=TH^{o}(b)-TH^{-}(b)\\
|TH^{o}(b)-TH^{+}(b)|=TH^{+}(b)-TH^{o}(b) 
\end{rcases}
\leq TH^{+}(b)-TH^{-}(b) \eqqcolon AE(b), \quad b\in \mathbb{N}, \label{eq:AEb}
\end{align}
and for the relative error of the bounds it holds
\begin{align}
\begin{rcases}
\displaystyle{\frac{|TH^{o}(b)-TH^{-}(b)|}{TH^{o}(b)}}\\
\displaystyle{\frac{|TH^{o}(b)-TH^{+}(b)|}{TH^{o}(b)}}\\
\end{rcases}
\leq \frac{TH^{+}(b)-TH^{-}(b)}{TH^{-}(b)} \eqqcolon RE(b),\quad b\in \mathbb{N}. \label{eq:REb}
\end{align}
With
\begin{align}
\beta_{b}(f)\coloneqq 
\sum_{k=1}^{b} \prod_{\ell =0}^{k-1} \left( \frac{\nu}{\lambda +\gamma \cdot (\ell + f) } \right),\quad f\in\{0,1\},  \label{eq:beta}
\end{align}
we get
\begin{align}
RE(b)
&\overset{\prettyref{eq:REb}}{=}\frac{TH^{+}(b)-TH^{-}(b)}{TH^{-}(b)}
\overset{\prettyref{eq:TH-star}}{=}\frac{C_{\theta^{-}}^{-1} - C_{\theta^{+}}^{-1}}{1-C_{\theta^{-}}^{-1}}
=\frac{C_{\theta^{+}} - C_{\theta^{-}}}{C_{\theta^{+}} \cdot  \left( C_{\theta^{-}} -1\right) }
\overset{\prettyref{eq:beta}}{=} \frac{\beta_{b} (0) - \beta_{b}(1)}{  \left(1+\beta_{b}(0)\right) \cdot \beta_{b}(1) }
\nonumber
\end{align}
and
\begin{align}
AE(b)&\overset{\prettyref{eq:AEb}}{=} TH^{+}(b)-TH^{-}(b) 
\overset{\prettyref{eq:TH-star}}{=}\lambda \cdot \left( C_{\theta^{-}}^{-1} - C_{\theta^{+}}^{-1} \right)
= \lambda  \cdot \left( \frac{ C_{\theta^{+}} - C_{\theta^{-}} }{ C_{\theta^{-}} \cdot C_{\theta^{+}} }  \right) \nonumber \\
&\overset{\prettyref{eq:beta}}{=} \lambda \cdot \left( \frac{ \beta_{b}(0) - \beta_{b}(1) }{ \left(1+\beta_{b}(1)\right) \cdot \left(1+\beta_{b}(0)\right) }  \right). \nonumber
\end{align}

To assess the quality of these rough approximations we have investigated several scenarios and provide numerical outputs of the absolute and relative errors below in Figures \ref{fig:AE-fast}--\ref{fig:RE-slow}. Two preliminary facts are immediate: Under scaling of all rates with the same factor $a\in (0,\infty): \lambda\to a\cdot \lambda,$ etc. concurrently, the terms $\beta_{b}(f)$ are invariant, and therefore the relative error $RE(b)$ is scale invariant, and the absolute error $AE(b)$ scales linear in $a$ (with a constant which is scale invariant).

\begin{figure}[h]
	\begin{minipage}[b]{.49\linewidth} 
		\includegraphics[width=\linewidth]{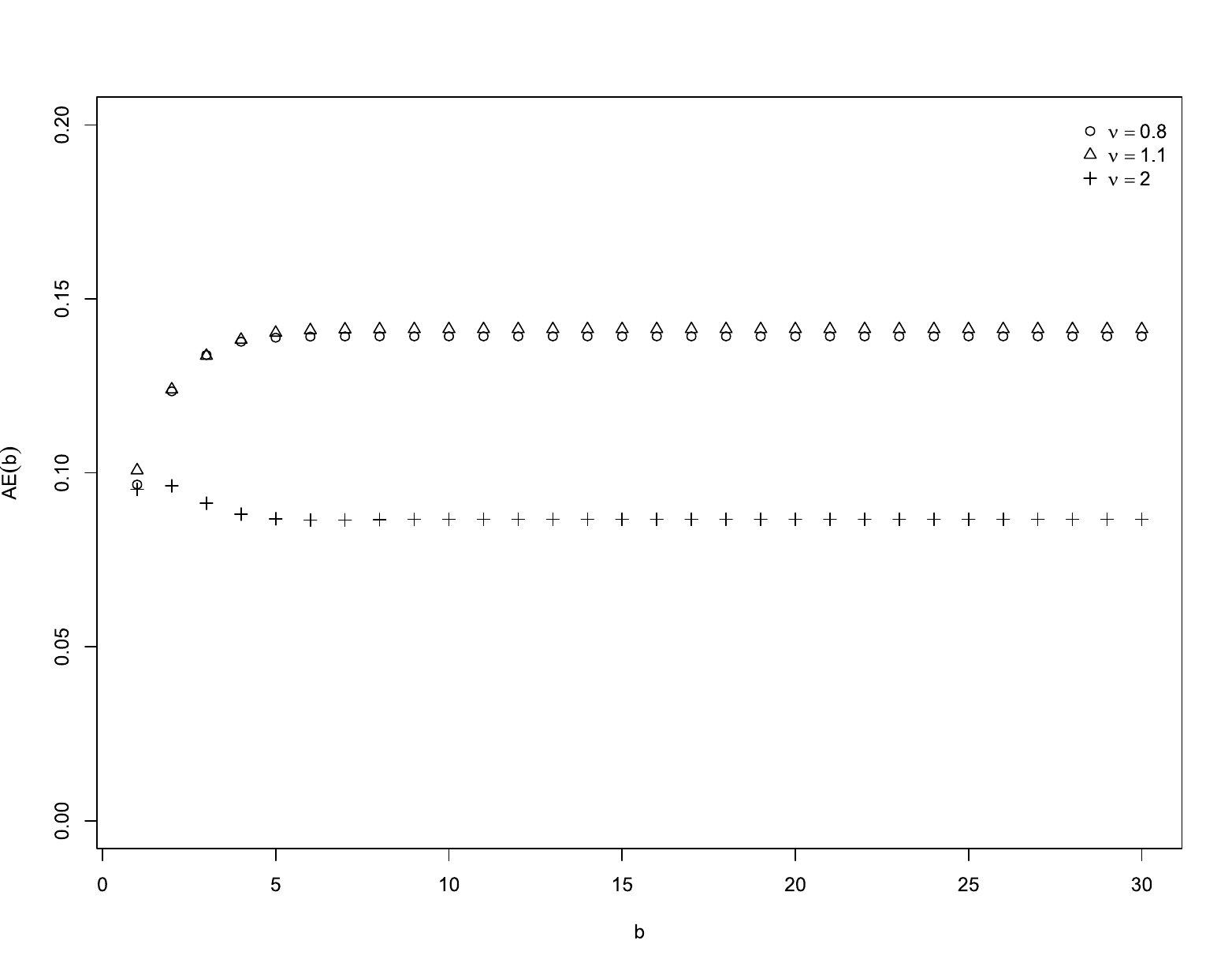}
		\caption{\label{fig:AE-fast}$AE(b)$ for fast perishing $\gamma=0.5$}
	\end{minipage}
	\hspace{.0\linewidth}
	\begin{minipage}[b]{.49\linewidth} 
		\includegraphics[width=\linewidth]{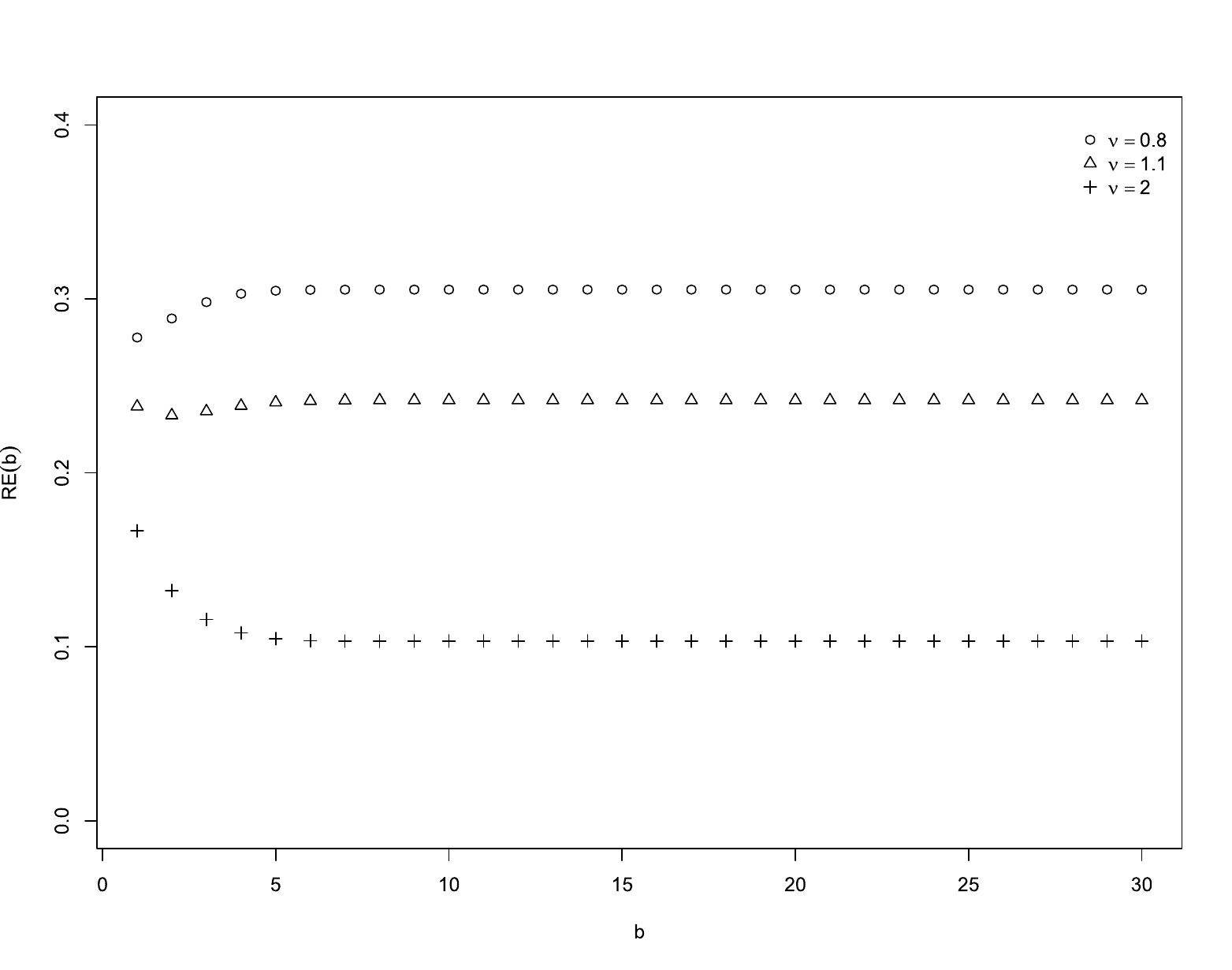}
		\caption{\label{fig:RE-fast}$RE(b)$ for fast perishing $\gamma=0.5$}
	\end{minipage}
\end{figure}
\begin{figure}[h]
	\begin{minipage}[b]{.49\linewidth} 
		\includegraphics[width=\linewidth]{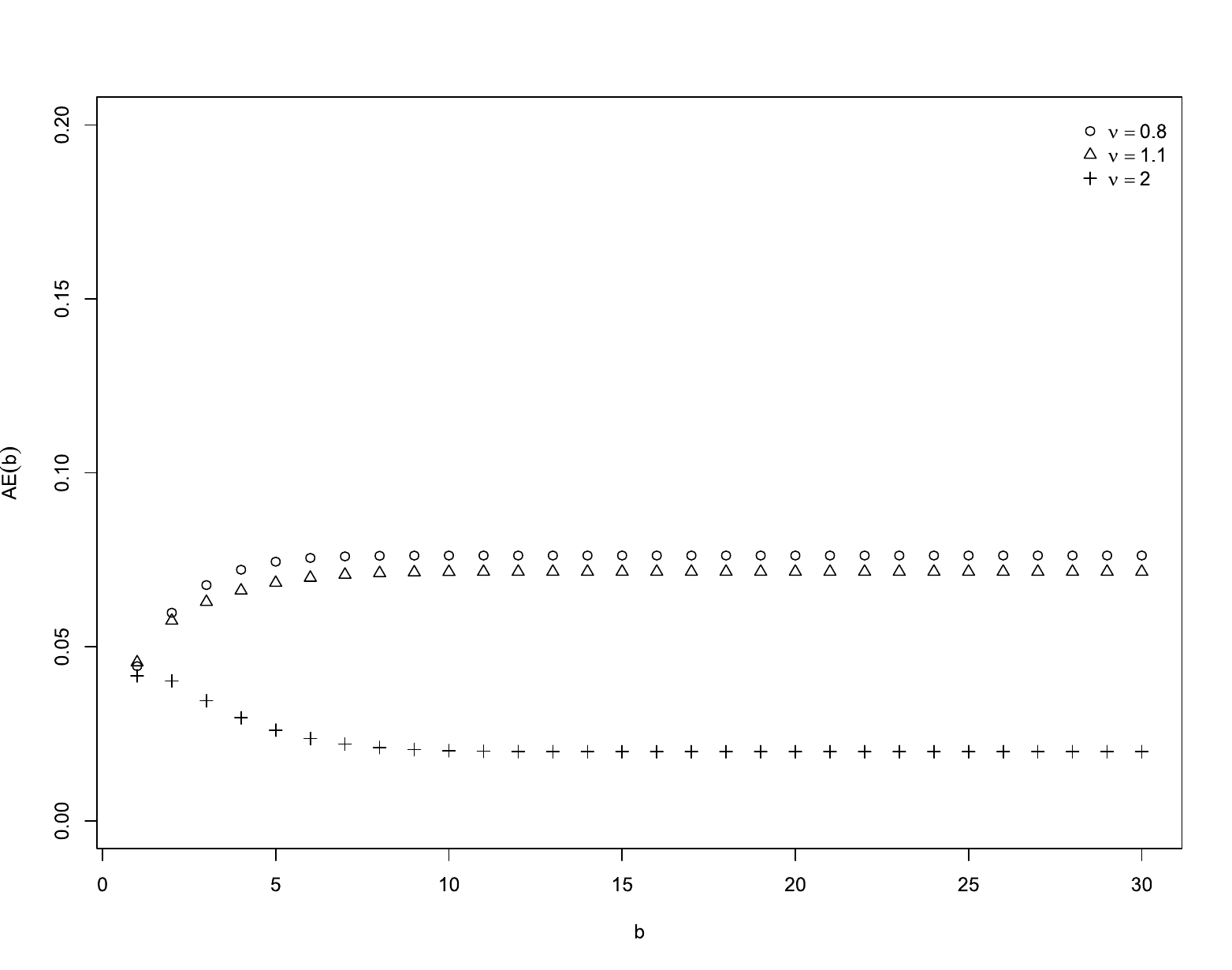}
		\caption{\label{fig:AE-moderate}$AE(b)$ for moderate perishing $\gamma=0.2$}
	\end{minipage}
	\hspace{.0\linewidth}
	\begin{minipage}[b]{.49\linewidth} 
		\includegraphics[width=\linewidth]{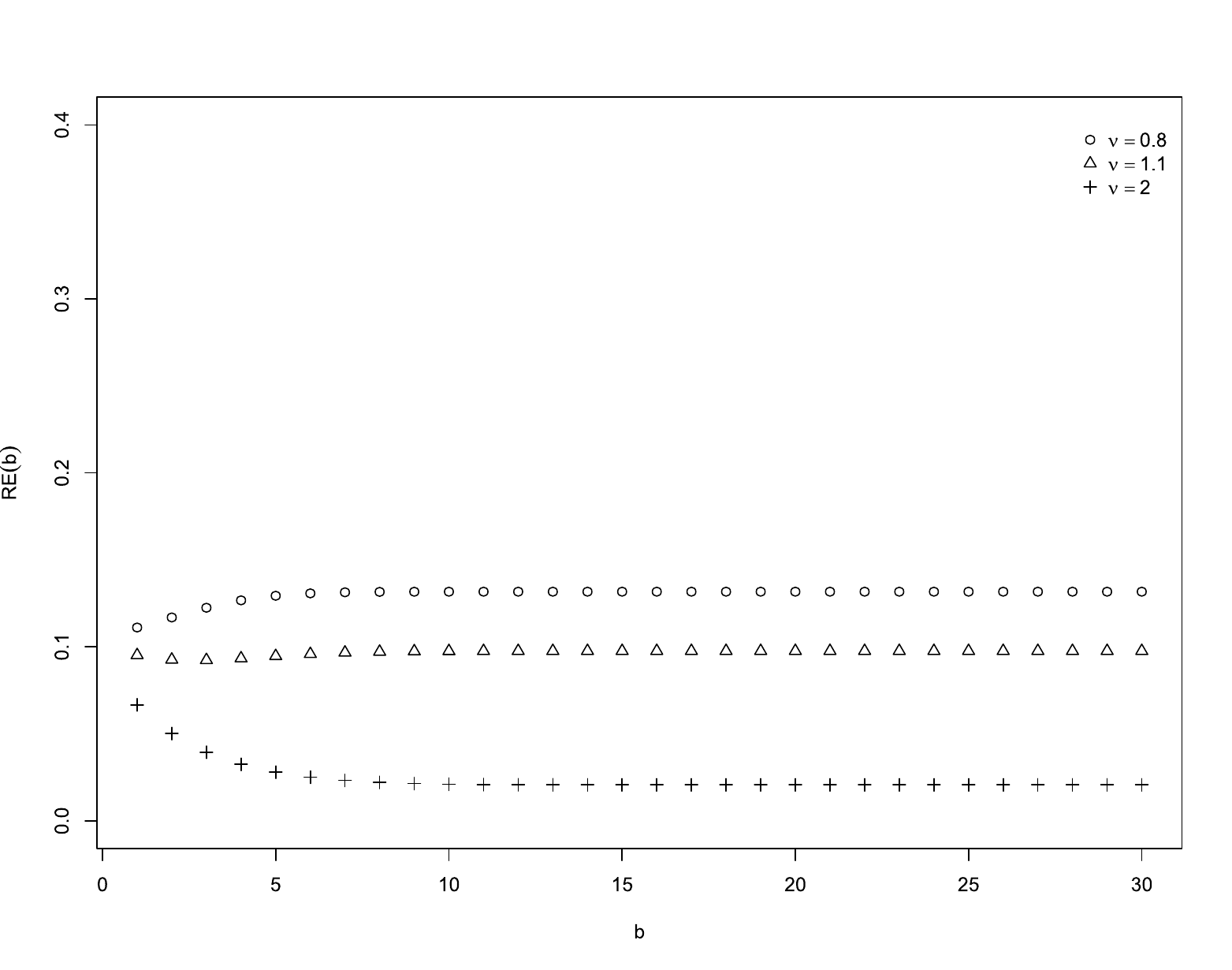}
		\caption{\label{fig:RE-moderate}$RE(b)$ for moderate perishing $\gamma=0.2$}
	\end{minipage}
\end{figure}
\begin{figure}[h]
	\begin{minipage}[b]{.49\linewidth} 
		\includegraphics[width=\linewidth]{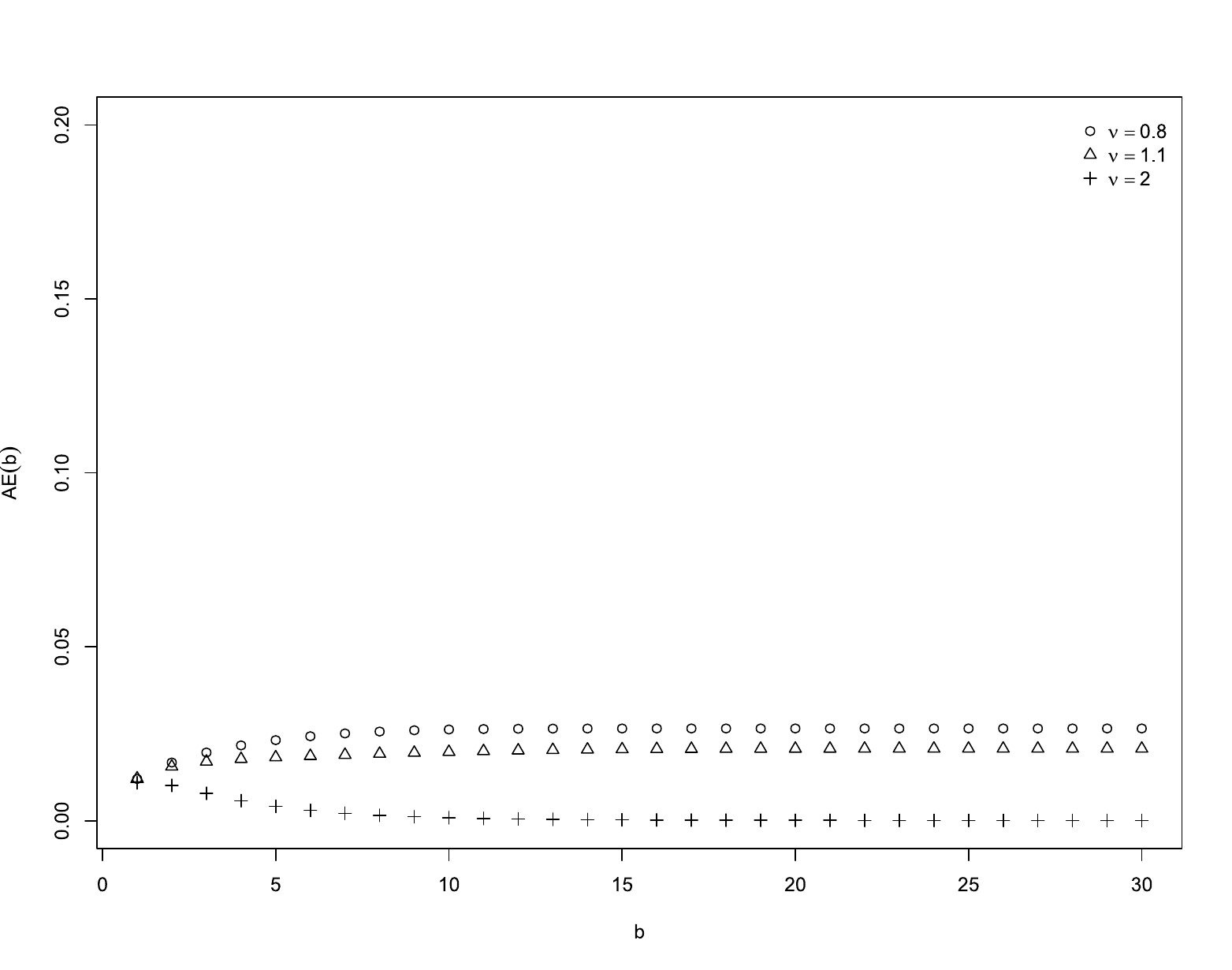}
		\caption{\label{fig:AE-slow}$AE(b)$ for slow perishing $\gamma=0.05$}
	\end{minipage}
	\hspace{.0\linewidth}
	\begin{minipage}[b]{.49\linewidth} 
		\includegraphics[width=\linewidth]{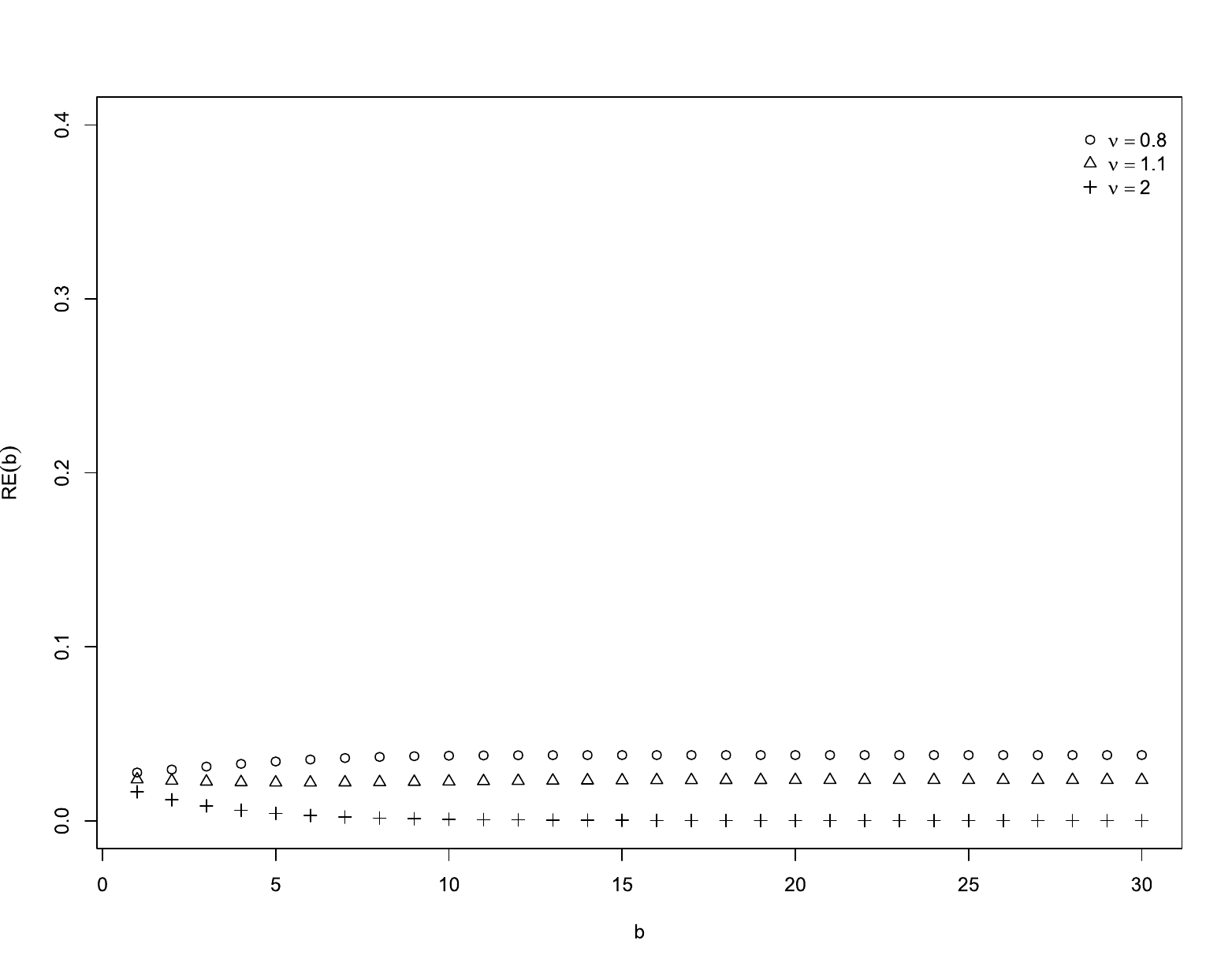}
		\caption{\label{fig:RE-slow}$RE(b)$ for slow perishing $\gamma=0.05$}
	\end{minipage}
\end{figure}

In realistic scenarios the replenishment rate $\nu$ should guarantee that enough inventory is available to allow servicing of a reasonable portion of arrivals, i.e.~we assume $\nu$
is at least of the order of $\lambda$. Moreover, the  individual perishing rate $\gamma$
should be less then the arrival rate.

Exploiting the scale invariance of $RE(b)$ and of the linear factor of $AE(b)$, we henceforth fix $\lambda =1$ and consider three scenarios with (1) fast perishing $\gamma = 0.5$, 
(2) moderate perishing $\gamma = 0.2$, and (3) slow perishing $\gamma = 0.05$.
In any scenario we consider replenishment rates $\nu= 0.8/1.1/2$.

We have included "fast perishing" in order to show that our approach is not a panacea. We believe that this extreme case is not a realistic scenario.

From Figures \ref{fig:AE-moderate}--\ref{fig:RE-slow} we see that in the moderate and slow perishing case the relative error is below 14\%. Because the bounds for the approximation errors are very rough (``worst cases''), these maximal deviations of the product form bounds from the non-product form target value
indicated by the experiments are astonishingly small.
%

\section{Conclusion}\label{sect:Conclusion}
For a large class of queueing-environment systems with bi-directional interaction we have obtained product form steady state distributions. These explicit steady state distributions open the possibility to compute performance metrics of the systems explicitly.

If this direct access is not possible, we developed ergodicity and exponential ergodicity criteria using Lyapunov functions and demonstrated by an example of a coupled production-inventory system how to compute two-sided bounds for the throughput. This indicates how to proceed in cases of other queueing-environment systems to obtain explicit bounds  for performance indices which are not directly accessible.
Another direction of research on non-separable queueing-environment systems is to modify the transition rate matrix of the system suitably (possibly only the environment coordinate) to obtain a product form approximation of the stationary distribution.
This is part of our ongoing research. 

\section*{Acknowledgements}
We thank the associated editor of \emph{Stochastic Systems} and two reviewers for their careful reading of previous versions and their constructive criticism which
enhanced the article. 


\section*{Appendix}
\begin{lemma}\label{lem:isotone-1}
	Consider the  ergodic production-inventory systems ``+'', ``$-$'', and ``o'' from Conjecture \ref{conj:Monotonicity}
	above
	with the same rates $\lambda,\mu,\gamma$,  and $\nu$.
	\begin{enumerate}
		\item[\textbf{(a)}] If $w^-_n$ is isotone, then for all $ (m,k)\in E$  it holds
		$w^-_n(m,k)\leq  w^o_n(m,k)$ for all $n\in\mathbb{N}.$
		\item[\textbf{(b)}] If $w^+_n$ is isotone, then	for all $ (m,k)\in E$ it holds
		$w^o_n(m,k)\leq  w^+_n(m,k)$ for all $ n\in\mathbb{N}.$
		\item[\textbf{(c)}] If $w^o_n$ is isotone, then  for all $ (m,k)\in E$  it holds
		$w^-_n(m,k)\leq  w^o_n(m,k)\leq  w^+_n(m,k)$ for all $ n\in\mathbb{N}.$
	\end{enumerate}
\end{lemma}

\proof
We proceed by induction over the number of jumps of the uniformization
chains $Z_{u}^{-},Z_{u}^{o},Z_{u}^{+}$ and compare the respective
cumulative rewards. By definition we have in any case
\[
w_{1}^{-}(m,k)=w_{1}^{o}(m,k)=w_{1}^{+}(m,k)=r(m,k)\quad\forall(m,k)\in E.
\]
Assume that for some $n\geq1$ it holds
\[
w_{n}^{-}(m,k)\leq w_{n}^{o}(m,k)\leq w_{n}^{+}(m,k),\quad\forall(m,k)\in E.
\]
To perform the induction step we have to show
\[
w_{n+1}^{-}(m,k)\leq w_{n+1}^{o}(m,k)\leq w_{n+1}^{+}(m,k),\quad\forall(m,k)\in E.
\]
By $w_{n+1}^{*}=r+R^{*}\cdot w_{n}^{*}$ for $*\in\{o,-,+\}$ this reduces to
\[
(R^{-}w_{n}^{-})(m,k)\leq(R^{o}w_{n}^{o})(m,k)\leq(R^{+}w_{n}^{+})(m,k),\quad\forall(m,k)\in E.
\]
Let $\alpha:=\lambda+\mu+\nu+\gamma\cdot b$. For states $(m,0),$ $m\in\mathbb{N}_{0}$, we have for $*\in\{o,-,+\}$
\[
(R^{*}w_{n}^{*})(m,0)=\frac{\nu}{\alpha}\cdot w_{n}^{*}(m,1)+\frac{\lambda+\mu+\gamma\cdot b}{\alpha}\cdot w_{n}^{*}(m,0),
\]
which proves the induction step in this case for \textbf{(a)}, \textbf{(b)},
\textbf{(c)}. The other cases need more detailed arguments. We first
compute expressions $(R^{*}w_{n}^{*})(m,k)$ and discuss then the
comparison arguments.

\noindent For state $(0,b)$ we have
\begin{align}
(R^{-}w_{n}^{-})(0,b) & =\frac{\lambda}{\alpha}\cdot w_{n}^{-}(1,b)+\frac{\gamma\cdot b}{\alpha}\cdot w_{n}^{-}(0,b-1)+\frac{\mu+\nu}{\alpha}\cdot w_{n}^{-}(0,b),\label{eq:ind1-}\\
(R^{o}w_{n}^{o})(0,b) & =\frac{\lambda}{\alpha}\cdot w_{n}^{o}(1,b)+\frac{\gamma\cdot b}{\alpha}\cdot w_{n}^{o}(0,b-1)+\frac{\mu+\nu}{\alpha}\cdot w_{n}^{o}(0,b),\label{eq:ind1o}\\
(R^{+}w_{n}^{+})(0,b) & =\frac{\lambda}{\alpha}\cdot w_{n}^{+}(1,b)+\frac{\gamma\cdot(b-1)}{\alpha}\cdot w_{n}^{+}(0,b-1)+\frac{\mu+\nu+\gamma}{\alpha}\cdot w_{n}^{+}(0,b).\label{eq:ind1+}
\end{align}

\noindent For states $(0,k)$ with $k\in\{1,\dots,b-1\}$ we have
\begin{align}
(R^{-}w_{n}^{-})(0,k)  &=\frac{\lambda}{\alpha}\cdot w_{n}^{-}(1,k)+\frac{\nu}{\alpha}\cdot w_{n}^{-}(0,k+1)+\frac{\gamma \cdot k}{\alpha}\cdot w_{n}^{-}(0,k-1)+\frac{\mu+\gamma\cdot (b-k)}{\alpha}\cdot w_{n}^{-}(0,k),\label{eq:ind2-}\\
(R^{o}w_{n}^{o})(0,k)  &=\frac{\lambda}{\alpha}\cdot w_{n}^{o}(1,k)+\frac{\nu}{\alpha}\cdot w_{n}^{o}(0,k+1)+\frac{\gamma \cdot k}{\alpha}\cdot w_{n}^{o}(0,k-1)
+\frac{\mu+\gamma\cdot (b-k)}{\alpha}\cdot w_{n}^{o}(0,k),\label{eq:ind2o}\\
(R^{+}w_{n}^{+})(0,k)  &=\frac{\lambda}{\alpha}\cdot w_{n}^{+}(1,k)+\frac{\nu}{\alpha}\cdot w_{n}^{+}(0,k+1)+\frac{\gamma\cdot (k-1)}{\alpha}\cdot w_{n}^{+}(0,k-1)\nonumber\\
&\phantomeq +\frac{\mu+\gamma\cdot (b-k+1)}{\alpha}\cdot w_{n}^{+}(0,k).\label{eq:ind2+}
\end{align}

\noindent For states $(m,b)$ with $m\in\mathbb{N}$ we have
\begin{align}
(R^{-}w_{n}^{-})(m,b)  &=\frac{\lambda}{\alpha}\cdot w_{n}^{-}(m+1,b)+\frac{\mu}{\alpha}\cdot w_{n}^{-}(m-1,b-1)
+\frac{\gamma \cdot b}{\alpha}\cdot w_{n}^{-}(m,b-1)+\frac{\nu}{\alpha}\cdot w_{n}^{-}(m,b),\label{eq:ind3-}\\
(R^{o}w_{n}^{o})(m,b)  &=\frac{\lambda}{\alpha}\cdot w_{n}^{o}(m+1,b)+\frac{\mu}{\alpha}\cdot w_{n}^{o}(m-1,b-1)
+\frac{\gamma\cdot (b-1)}{\alpha}\cdot w_{n}^{o}(m,b-1)\nonumber\\
&\phantomeq+\frac{\nu+\gamma}{\alpha}\cdot w_{n}^{o}(m,b),\label{eq:ind3o}\\
(R^{+}w_{n}^{+})(m,b)  &=\frac{\lambda}{\alpha}\cdot w_{n}^{+}(m+1,b)+\frac{\mu}{\alpha}\cdot w_{n}^{+}(m-1,b-1)
+\frac{\gamma\cdot (b-1)}{\alpha}\cdot w_{n}^{+}(m,b-1)\nonumber\\
&\phantomeq +\frac{\nu+\gamma}{\alpha}\cdot w_{n}^{+}(m,b).\label{eq:ind3+}
\end{align}

\noindent For states $(m,k)$ with $m\in\mathbb{N}$ and $k\in\{1,\dots,b-1\}$ we have
\begin{align}
(R^{-}w_{n}^{-})(m,k) & =\frac{\lambda}{\alpha}\cdot w_{n}^{-}(m+1,k)+\frac{\mu}{\alpha}\cdot w_{n}^{-}(m-1,k-1)+\frac{\nu}{\alpha}\cdot w_{n}^{-}(m,k+1)+\frac{\gamma\cdot k}{\alpha}\cdot w_{n}^{-}(m,k-1)\nonumber \\
& \phantomeq+\frac{\gamma\cdot(b-k)}{\alpha}\cdot w_{n}^{-}(m,k),\label{eq:ind4-}\\
(R^{o}w_{n}^{o})(m,k) & =\frac{\lambda}{\alpha}\cdot w_{n}^{o}(m+1,k)+\frac{\mu}{\alpha}\cdot w_{n}^{o}(m-1,k-1)+\frac{\nu}{\alpha}\cdot w_{n}^{o}(m,k+1)\nonumber \\
& \phantomeq+\frac{\gamma\cdot(k-1)}{\alpha}\cdot w_{n}^{o}(m,k-1)+\frac{\gamma\cdot(b-k+1)}{\alpha}\cdot w_{n}^{o}(m,k),\label{eq:ind4o}\\
(R^{+}w_{n}^{+})(m,k) & =\frac{\lambda}{\alpha}\cdot w_{n}^{+}(m+1,k)+\frac{\mu}{\alpha}\cdot w_{n}^{+}(m-1,k-1)+\frac{\nu}{\alpha}\cdot w_{n}^{+}(m,k+1)\nonumber \\
& \phantomeq+\frac{\gamma\cdot(k-1)}{\alpha}\cdot w_{n}^{+}(m,k-1)+\frac{\gamma\cdot(b-k+1)}{\alpha}\cdot w_{n}^{+}(m,k).\label{eq:ind4+}
\end{align}

\noindent \textbf{(a)} Comparing \eqref{eq:ind1-} and \eqref{eq:ind1o} resp.\ \eqref{eq:ind2-}
and \eqref{eq:ind2o} shows that for initial states $(0,k)$ for all
$k\in\{1,\dots,b\}$ the proposed inequality $w_{n+1}^{-}(0,k)\leq w_{n+1}^{o}(0,k)$
holds.
We rewrite \eqref{eq:ind3-} and \eqref{eq:ind3o} as
\begin{align*}
(R^{-}w_{n}^{-})(m,b) & =\frac{\lambda}{\alpha}\cdot w_{n}^{-}(m+1,b)+\frac{\mu}{\alpha}\cdot w_{n}^{-}(m-1,b-1)+\frac{\gamma\cdot(b-1)}{\alpha}\cdot w_{n}^{-}(m,b-1)\\
& \phantomeq+\frac{\nu+\gamma}{\alpha}\cdot w_{n}^{-}(m,b)+{\left[\frac{\gamma}{\alpha}\cdot w_{n}^{-}(m,b-1)-\frac{\gamma}{\alpha}\cdot w_{n}^{-}(m,b)\right]},\\
(R^{o}w_{n}^{o})(m,b) & =\frac{\lambda}{\alpha}\cdot w_{n}^{o}(m+1,b)+\frac{\mu}{\alpha}\cdot w_{n}^{o}(m-1,b-1)+\frac{\gamma\cdot(b-1)}{\alpha}\cdot w_{n}^{o}(m,b-1)+\frac{\nu+\gamma}{\alpha}\cdot w_{n}^{o}(m,b)
\end{align*}
and rewrite \eqref{eq:ind4-} and \eqref{eq:ind4o} as
\begin{align*}
(R^{-}w_{n}^{-})(m,k) & =\frac{\lambda}{\alpha}\cdot w_{n}^{-}(m+1,k)+\frac{\mu}{\alpha}\cdot w_{n}^{-}(m-1,k-1)+\frac{\nu}{\alpha}\cdot w_{n}^{-}(m,k+1)\\
& \phantomeq+\frac{\gamma\cdot(k-1)}{\alpha}\cdot w_{n}^{-}(m,k-1)+\frac{\gamma\cdot(b-k+1)}{\alpha}\cdot w_{n}^{-}(m,k)\\
& \phantomeq+\left[\frac{\gamma}{\alpha}\cdot w_{n}^{-}(m,k-1)-\frac{\gamma}{\alpha}\cdot w_{n}^{-}(m,k)\right],\\
(R^{o}w_{n}^{o})(m,k) & =\frac{\lambda}{\alpha}\cdot w_{n}^{o}(m+1,k)+\frac{\mu}{\alpha}\cdot w_{n}^{o}(m-1,k-1)+\frac{\nu}{\alpha}\cdot w_{n}^{o}(m,k+1)\\
& \phantomeq+\frac{\gamma\cdot(k-1)}{\alpha}\cdot w_{n}^{o}(m,k-1)+\frac{\gamma\cdot(b-k+1)}{\alpha}\cdot w_{n}^{o}(m,k).
\end{align*}
If $w_{n}^{-}$ is isotone, the differences in the squared brackets are non-positive.
This proves \textbf{(a)}.
$\ $\\

\noindent\textbf{(b)} Comparing
\eqref{eq:ind3o} and \eqref{eq:ind3+} resp.\ \eqref{eq:ind4o}
and \eqref{eq:ind4+} shows that for initial states $(m,k)$ with
$m\geq1$ and $k\in\{1,\dots,b\}$ the proposed inequality $w_{n+1}^{o}(m,k)\leq w_{n+1}^{+}(m,k)$
holds.\\
\noindent We rewrite \eqref{eq:ind1o} and \eqref{eq:ind1+} as
\begin{align*}
&(R^{o}w_{n}^{o})(0,b)  =\frac{\lambda}{\alpha}\cdot w_{n}^{o}(1,b)+\frac{\gamma\cdot b}{\alpha}\cdot w_{n}^{o}(0,b-1)+\frac{\mu+\nu}{\alpha}\cdot w_{n}^{o}(0,b),\\
&(R^{+}w_{n}^{+})(0,b)  =\frac{\lambda}{\alpha}\cdot w_{n}^{+}(1,b)+\frac{\gamma\cdot b}{\alpha}\cdot w_{n}^{+}(0,b-1)+\frac{\mu+\nu}{\alpha}\cdot w_{n}^{+}(0,b)+\left[\frac{\gamma}{\alpha}\cdot w_{n}^{+}(0,b)-\frac{\gamma}{\alpha}\cdot w_{n}^{+}(0,b-1)\right]
\end{align*}
and rewrite \eqref{eq:ind2o} and \eqref{eq:ind2+} as
\begin{align*}
(R^{o}w_{n}^{o})(0,k) & =\frac{\lambda}{\alpha}\cdot w_{n}^{o}(1,k)+\frac{\nu}{\alpha}\cdot w_{n}^{o}(0,k+1)+\frac{\gamma\cdot k}{\alpha}\cdot w_{n}^{o}(0,k-1)+\frac{\mu+\gamma\cdot(b-k)}{\alpha}\cdot w_{n}^{o}(0,k),\\
(R^{+}w_{n}^{+})(0,k) & =\frac{\lambda}{\alpha}\cdot w_{n}^{+}(1,k)+\frac{\nu}{\alpha}\cdot w_{n}^{+}(0,k+1)+\frac{\gamma\cdot k}{\alpha}\cdot w_{n}^{+}(0,k-1)+\frac{\mu+\gamma\cdot(b-k)}{\alpha}\cdot w_{n}^{+}(0,k)\\
& \phantomeq+\left[\frac{\gamma}{\alpha}\cdot w_{n}^{+}(0,k)-\frac{\gamma}{\alpha}\cdot w_{n}^{+}(0,k-1)\right].
\end{align*}
If $w_{n}^{+}$ is isotone, the differences in the squared brackets are non-negative. 
This proves \textbf{(b)}.\\
$\ $\\
\noindent \textbf{(c)} To prove the two-sided bounds $w_{n+1}^{-}(m,k)\leq w_{n+1}^{o}(m,k)\leq w_{n+1}^{+}(m,k)$
we first check again \eqref{eq:ind1-} and \eqref{eq:ind1o}
resp.\ \eqref{eq:ind2-} and \eqref{eq:ind2o} and see that for initial
states $(0,k)$ for all $k\in\{1,\dots,b\}$ the proposed inequality
$w_{n+1}^{-}(0,k)\leq w_{n+1}^{o}(0,k)$ holds. 
We rewrite \eqref{eq:ind1o} and \eqref{eq:ind1+} as 
\begin{align*}
(R^{o}w_{n}^{o})(0,b) & =\frac{\lambda}{\alpha}\cdot w_{n}^{o}(1,b)+\frac{\gamma\cdot(b-1)}{\alpha}\cdot w_{n}^{o}(0,b-1)+\frac{\mu+\nu+\gamma}{\alpha}\cdot w_{n}^{o}(0,b)\\
& \phantomeq+\left[\frac{\gamma}{\alpha}\cdot w_{n}^{o}(0,b-1)-\frac{\gamma}{\alpha}\cdot w_{n}^{o}(0,b)\right],\\
(R^{+}w_{n}^{+})(0,b) & =\frac{\lambda}{\alpha}\cdot w_{n}^{+}(1,b)+\frac{\gamma\cdot(b-1)}{\alpha}\cdot w_{n}^{+}(0,b-1)+\frac{\mu+\nu+\gamma}{\alpha}\cdot w_{n}^{+}(0,b)
\end{align*}
and rewrite \eqref{eq:ind2o} and \eqref{eq:ind2+} as
\begin{align*}
(R^{o}w_{n}^{o})(0,k) & =\frac{\lambda}{\alpha}\cdot w_{n}^{o}(1,k)+\frac{\nu}{\alpha}\cdot w_{n}^{o}(0,k+1)+\frac{\gamma\cdot(k-1)}{\alpha}\cdot w_{n}^{o}(0,k-1)\\
& \phantomeq+\frac{\mu+\gamma\cdot(b-k+1)}{\alpha}\cdot w_{n}^{o}(0,k)+\left[\frac{\gamma}{\alpha}\cdot w_{n}^{o}(0,k-1)-\frac{\gamma}{\alpha}\cdot w_{n}^{o}(0,k)\right],\\
(R^{+}w_{n}^{+})(0,k) & =\frac{\lambda}{\alpha}\cdot w_{n}^{+}(1,k)+\frac{\nu}{\alpha}\cdot w_{n}^{+}(0,k+1)+\frac{\gamma\cdot(k-1)}{\alpha}\cdot w_{n}^{+}(0,k-1)\\
& \phantomeq+\frac{\mu+\gamma\cdot(b-k+1)}{\alpha}\cdot w_{n}^{+}(0,k).
\end{align*}
If $w_{n}^{o}$ is isotone, the differences in squared brackets
are non-positive, which proves this part of \textbf{(c)}.\\

We next check \eqref{eq:ind3o} and \eqref{eq:ind3+} resp.\ \eqref{eq:ind4o}
and \eqref{eq:ind4+} and see that for initial states $(m,k)$ with
$m\geq1$ and $k\in\{1,\dots,b\}$ the proposed inequality $w_{n+1}^{o}(m,k)\leq w_{n+1}^{+}(m,k)$
holds.
We rewrite \eqref{eq:ind3-} and \eqref{eq:ind3o} as
\begin{align*}
(R^{-}w_{n}^{-})(m,b) & =\frac{\lambda}{\alpha}\cdot w_{n}^{-}(m+1,b)+\frac{\mu}{\alpha}\cdot w_{n}^{-}(m-1,b-1)+\frac{\gamma\cdot b}{\alpha}\cdot w_{n}^{-}(m,b-1)+\frac{\nu}{\alpha}\cdot w_{n}^{-}(m,b),\\
(R^{o}w_{n}^{o})(m,b) & =\frac{\lambda}{\alpha}\cdot w_{n}^{o}(m+1,b)+\frac{\mu}{\alpha}\cdot w_{n}^{o}(m-1,b-1)+\frac{\gamma\cdot b}{\alpha}\cdot w_{n}^{o}(m,b-1)+\frac{\nu}{\alpha}\cdot w_{n}^{o}(m,b)\\
& \phantomeq+\left[\frac{\gamma}{\alpha}\cdot w_{n}^{o}(m,b)-\frac{\gamma}{\alpha}\cdot w_{n}^{o}(m,b-1)\right]
\end{align*}
and rewrite \eqref{eq:ind4-} and \eqref{eq:ind4o} as
\begin{align*}
(R^{-}w_{n}^{-})(m,k) & =\frac{\lambda}{\alpha}\cdot w_{n}^{-}(m+1,k)+\frac{\mu}{\alpha}\cdot w_{n}^{-}(m-1,k-1)+\frac{\nu}{\alpha}\cdot w_{n}^{-}(m,k+1)\\
& \phantomeq+\frac{\gamma\cdot k}{\alpha}\cdot w_{n}^{-}(m,k-1)+\frac{\gamma\cdot(b-k)}{\alpha}\cdot w_{n}^{-}(m,k),\\
(R^{o}w_{n}^{o})(m,k) & =\frac{\lambda}{\alpha}\cdot w_{n}^{o}(m+1,k)+\frac{\mu}{\alpha}\cdot w_{n}^{o}(m-1,k-1)+\frac{\nu}{\alpha}\cdot w_{n}^{o}(m,k+1)\\
& \phantomeq+\frac{\gamma\cdot k}{\alpha}\cdot w_{n}^{o}(m,k-1)+\frac{\gamma\cdot(b-k)}{\alpha}\cdot w_{n}^{o}(m,k)+\left[\frac{\gamma}{\alpha}\cdot w_{n}^{o}(m,k)-\frac{\gamma}{\alpha}\cdot w_{n}^{o}(m,k-1)\right].
\end{align*}
If $w_{n}^{o}$ is isotone, the differences in squared brackets
are non-negative. 
This proves the remaining part of \textbf{(c)}. 

\endproof


\begin{lemma}\label{lem:isotone-2}
	Consider the ergodic production-inventory systems ``$-$'', and ``o'' from Conjecture \ref{conj:Monotonicity}
	with corresponding rates $\lambda,\mu,\gamma$,  and $\nu$.
	\begin{enumerate}
		\item[\textbf{(a)}]For the "$-$"-system it holds: $\lambda\leq \gamma$ implies that $w_n^-$ is isotone for all $n\in\mathbb{N}$.
		\item[\textbf{(b)}]
		For the "$o$"-system it holds: $\mu=\gamma$ implies  that $w_n^o$ is isotone for all $n\in\mathbb{N}$.
	\end{enumerate}	
\end{lemma}

\proof

\textbf{(a)} 
We show by induction that for all  $n\in\mathbb{N}$ it holds  with $\alpha:=\lambda+\mu+\nu+\gamma\cdot b$
\begin{align}
w_{n}^{-}(m,k)-w_{n}^{-}(m,k-1)\geq0, &  && \forall k\in\{1,\dots,b\},\ m\in\mathbb{N}_{0},\label{eq:iso-A}\\
w_{n}^{-}(m+1,k)-w_{n}^{-}(m,k)\geq0, &  && \forall k\in\{0,1,\dots,b\},\ m\in\mathbb{N}_{0},\label{eq:iso-B}\\
w_{n}^{-}(m+1,k)-w_{n}^{-}(m,k)\leq\alpha, &  && \forall k\in\{0,1,\dots,b\},\ m\in\mathbb{N}_{0},\label{eq:iso-C}\\
w_{n}^{-}(m,k)-w_{n}^{-}(m,k-1)\leq\alpha, &  && \forall k\in\{1,\dots,b\},\ m\in\mathbb{N}_{0}.\label{eq:iso-D}
\end{align}
For $n=1$ we have $w_{1}^{-}(m,k)=r(m,k)=\mu\cdot1_{\left\{ m>0\right\} }\cdot1_{\left\{ k>0\right\} }$ for all $k\in\{0,1,\dots,b\},\ m\in\mathbb{N}_{0},$
so \eqref{eq:iso-A}--\eqref{eq:iso-D} are trivially true.

Let $n\in\mathbb{N}$ such that \eqref{eq:iso-A}--\eqref{eq:iso-D} hold.
We shall verify these properties for $n$ replaced by $n+1$. In any case we
exploit $w_{n+1}^{-}=r+R^{-}w_{n}^{-},$ $n\geq1$. \\

\noindent$\blacktriangleright$ First, we check \eqref{eq:iso-A}.
For $m=0$ and $k=1$ it holds 
\begin{align*}
& \phantomeq w_{n+1}^{-}(0,1)-w_{n+1}^{-}(0,0)\\
& =\left[\frac{\lambda}{\alpha}\cdot w_{n}^{-}(1,1)+\frac{\nu}{\alpha}\cdot w_{n}^{-}(0,2)+\frac{\gamma}{\alpha}\cdot w_{n}^{-}(0,0)+\frac{\mu+\gamma\cdot(b-1)}{\alpha}\cdot w_{n}^{-}(0,1)\right]\\
& \phantomeq-\left[\frac{\nu}{\alpha}\cdot w_{n}^{-}(0,1)+\frac{\lambda+\mu+\gamma\cdot b}{\alpha}\cdot w_{n}^{-}(0,0)\right]\\
& =\frac{\lambda}{\alpha}\cdot\underbrace{(w_{n}^{-}(1,1)-w_{n}^{-}(1,0))}_{\geq0,~~\text{by }\eqref{eq:iso-A}}+\frac{\lambda}{\alpha}\cdot\underbrace{(w_{n}^{-}(1,0)-w_{n}^{-}(0,0))}_{\geq0,~~\text{by }\eqref{eq:iso-B}}+\frac{\nu}{\alpha}\cdot\underbrace{(w_{n}^{-}(0,2)-w_{n}^{-}(0,1))}_{\geq0,~~\text{by }\eqref{eq:iso-A}}\\
& \phantomeq+\frac{\mu+\gamma\cdot(b-1)}{\alpha}\cdot\underbrace{(w_{n}^{-}(0,1)-w_{n}^{-}(0,0))}_{\geq0,~~\text{by }\eqref{eq:iso-A}}\geq0.
\end{align*}
For $m=0$ and $k\in\{2,\dots,b-1\}$ it holds
\begin{align*}
& \phantomeq w_{n+1}^{-}(0,k)-w_{n+1}^{-}(0,k-1)\\
& =\left[\frac{\lambda}{\alpha}\cdot w_{n}^{-}(1,k)+\frac{\nu}{\alpha}\cdot w_{n}^{-}(0,k+1)+\frac{\gamma\cdot k}{\alpha}\cdot w_{n}^{-}(0,k-1)+\frac{\mu+\gamma\cdot(b-k)}{\alpha}\cdot w_{n}^{-}(0,k)\right]\\
& \phantomeq-\left[\frac{\lambda}{\alpha}\cdot w_{n}^{-}(1,k-1)+\frac{\nu}{\alpha}\cdot w_{n}^{-}(0,k)+\frac{\gamma\cdot(k-1)}{\alpha}\cdot w_{n}^{-}(0,k-2)+\frac{\mu+\gamma\cdot(b-k+1)}{\alpha}\cdot w_{n}^{-}(0,k-1)\right]\\
& =\frac{\lambda}{\alpha}\cdot\underbrace{(w_{n}^{-}(1,k)-w_{n}^{-}(1,k-1))}_{\geq0,~~\text{by }\eqref{eq:iso-A}}+\frac{\nu}{\alpha}\cdot\underbrace{(w_{n}^{-}(0,k+1)-w_{n}^{-}(0,k))}_{\geq0,~~\text{by }\eqref{eq:iso-A}}\qquad\qquad\quad\quad\qquad\qquad\\
& \phantomeq+\frac{\gamma\cdot(k-1)}{\alpha}\cdot\underbrace{(w_{n}^{-}(0,k-1)-w_{n}^{-}(0,k-2))}_{\geq0,~~\text{by }\eqref{eq:iso-A}}
+\frac{\mu+\gamma\cdot(b-k)}{\alpha}\cdot\underbrace{(w_{n}^{-}(0,k)-w_{n}^{-}(0,k-1))}_{\geq0,~~\text{by }\eqref{eq:iso-A}}\geq0.
\end{align*}
For $m=0$ and $k=b$ it holds
\begin{align*}
& \phantomeq w_{n+1}^{-}(0,b)-w_{n+1}^{-}(0,b-1)\\
& =\left[\frac{\lambda}{\alpha}\cdot w_{n}^{-}(1,b)+\frac{\gamma\cdot b}{\alpha}\cdot w_{n}^{-}(0,b-1)+\frac{\mu+\nu}{\alpha}\cdot w_{n}^{-}(0,b)\right]\\
& \phantomeq-\left[\frac{\lambda}{\alpha}\cdot w_{n}^{-}(1,b-1)+\frac{\nu}{\alpha}\cdot w_{n}^{-}(0,b)+\frac{\gamma\cdot(b-1)}{\alpha}\cdot w_{n}^{-}(0,b-2)+\frac{\mu+\gamma}{\alpha}\cdot w_{n}^{-}(0,b-1)\right]\\
& =\frac{\lambda}{\alpha}\cdot\underbrace{(w_{n}^{-}(1,b)-w_{n}^{-}(1,b-1))}_{\geq0,~~\text{by }\eqref{eq:iso-A}}+\frac{\gamma\cdot(b-1)}{\alpha}\cdot\underbrace{(w_{n}^{-}(0,b-1)-w_{n}^{-}(0,b-2))}_{\geq0,~~\text{by }\eqref{eq:iso-A}}\\
& \phantomeq+\frac{\mu}{\alpha}\cdot\underbrace{(w_{n}^{-}(0,b)-w_{n}^{-}(0,b-1))}_{\geq0,~~\text{by }\eqref{eq:iso-A}}\geq0.
\end{align*}
For $m\geq1$ and $k=1$ it holds
\begin{align*}
& \phantomeq w_{n+1}^{-}(m,1)-w_{n+1}^{-}(m,0)\\
& =\left[\mu+\frac{\lambda}{\alpha}\cdot w_{n}^{-}(m+1,1)+\frac{\mu}{\alpha}\cdot w_{n}^{-}(m-1,0)+\frac{\nu}{\alpha}\cdot w_{n}^{-}(m,2)+\frac{\gamma}{\alpha}\cdot w_{n}^{-}(m,0)+\frac{\gamma\cdot(b-1)}{\alpha}\cdot w_{n}^{-}(m,1)\right]\\
& \phantomeq-\left[0+\frac{\nu}{\alpha}\cdot w_{n}^{-}(m,1)+\frac{\lambda+\mu+\gamma\cdot b}{\alpha}\cdot w_{n}^{-}(m,0)\right]\\
& =\frac{\lambda}{\alpha}\cdot\underbrace{(w_{n}^{-}(m+1,1)-w_{n}^{-}(m,1))}_{\geq0,~~\text{by }\eqref{eq:iso-B}}+\frac{\lambda}{\alpha}\cdot\underbrace{(w_{n}^{-}(m,1)-w_{n}^{-}(m,0))}_{\geq0,~~\text{by }\eqref{eq:iso-A}}+\frac{\nu}{\alpha}\cdot\underbrace{(w_{n}^{-}(m,2)-w_{n}^{-}(m,1))}_{\geq0,~~\text{by }\eqref{eq:iso-A}}\\
& \phantomeq+\frac{\gamma\cdot(b-1)}{\alpha}\cdot\underbrace{(w_{n}^{-}(m,1)-w_{n}^{-}(m,0))}_{\geq0,~~\text{by }\eqref{eq:iso-A}}+\underbrace{\mu-\frac{\mu}{\alpha}\cdot\underbrace{(w_{n}^{-}(m,0)-w_{n}^{-}(m-1,0))}_{\in[0,\alpha],~~\text{by }\eqref{eq:iso-B},\ \eqref{eq:iso-C}}}_{\geq0}\geq0.
\end{align*}
For $m\geq1$ and $k\in\{2,\dots,b-1\}$ it holds
\begin{align*}
& \phantomeq w_{n+1}^{-}(m,k)-w_{n+1}^{-}(m,k-1)\\
& =\left[\mu+\frac{\lambda}{\alpha}\cdot w_{n}^{-}(m+1,k)+\frac{\mu}{\alpha}\cdot w_{n}^{-}(m-1,k-1)+\frac{\nu}{\alpha}\cdot w_{n}^{-}(m,k+1)\right.\\
& \phantomeq\left.+\frac{\gamma\cdot k}{\alpha}\cdot w_{n}^{-}(m,k-1)+\frac{\gamma\cdot(b-k)}{\alpha}\cdot w_{n}^{-}(m,k)\right]\\
& \phantomeq-\left[\mu+\frac{\lambda}{\alpha}\cdot w_{n}^{-}(m+1,k-1)+\frac{\mu}{\alpha}\cdot w_{n}^{-}(m-1,k-2)+\frac{\nu}{\alpha}\cdot w_{n}^{-}(m,k)\right.\\
& \phantomeq+\left.\frac{\gamma\cdot(k-1)}{\alpha}\cdot w_{n}^{-}(m,k-2)+\frac{\gamma\cdot(b-k+1)}{\alpha}\cdot w_{n}^{-}(m,k-1)\right]\\
& =\frac{\lambda}{\alpha}\cdot\underbrace{(w_{n}^{-}(m+1,k)-w_{n}^{-}(m+1,k-1))}_{\geq0,~~\text{by }\eqref{eq:iso-A}}+\frac{\mu}{\alpha}\cdot\underbrace{(w_{n}^{-}(m-1,k-1)-w_{n}^{-}(m-1,k-2))}_{\geq0,~~\text{by }\eqref{eq:iso-A}}\\
& \phantomeq+\frac{\nu}{\alpha}\cdot\underbrace{(w_{n}^{-}(m,k+1)-w_{n}^{-}(m,k))}_{\geq0,~~\text{by }\eqref{eq:iso-A}}+\frac{\gamma\cdot(k-1)}{\alpha}\cdot\underbrace{(w_{n}^{-}(m,k-1)-w_{n}^{-}(m,k-2))}_{\geq0,~~\text{by }\eqref{eq:iso-A}}\\
& \phantomeq+\frac{\gamma\cdot(b-k)}{\alpha}\cdot\underbrace{(w_{n}^{-}(m,k)-w_{n}^{-}(m,k-1))}_{\geq0,~~\text{by }\eqref{eq:iso-A}}\geq0.
\end{align*}
For $m\geq1$ and $k=b$ holds
\begin{align*}
& \phantomeq w_{n+1}^{-}(m,b)-w_{n+1}^{-}(m,b-1)\\
& =\left[\mu+\frac{\lambda}{\alpha}\cdot w_{n}^{-}(m+1,b)+\frac{\mu}{\alpha}\cdot w_{n}^{-}(m-1,b-1)+\frac{\gamma\cdot b}{\alpha}\cdot w_{n}^{-}(m,b-1)+\frac{\nu}{\alpha}w_{n}^{-}(m,b)\right]\\
& \phantomeq-\left[\mu+\frac{\lambda}{\alpha}\cdot w_{n}^{-}(m+1,b-1)+\frac{\mu}{\alpha}\cdot w_{n}^{-}(m-1,b-2)+\frac{\nu}{\alpha}\cdot w_{n}^{-}(m,b)\right.\\
& \phantomeq+\left.\frac{\gamma\cdot(b-1)}{\alpha}\cdot w_{n}^{-}(m,b-2)+\frac{\gamma}{\alpha}\cdot w_{n}^{-}(m,b-1)\right]\\
& =\frac{\lambda}{\alpha}\cdot\underbrace{(w_{n}^{-}(m+1,b)-w_{n}^{-}(m+1,b-1))}_{\geq0,~~\text{by }\eqref{eq:iso-A}}+\frac{\mu}{\alpha}\cdot\underbrace{(w_{n}^{-}(m-1,b-1)-w_{n}^{-}(m-1,b-2))}_{\geq0,~~\text{by }\eqref{eq:iso-A}}\\
& \phantomeq+\frac{\gamma\cdot(b-1)}{\alpha}\cdot\underbrace{(w_{n}^{-}(m,b-1)-w_{n}^{-}(m,b-2))}_{\geq0,~~\text{by }\eqref{eq:iso-A}}\geq0.
\end{align*}
\noindent$\blacktriangleright$ Second, we check \eqref{eq:iso-B}.
For $m=0$ and $k=0$ it holds 
\begin{align*}
& \phantomeq w_{n+1}^{-}(1,0)-w_{n+1}^{-}(0,0)\\
& =\left[\frac{\nu}{\alpha}\cdot w_{n}^{-}(1,1)+\frac{\lambda+\mu+\gamma\cdot b}{\alpha}\cdot w_{n}^{-}(1,0)\right]-\left[\frac{\nu}{\alpha}\cdot w_{n}^{-}(0,1)+\frac{\lambda+\mu+\gamma\cdot b}{\alpha}\cdot w_{n}^{-}(0,0)\right]\\
& =\frac{\nu}{\alpha}\cdot\underbrace{(w_{n}^{-}(1,1)-w_{n}^{-}(0,1))}_{\geq0,~~\text{by }\eqref{eq:iso-B}}+\frac{\lambda+\mu+\gamma\cdot b}{\alpha}\cdot\underbrace{(w_{n}^{-}(1,0)-w_{n}^{-}(0,0))}_{\geq0,~~\text{by }\eqref{eq:iso-B}}\geq0.
\end{align*}
For $m=0$ and $k\in\{1,\dots,b-1\}$ it holds
\begin{align*}
& \phantomeq w_{n+1}^{-}(1,k)-w_{n+1}^{-}(0,k)\\
& =\mu+\frac{\lambda}{\alpha}\cdot w_{n}^{-}(2,k)+\frac{\mu}{\alpha}\cdot w_{n}^{-}(0,k-1)+\frac{\nu}{\alpha}\cdot w_{n}^{-}(1,k+1)+\frac{\gamma\cdot k}{\alpha}\cdot w_{n}^{-}(1,k-1)+\frac{\gamma\cdot(b-k)}{\alpha}\cdot w_{n}^{-}(1,k)\\
& \phantomeq-\left[0+\frac{\lambda}{\alpha}\cdot w_{n}^{-}(1,k)+\frac{\nu}{\alpha}\cdot w_{n}^{-}(0,k+1)+\frac{\gamma\cdot k}{\alpha}\cdot w_{n}^{-}(0,k-1)+\frac{\mu+\gamma\cdot(b-k)}{\alpha}\cdot w_{n}^{-}(0,k)\right]\\
& =\frac{\lambda}{\alpha}\cdot\underbrace{(w_{n}^{-}(2,k)-w_{n}^{-}(1,k))}_{\geq0,~~\text{by }\eqref{eq:iso-B}}+\frac{\nu}{\alpha}\cdot\underbrace{(w_{n}^{-}(1,k+1)-w_{n}^{-}(0,k+1))}_{\geq0,~~\text{by }\eqref{eq:iso-B}}
+\frac{\gamma\cdot k}{\alpha}\cdot\underbrace{(w_{n}^{-}(1,k-1)-w_{n}^{-}(0,k-1))}_{\geq0,~~\text{by }\eqref{eq:iso-B}}\\
& \phantomeq+\frac{\gamma\cdot(b-k)}{\alpha}\cdot\underbrace{(w_{n}^{-}(1,k)-w_{n}^{-}(0,k))}_{\geq0,~~\text{by }\eqref{eq:iso-B}}
+\underbrace{\mu-\frac{\mu}{\alpha}\cdot\underbrace{(w_{n}^{-}(0,k)-w_{n}^{-}(0,k-1))}_{\in[0,\alpha],~~\text{by }\eqref{eq:iso-A},\ \eqref{eq:iso-D}}}_{\geq0}\geq0.
\end{align*}
For $m=0$ and $k=b$ it holds
\begin{align*}
& \phantomeq w_{n+1}^{-}(1,b)-w_{n+1}^{-}(0,b)\\
& =\left[\mu+\frac{\lambda}{\alpha}\cdot w_{n}^{-}(2,b)+\frac{\mu}{\alpha}\cdot w_{n}^{-}(0,b-1)+\frac{\gamma\cdot b}{\alpha}\cdot w_{n}^{-}(1,b-1)+\frac{\nu}{\alpha}\cdot w_{n}^{-}(1,b)\right]\\
& \phantomeq-\left[0+\frac{\lambda}{\alpha}\cdot w_{n}^{-}(1,b)+\frac{\gamma\cdot b}{\alpha}\cdot w_{n}^{-}(0,b-1)+\frac{\mu+\nu}{\alpha}\cdot w_{n}^{-}(0,b)\right]\\
& =\frac{\lambda}{\alpha}\cdot\underbrace{(w_{n}^{-}(2,b)-w_{n}^{-}(1,b))}_{\geq0,~~\text{by }\eqref{eq:iso-B}}+\frac{\nu}{\alpha}\cdot\underbrace{(w_{n}^{-}(1,b)-w_{n}^{-}(0,b))}_{\geq0,~~\text{by }\eqref{eq:iso-B}}\\
& \phantomeq+\frac{\gamma\cdot b}{\alpha}\cdot\underbrace{(w_{n}^{-}(1,b-1)-w_{n}^{-}(0,b-1))}_{\geq0,~~\text{by }\eqref{eq:iso-B}}+\underbrace{\mu-\frac{\mu}{\alpha}\cdot\underbrace{(w_{n}^{-}(0,b)-w_{n}^{-}(0,b-1))}_{\in[0,\alpha],~~\text{by }\eqref{eq:iso-A},\ \eqref{eq:iso-D}}}_{\geq0}\geq0.
\end{align*}
For $m\geq1$ and $k=0$ it holds
\begin{align*}
& \phantomeq w_{n+1}^{-}(m+1,0)-w_{n+1}^{-}(m,0)\\
& =\left[\frac{\nu}{\alpha}\cdot w_{n}^{-}(m+1,1)+\frac{\lambda+\mu+\gamma\cdot b}{\alpha}\cdot w_{n}^{-}(m+1,0)\right]-\left[\frac{\nu}{\alpha}\cdot w_{n}^{-}(m,1)+\frac{\lambda+\mu+\gamma\cdot b}{\alpha}\cdot w_{n}^{-}(m,0)\right]\\
& =\frac{\nu}{\alpha}\cdot\underbrace{(w_{n}^{-}(m+1,1)-w_{n}^{-}(m,1))}_{\geq0,~~\text{by }\eqref{eq:iso-B}}+\frac{\lambda+\mu+\gamma\cdot b}{\alpha}\cdot\underbrace{(w_{n}^{-}(m+1,0)-w_{n}^{-}(m,0))}_{\geq0,~~\text{by }\eqref{eq:iso-B}}\geq0.
\end{align*}
For $m\geq1$ and $k\in\{1,\dots,b-1\}$ it holds
\begin{align*}
& \phantomeq w_{n+1}^{-}(m+1,k)-w_{n+1}^{-}(m,k)\\
& =\left[\mu+\frac{\lambda}{\alpha}\cdot w_{n}^{-}(m+2,k)+\frac{\mu}{\alpha}\cdot w_{n}^{-}(m,k-1)+\frac{\nu}{\alpha}\cdot w_{n}^{-}(m+1,k+1)\right.\\
& \phantomeq\left.+\frac{\gamma\cdot k}{\alpha}\cdot w_{n}^{-}(m+1,k-1)+\frac{\gamma\cdot(b-k)}{\alpha}\cdot w_{n}^{-}(m+1,k)\right]\\
& \phantomeq-\left[\mu+\frac{\lambda}{\alpha}\cdot w_{n}^{-}(m+1,k)+\frac{\mu}{\alpha}\cdot w_{n}^{-}(m-1,k-1)+\frac{\nu}{\alpha}\cdot w_{n}^{-}(m,k+1)\right.\\
& \phantomeq+\left.\frac{\gamma\cdot k}{\alpha}\cdot w_{n}^{-}(m,k-1)+\frac{\gamma\cdot(b-k)}{\alpha}\cdot w_{n}^{-}(m,k)\right]\\
& =\frac{\lambda}{\alpha}\cdot\underbrace{(w_{n}^{-}(m+2,k)-w_{n}^{-}(m+1,k))}_{\geq0,~~\text{by }\eqref{eq:iso-B}}+\frac{\mu}{\alpha}\cdot\underbrace{(w_{n}^{-}(m,k-1)-w_{n}^{-}(m-1,k-1))}_{\geq0,~~\text{by }\eqref{eq:iso-B}}\\
& \phantomeq+\frac{\nu}{\alpha}\cdot\underbrace{(w_{n}^{-}(m+1,k+1)-w_{n}^{-}(m,k+1))}_{\geq0,~~\text{by }\eqref{eq:iso-B}}+\frac{\gamma\cdot k}{\alpha}\cdot\underbrace{(w_{n}^{-}(m+1,k-1)-w_{n}^{-}(m,k-1))}_{\geq0,~~\text{by }\eqref{eq:iso-B}}\\
& \phantomeq+\frac{\gamma\cdot(b-k)}{\alpha}\cdot\underbrace{(w_{n}^{-}(m+1,k)-w_{n}^{-}(m,k))}_{\geq0,~~\text{by }\eqref{eq:iso-B}}\geq0.
\end{align*}
For $m\geq1$ and $k=b$ it holds
\begin{align*}
& \phantomeq w_{n+1}^{-}(m+1,b)-w_{n+1}^{-}(m,b)\\
& =\left[\mu+\frac{\lambda}{\alpha}\cdot w_{n}^{-}(m+2,b)+\frac{\mu}{\alpha}\cdot w_{n}^{-}(m,b-1)+\frac{\gamma\cdot b}{\alpha}\cdot w_{n}^{-}(m+1,b-1)+\frac{\nu}{\alpha}\cdot w_{n}^{-}(m+1,b)\right]\\
& \phantomeq-\left[\mu+\frac{\lambda}{\alpha}\cdot w_{n}^{-}(m+1,b)+\frac{\mu}{\alpha}\cdot w_{n}^{-}(m-1,b-1)+\frac{\gamma\cdot b}{\alpha}\cdot w_{n}^{-}(m,b-1)+\frac{\nu}{\alpha}\cdot w_{n}^{-}(m,b)\right]\\
& =\frac{\lambda}{\alpha}\cdot\underbrace{(w_{n}^{-}(m+2,b)-w_{n}^{-}(m+1,b))}_{\geq0,~~\text{by }\eqref{eq:iso-B}}+\underbrace{\frac{\mu}{\alpha}\cdot(w_{n}^{-}(m,b-1)-w_{n}^{-}(m-1,b-1))}_{\geq0,~~\text{by }\eqref{eq:iso-B}}\\
& \phantomeq+\frac{\gamma\cdot b}{\alpha}\cdot\underbrace{(w_{n}^{-}(m+1,b-1)-w_{n}^{-}(m,b-1))}_{\geq0,~~\text{by }\eqref{eq:iso-B}}+\frac{\nu}{\alpha}\cdot\underbrace{(w_{n}^{-}(m+1,b)-w_{n}^{-}(m,b))}_{\geq0,~~\text{by }\eqref{eq:iso-B}}\geq0.
\end{align*}
\noindent$\blacktriangleright$ Third, we check \eqref{eq:iso-C}.
For $m=0$ and $k=0$ it holds 
\begin{align*}
& \phantomeq w_{n+1}^{-}(1,0)-w_{n+1}^{-}(0,0)\\
& =\left[\frac{\nu}{\alpha}\cdot w_{n}^{-}(1,1)+\frac{\lambda+\mu+\gamma\cdot b}{\alpha}\cdot w_{n}^{-}(1,0)\right]-\left[\frac{\nu}{\alpha}\cdot w_{n}^{-}(0,1)+\frac{\lambda+\mu+\gamma\cdot b}{\alpha}\cdot w_{n}^{-}(0,0)\right]\\
& =\frac{\nu}{\alpha}\cdot\underbrace{(w_{n}^{-}(1,1)-w_{n}^{-}(0,1))}_{\leq\alpha,~~\text{by }\eqref{eq:iso-C}}+\frac{\lambda+\mu+\gamma\cdot b}{\alpha}\cdot\underbrace{(w_{n}^{-}(1,0)-w_{n}^{-}(0,0))}_{\leq\alpha,~~\text{by }\eqref{eq:iso-C}}\leq\alpha.
\end{align*}
For $m=0$ and $k\in\{1,\dots,b-1\}$ it holds
\begin{align*}
& \phantomeq w_{n+1}^{-}(1,k)-w_{n+1}^{-}(0,k)\\
& =\left[\mu+\frac{\lambda}{\alpha}\cdot w_{n}^{-}(2,k)+\frac{\mu}{\alpha}\cdot w_{n}^{-}(0,k-1)+\frac{\nu}{\alpha}\cdot w_{n}^{-}(1,k+1)+\frac{\gamma\cdot k}{\alpha}\cdot w_{n}^{-}(1,k-1)+\frac{\gamma\cdot(b-k)}{\alpha}\cdot w_{n}^{-}(1,k)\right]\\
& \phantomeq-\left[0+\frac{\lambda}{\alpha}\cdot w_{n}^{-}(1,k)+\frac{\nu}{\alpha}\cdot w_{n}^{-}(0,k+1)+\frac{\gamma\cdot k}{\alpha}\cdot w_{n}^{-}(0,k-1)+\frac{\mu+\gamma\cdot(b-k)}{\alpha}\cdot w_{n}^{-}(0,k)\right]\\
& =\frac{\lambda}{\alpha}\cdot\underbrace{(w_{n}^{-}(2,k)-w_{n}^{-}(1,k))}_{\leq\alpha,~~\text{by }\eqref{eq:iso-C}}+\frac{\nu}{\alpha}\cdot\underbrace{(w_{n}^{-}(1,k+1)-w_{n}^{-}(0,k+1))}_{\leq\alpha,~~\text{by }\eqref{eq:iso-C}}\\
& \phantomeq+\frac{\gamma\cdot k}{\alpha}\cdot\underbrace{(w_{n}^{-}(1,k-1)-w_{n}^{-}(0,k-1))}_{\leq\alpha,~~\text{by }\eqref{eq:iso-C}}+\frac{\gamma\cdot(b-k)}{\alpha}\cdot\underbrace{(w_{n}^{-}(1,k)-w_{n}^{-}(0,k))}_{\leq\alpha,~~\text{by }\eqref{eq:iso-C}}\\
& \phantomeq+\underbrace{\mu-\frac{\mu}{\alpha}\cdot\underbrace{(w_{n}^{-}(0,k)-w_{n}^{-}(0,k-1))}_{\in[0,\alpha],~~\text{by }\eqref{eq:iso-A},\ \eqref{eq:iso-D}}}_{\leq\mu}\leq\alpha.
\end{align*}
For $m=0$ and $k=b$ it holds
\begin{align*}
& \phantomeq w_{n+1}^{-}(1,b)-w_{n+1}^{-}(0,b)\\
& =\left[\mu+\frac{\lambda}{\alpha}\cdot w_{n}^{-}(2,b)+\frac{\mu}{\alpha}\cdot w_{n}^{-}(0,b-1)+\frac{\gamma\cdot b}{\alpha}\cdot w_{n}^{-}(1,b-1)+\frac{\nu}{\alpha}\cdot w_{n}^{-}(1,b)\right]\\
& \phantomeq-\left[0+\frac{\lambda}{\alpha}\cdot w_{n}^{-}(1,b)+\frac{\gamma\cdot b}{\alpha}\cdot w_{n}^{-}(0,b-1)+\frac{\mu+\nu}{\alpha}\cdot w_{n}^{-}(0,b)\right]\\
& =\frac{\lambda}{\alpha}\cdot\underbrace{(w_{n}^{-}(2,b)-w_{n}^{-}(1,b))}_{\leq\alpha,~~\text{by }\eqref{eq:iso-C}}+\frac{\nu}{\alpha}\cdot\underbrace{(w_{n}^{-}(1,b)-w_{n}^{-}(0,b))}_{\leq\alpha,~~\text{by }\eqref{eq:iso-C}}\\
& \phantomeq+\frac{\gamma\cdot b}{\alpha}\cdot\underbrace{(w_{n}^{-}(1,b-1)-w_{n}^{-}(0,b-1))}_{\leq\alpha,~~\text{by }\eqref{eq:iso-C}}+\underbrace{\mu-\frac{\mu}{\alpha}\cdot\underbrace{(w_{n}^{-}(0,b)-w_{n}^{-}(0,b-1))}_{\in[0,\alpha],~~\text{by }\eqref{eq:iso-A},\ \eqref{eq:iso-D}}}_{\leq\mu}\leq\alpha.
\end{align*}
\noindent
For $m\geq1$ and $k=0$ it holds
\begin{align*}
& \phantomeq w_{n+1}^{-}(m+1,0)-w_{n+1}^{-}(m,0)\\
& =\left[\frac{\nu}{\alpha}\cdot w_{n}^{-}(m+1,1)+\frac{\lambda+\mu+\gamma\cdot b}{\alpha}\cdot w_{n}^{-}(m+1,0)\right]-\left[\frac{\nu}{\alpha}\cdot w_{n}^{-}(m,1)+\frac{\lambda+\mu+\gamma\cdot b}{\alpha}\cdot w_{n}^{-}(m,0)\right]\\
& =\frac{\nu}{\alpha}\cdot \underbrace{(w_{n}^{-}(m+1,1)-w_{n}^{-}(m,1))}_{\leq\alpha,~~\text{by }\eqref{eq:iso-C}}+\frac{\lambda+\mu+\gamma\cdot b}{\alpha}\cdot\underbrace{(w_{n}^{-}(m+1,0)-w_{n}^{-}(m,0))}_{\leq\alpha,~~\text{by }\eqref{eq:iso-C}}\leq\alpha.
\end{align*}
For $m\geq1$ and $k\in\{1,\dots,b-1\}$ it holds
\begin{align*}
& \phantomeq w_{n+1}^{-}(m+1,k)-w_{n+1}^{-}(m,k)\\
& =\left[\mu+\frac{\lambda}{\alpha}\cdot w_{n}^{-}(m+2,k)+\frac{\mu}{\alpha}\cdot w_{n}^{-}(m,k-1)+\frac{\nu}{\alpha}\cdot w_{n}^{-}(m+1,k+1)\right.\\
& \phantomeq\left.+\frac{\gamma\cdot k}{\alpha}\cdot w_{n}^{-}(m+1,k-1)+\frac{\gamma\cdot(b-k)}{\alpha}\cdot w_{n}^{-}(m+1,k)\right]\\
& \phantomeq-\left[\mu+\frac{\lambda}{\alpha}\cdot w_{n}^{-}(m+1,k)+\frac{\mu}{\alpha}\cdot w_{n}^{-}(m-1,k-1)+\frac{\nu}{\alpha}\cdot w_{n}^{-}(m,k+1)\right.\\
& \phantomeq\left.+\frac{\gamma\cdot k}{\alpha}\cdot w_{n}^{-}(m,k-1)+\frac{\gamma\cdot(b-k)}{\alpha}\cdot w_{n}^{-}(m,k)\right]\\
& =\frac{\lambda}{\alpha}\cdot\underbrace{(w_{n}^{-}(m+2,k)-w_{n}^{-}(m+1,k))}_{\leq\alpha,~~\text{by }\eqref{eq:iso-C}}+\frac{\mu}{\alpha}\cdot\underbrace{(w_{n}^{-}(m,k-1)-w_{n}^{-}(m-1,k-1))}_{\leq\alpha,~~\text{by }\eqref{eq:iso-C}}\\
& \phantomeq+\frac{\nu}{\alpha}\cdot\underbrace{(w_{n}^{-}(m+1,k+1)-w_{n}^{-}(m,k+1))}_{\leq\alpha,~~\text{by }\eqref{eq:iso-C}}+\frac{\gamma\cdot k}{\alpha}\cdot\underbrace{(w_{n}^{-}(m+1,k-1)-w_{n}^{-}(m,k-1))}_{\leq\alpha,~~\text{by }\eqref{eq:iso-C}}\\
& \phantomeq+\frac{\gamma\cdot(b-k)}{\alpha}\cdot\underbrace{(w_{n}^{-}(m+1,k)-w_{n}^{-}(m,k))}_{\leq\alpha,~~\text{by }\eqref{eq:iso-C}}\leq\alpha.
\end{align*}
For $m\geq1$ and $k=b$ it holds
\begin{align*}
& \phantomeq w_{n+1}^{-}(m+1,b)-w_{n+1}^{-}(m,b)\\
& =\left[\mu+\frac{\lambda}{\alpha}\cdot w_{n}^{-}(m+2,b)+\frac{\mu}{\alpha}\cdot w_{n}^{-}(m,b-1)+\frac{\gamma\cdot b}{\alpha}\cdot w_{n}^{-}(m+1,b-1)+\frac{\nu}{\alpha}\cdot w_{n}^{-}(m+1,b)\right]\\
& \phantomeq-\left[\mu+\frac{\lambda}{\alpha}\cdot w_{n}^{-}(m+1,b)+\frac{\mu}{\alpha}\cdot w_{n}^{-}(m-1,b-1)+\frac{\gamma\cdot b}{\alpha}\cdot w_{n}^{-}(m,b-1)+\frac{\nu}{\alpha}\cdot w_{n}^{-}(m,b)\right]\\
& =\frac{\lambda}{\alpha}\cdot\underbrace{(w_{n}^{-}(m+2,b)-w_{n}^{-}(m+1,b))}_{\leq\alpha,~~\text{by }\eqref{eq:iso-C}}+\frac{\mu}{\alpha}\cdot(w_{n}^{-}(m,b-1)-w_{n}^{-}(m-1,b-1))\\
& \phantomeq+\frac{\gamma\cdot b}{\alpha}\cdot\underbrace{(w_{n}^{-}(m+1,b-1)-w_{n}^{-}(m,b-1))}_{\leq\alpha,~~\text{by }\eqref{eq:iso-C}}+\frac{\nu}{\alpha}\cdot\underbrace{(w_{n}^{-}(m+1,b)-w_{n}^{-}(m,b))}_{\leq\alpha,~~\text{by }\eqref{eq:iso-C}}\leq\alpha.
\end{align*}
\noindent$\blacktriangleright$ Fourth, we check \eqref{eq:iso-D}.
For $m=0$ and $k=1$ it holds 
\begin{align*}
& \phantomeq w_{n+1}^{-}(0,1)-w_{n+1}^{-}(0,0)\\
& =\left[\frac{\lambda}{\alpha}\cdot w_{n}^{-}(1,1)+\frac{\nu}{\alpha}\cdot w_{n}^{-}(0,2)+\frac{\gamma}{\alpha}\cdot w_{n}^{-}(0,0)+\frac{\mu+\gamma\cdot(b-1)}{\alpha}\cdot w_{n}^{-}(0,1)\right]\\
& \phantomeq-\left[\frac{\nu}{\alpha}\cdot w_{n}^{-}(0,1)+\frac{\lambda+\mu+\gamma\cdot b}{\alpha}w_{n}^{-}(0,0)\right]\\
& =\frac{\lambda}{\alpha}\cdot\underbrace{(w_{n}^{-}(1,1)-w_{n}^{-}(0,1))}_{\leq\alpha,~~\text{by }\eqref{eq:iso-C}}+\frac{\lambda}{\alpha}\cdot\underbrace{(w_{n}^{-}(0,1)-w_{n}^{-}(0,0))}_{\leq\alpha,~~\text{by }\eqref{eq:iso-D}}+\frac{\nu}{\alpha}\cdot\underbrace{(w_{n}^{-}(0,2)-w_{n}^{-}(0,1))}_{\leq\alpha,~~\text{by }\eqref{eq:iso-D}}\\
& \phantomeq+\frac{\mu+\gamma\cdot(b-1)}{\alpha}\cdot\underbrace{(w_{n}^{-}(0,1)-w_{n}^{-}(0,0))}_{\leq\alpha,~~\text{by }\eqref{eq:iso-D}}\leq(2\cdot \lambda+\mu+\nu+\gamma\cdot (b-1))\underset{(\lambda\leq\gamma)}{\leq}\alpha.
\end{align*}
For $m=0$ and $k\in\{2,\dots,b-1\}$ it holds
\begin{align*}
& \phantomeq w_{n+1}^{-}(0,k)-w_{n+1}^{-}(0,k-1)\\
& =\left[\frac{\lambda}{\alpha}\cdot w_{n}^{-}(1,k)+\frac{\nu}{\alpha}\cdot w_{n}^{-}(0,k+1)+\frac{\gamma\cdot k}{\alpha}\cdot w_{n}^{-}(0,k-1)+\frac{\mu+\gamma\cdot(b-k)}{\alpha}\cdot w_{n}^{-}(0,k)\right]\\
& \phantomeq-\left[\frac{\lambda}{\alpha}\cdot w_{n}^{-}(1,k-1)+\frac{\nu}{\alpha}\cdot w_{n}^{-}(0,k)+\frac{\gamma\cdot(k-1)}{\alpha}\cdot w_{n}^{-}(0,k-2)+\frac{\mu+\gamma\cdot(b-k+1)}{\alpha}\cdot w_{n}^{-}(0,k-1)\right]\\
& =\frac{\lambda}{\alpha}\cdot\underbrace{(w_{n}^{-}(1,k)-w_{n}^{-}(1,k-1))}_{\leq\alpha,~~\text{by }\eqref{eq:iso-D}}+\frac{\nu}{\alpha}\cdot\underbrace{(w_{n}^{-}(0,k+1)-w_{n}^{-}(0,k))}_{\leq\alpha,~~\text{by }\eqref{eq:iso-D}}\\
& \phantomeq+\frac{\gamma\cdot(k-1)}{\alpha}\cdot\underbrace{(w_{n}^{-}(0,k-1)-w_{n}^{-}(0,k-2))}_{\leq\alpha,~~\text{by }\eqref{eq:iso-D}}+\frac{\mu+\gamma\cdot(b-k)}{\alpha}\cdot\underbrace{(w_{n}^{-}(0,k)-w_{n}^{-}(0,k-1))}_{\leq\alpha,~~\text{by }\eqref{eq:iso-D}}\leq\alpha-\gamma.
\end{align*}
\noindent
For $m=0$ and $k=b$ it holds
\begin{align*}
& \phantomeq w_{n+1}^{-}(0,b)-w_{n+1}^{-}(0,b-1)\\
& =\left[\frac{\lambda}{\alpha}\cdot w_{n}^{-}(1,b)+\frac{\gamma\cdot b}{\alpha}\cdot w_{n}^{-}(0,b-1)+\frac{\mu+\nu}{\alpha}\cdot w_{n}^{-}(0,b)\right]\\
& \phantomeq-\left[\frac{\lambda}{\alpha}\cdot w_{n}^{-}(1,b-1)+\frac{\nu}{\alpha}\cdot w_{n}^{-}(0,b)+\frac{\gamma\cdot(b-1)}{\alpha}\cdot w_{n}^{-}(0,b-2)+\frac{\mu+\gamma}{\alpha}\cdot w_{n}^{-}(0,b-1)\right]\\
& =\frac{\lambda}{\alpha}\cdot\underbrace{(w_{n}^{-}(1,b)-w_{n}^{-}(1,b-1))}_{\leq\alpha,~~\text{by }\eqref{eq:iso-D}}+\frac{\gamma\cdot(b-1)}{\alpha}\cdot\underbrace{(w_{n}^{-}(0,b-1)-w_{n}^{-}(0,b-2))}_{\leq\alpha,~~\text{by }\eqref{eq:iso-D}}\\
& \phantomeq+\frac{\mu}{\alpha}\cdot\underbrace{(w_{n}^{-}(0,b)-w_{n}^{-}(0,b-1))}_{\leq\alpha,~~\text{by }\eqref{eq:iso-D}}\leq\alpha-\nu-\gamma.
\end{align*}
For $m\geq1$ and $k=1$ it holds
\begin{align*}
& \phantomeq w_{n+1}^{-}(m,1)-w_{n+1}^{-}(m,0)\\
& =\left[\mu+\frac{\lambda}{\alpha}\cdot w_{n}^{-}(m+1,1)+\frac{\mu}{\alpha}\cdot w_{n}^{-}(m-1,0)+\frac{\nu}{\alpha}\cdot w_{n}^{-}(m,2)+\frac{\gamma}{\alpha}\cdot w_{n}^{-}(m,0)+\frac{\gamma\cdot(b-1)}{\alpha}\cdot w_{n}^{-}(m,1)\right]\\
& \phantomeq-\left[0+\frac{\nu}{\alpha}\cdot w_{n}^{-}(m,1)+\frac{\lambda+\mu+\gamma\cdot b}{\alpha}\cdot w_{n}^{-}(m,0)\right]\\
& =\frac{\lambda}{\alpha}\cdot\underbrace{(w_{n}^{-}(m+1,1)-w_{n}^{-}(m,1))}_{\leq\alpha,~~\text{by }\eqref{eq:iso-C}}+\frac{\lambda}{\alpha}\cdot\underbrace{(w_{n}^{-}(m,1)-w_{n}^{-}(m,0))}_{\leq\alpha,~~\text{by }\eqref{eq:iso-D}}+\frac{\nu}{\alpha}\cdot\underbrace{(w_{n}^{-}(m,2)-w_{n}^{-}(m,1))}_{\leq\alpha,~~\text{by }\eqref{eq:iso-D}}\\
& \phantomeq+\frac{\gamma\cdot(b-1)}{\alpha}\cdot\underbrace{(w_{n}^{-}(m,1)-w_{n}^{-}(m,0))}_{\leq\alpha,~~\text{by }\eqref{eq:iso-D}}+\underbrace{\mu-\frac{\mu}{\alpha}\cdot\underbrace{(w_{n}^{-}(m,0)-w_{n}^{-}(m-1,0))}_{\in[0,\alpha],~~\text{by }\eqref{eq:iso-B},\ \eqref{eq:iso-C}}}_{\leq\mu}\\
& \phantomeq\leq2\cdot\lambda+\nu+\mu+\gamma\cdot (b-1)\underset{(\lambda\leq\gamma)}{\leq}\alpha.
\end{align*}
For $m\geq1$ and $k\in\{2,\dots,b-1\}$ it holds
\begin{align*}
& \phantomeq w_{n+1}^{-}(m,k)-w_{n+1}^{-}(m,k-1)\\
& =\left[\mu+\frac{\lambda}{\alpha}\cdot w_{n}^{-}(m+1,k)+\frac{\mu}{\alpha}\cdot w_{n}^{-}(m-1,k-1)+\frac{\nu}{\alpha}\cdot w_{n}^{-}(m,k+1)\right.\\
& \phantomeq\left.+\frac{\gamma\cdot k}{\alpha}\cdot w_{n}^{-}(m,k-1)+\frac{\gamma\cdot(b-k)}{\alpha}\cdot w_{n}^{-}(m,k)\right]\\
& \phantomeq-\left[\mu+\frac{\lambda}{\alpha}\cdot w_{n}^{-}(m+1,k-1)+\frac{\mu}{\alpha}\cdot w_{n}^{-}(m-1,k-2)+\frac{\nu}{\alpha}\cdot w_{n}^{-}(m,k)\right.\\
& \phantomeq\left.+\frac{\gamma\cdot(k-1)}{\alpha}\cdot w_{n}^{-}(m,k-2)+\frac{\gamma\cdot(b-k+1)}{\alpha}\cdot w_{n}^{-}(m,k-1)\right]\\
& =\frac{\lambda}{\alpha}\cdot\underbrace{(w_{n}^{-}(m+1,k)-w_{n}^{-}(m+1,k-1))}_{\leq\alpha,~~\text{by }\eqref{eq:iso-D}}+\frac{\mu}{\alpha}\cdot\underbrace{(w_{n}^{-}(m-1,k-1)-w_{n}^{-}(m-1,k-2))}_{\leq\alpha,~~\text{by }\eqref{eq:iso-D}}\\
& \phantomeq+\frac{\nu}{\alpha}\cdot\underbrace{(w_{n}^{-}(m,k+1)-w_{n}^{-}(m,k))}_{\leq\alpha,~~\text{by }\eqref{eq:iso-D}}+\frac{\gamma\cdot(k-1)}{\alpha}\cdot\underbrace{(w_{n}^{-}(m,k-1)-w_{n}^{-}(m,k-2))}_{\leq\alpha,~~\text{by }\eqref{eq:iso-D}}\\
& \phantomeq+\frac{\gamma\cdot(b-k)}{\alpha}\cdot\underbrace{(w_{n}^{-}(m,k)-w_{n}^{-}(m,k-1))}_{\leq\alpha,~~\text{by }\eqref{eq:iso-D}}\leq\alpha-\gamma.
\end{align*}
For $m\geq1$ and $k=b$ it holds
\begin{align*}
& \phantomeq w_{n+1}^{-}(m,b)-w_{n+1}^{-}(m,b-1)\\
& =\left[\mu+\frac{\lambda}{\alpha}\cdot w_{n}^{-}(m+1,b)+\frac{\mu}{\alpha}\cdot w_{n}^{-}(m-1,b-1)+\frac{\gamma\cdot b}{\alpha}\cdot w_{n}^{-}(m,b-1)+\frac{\nu}{\alpha}\cdot w_{n}^{-}(m,b)\right]\\
& \phantomeq-\left[\mu+\frac{\lambda}{\alpha}\cdot w_{n}^{-}(m+1,b-1)+\frac{\mu}{\alpha}\cdot w_{n}^{-}(m-1,b-2)+\frac{\nu}{\alpha}\cdot w_{n}^{-}(m,b)\right.\\
& \phantomeq\left.+\frac{\gamma\cdot(b-1)}{\alpha}\cdot w_{n}^{-}(m,b-2)+\frac{\gamma}{\alpha}\cdot w_{n}^{-}(m,b-1)\right]\\
& =\frac{\lambda}{\alpha}\underbrace{\cdot(w_{n}^{-}(m+1,b)-w_{n}^{-}(m+1,b-1))}_{\leq\alpha,~~\text{by }\eqref{eq:iso-D}}+\frac{\mu}{\alpha}\cdot\underbrace{(w_{n}^{-}(m-1,b-1)-w_{n}^{-}(m-1,b-2))}_{\leq\alpha,~~\text{by }\eqref{eq:iso-D}}\\
& \phantomeq+\frac{\gamma\cdot(b-1)}{\alpha}\underbrace{\cdot(w_{n}^{-}(m,b-1)-w_{n}^{-}(m,b-2))}_{\leq\alpha,~~\text{by }\eqref{eq:iso-D}}\leq(\lambda+\mu+\gamma\cdot(b-1))=\alpha-\nu-\gamma.
\end{align*}


\noindent \textbf{(b)} 
We show by induction isotonicity in both directions, that the increase
is bounded, and that $w_{n}^{o}(m,k)$ is concave in $m$ for fixed $k$, this means that for all $n\in\mathbb{N}$ the following holds
\begin{align}
w_{n}^{o}(m,k)-w_{n}^{o}(m,k-1)\geq0, &  & \forall k\in\{1,\dots,b\},\ m\in\mathbb{N}_{0},\label{eq:isooA}\\
w_{n}^{o}(m+1,k)-w_{n}^{o}(m,k)\geq0, &  & \forall k\in\{0,1,\dots,b\},\ m\in\mathbb{N}_{0},\label{eq:isooB}\\
w_{n}^{o}(m+1,k)-w_{n}^{o}(m,k)\leq\alpha, &  & \forall k\in\{0,1,\dots,b\},\ m\in\mathbb{N}_{0},\label{eq:isooC}\\
w_{n}^{o}(m,k)-w_{n}^{o}(m,k-1)\leq\alpha, &  & \forall k\in\{1,\dots,b\},\ m\in\mathbb{N}_{0},\label{eq:isooD}\\
w_{n}^{o}(m+1,k)-2\cdot w_{n}^{o}(m,k)+w_{n}^{o}(m-1,k)\leq0, &  & \forall k\in\{0,1,\dots,b\},\ m\in\mathbb{N}.\hphantom{_{0}}\label{eq:isooE}
\end{align}
For $n=1$ we have $w_{1}^{o}(m,k)=r(m,k)=\mu\cdot1_{\left\{ m>0\right\} }\cdot1_{\left\{ k>0\right\}}$ for all $k\in\{0,1,\dots,b\},\ m\in\mathbb{N}_{0},$
so \eqref{eq:isooA}--\eqref{eq:isooD} are trivially true.\\
Let $n\in\mathbb{N}$ such that \eqref{eq:isooA}--\eqref{eq:isooE} hold.
We shall verify these properties for $n$ replaced by $n+1$. In any case we
exploit again $w_{n+1}^{o}=r+R^{o}\cdot w_{n}^{o}$, $n\geq1$.\\
$\blacktriangleright$ First, we check \eqref{eq:isooA}.
For $m=0$ and $k=1$ it holds
\begin{align*}
& \phantomeq w_{n+1}^{o}(0,1)-w_{n+1}^{o}(0,0)\\
& =\left[\frac{\lambda}{\alpha}\cdot w_{n}^{o}(1,1)+\frac{\nu}{\alpha}\cdot w_{n}^{o}(0,2)+\frac{\gamma}{\alpha}\cdot w_{n}^{o}(0,0)+\frac{\mu+\gamma\cdot(b-1)}{\alpha}\cdot w_{n}^{o}(0,1)\right]\\
& \phantomeq-\left[\frac{\nu}{\alpha}\cdot w_{n}^{o}(0,1)+\frac{\lambda+\mu+\gamma\cdot b}{\alpha}\cdot w_{n}^{o}(0,0)\right]\\
& =\frac{\lambda}{\alpha}\cdot\underbrace{(w_{n}^{o}(1,1)-w_{n}^{o}(0,1))}_{\geq0,~~\text{by }\eqref{eq:isooB}}+\frac{\lambda}{\alpha}\cdot\underbrace{(w_{n}^{o}(0,1)-w_{n}^{o}(0,0))}_{\geq0,~~\text{by }\eqref{eq:isooA}}+\frac{\nu}{\alpha}\cdot\underbrace{(w_{n}^{o}(0,2)-w_{n}^{o}(0,1))}_{\geq0,~~\text{by }\eqref{eq:isooA}}\\
& \phantomeq+\frac{\mu+\gamma\cdot(b-1)}{\alpha}\cdot\underbrace{(w_{n}^{o}(0,1)-w_{n}^{o}(0,0))}_{\geq0,~~\text{by }\eqref{eq:isooA}}\geq0.
\end{align*}
For $m=0$ and $k\in\{2,\dots,b-1\}$ it holds
\begin{align*}
& \phantomeq w_{n+1}^{o}(0,k)-w_{n+1}^{o}(0,k-1)\\
& =\left[\frac{\lambda}{\alpha}\cdot w_{n}^{o}(1,k)+\frac{\nu}{\alpha}\cdot w_{n}^{o}(0,k+1)+\frac{\gamma\cdot k}{\alpha}\cdot w_{n}^{o}(0,k-1)+\frac{\mu+\gamma\cdot(b-k)}{\alpha}\cdot w_{n}^{o}(0,k)\right]\\
& \phantomeq-\left[\frac{\lambda}{\alpha}\cdot w_{n}^{o}(1,k-1)+\frac{\nu}{\alpha}\cdot w_{n}^{o}(0,k)+\frac{\gamma\cdot(k-1)}{\alpha}\cdot w_{n}^{o}(0,k-2)
+\frac{\mu+\gamma\cdot(b-k+1)}{\alpha}\cdot w_{n}^{o}(0,k-1)\right]\\
& =\frac{\lambda}{\alpha}\cdot\underbrace{(w_{n}^{o}(1,k)-w_{n}^{o}(1,k-1))}_{\geq0,~~\text{by }\eqref{eq:isooA}}+\frac{\nu}{\alpha}\cdot\underbrace{(w_{n}^{o}(0,k+1)-w_{n}^{o}(0,k))}_{\geq0,~~\text{by }\eqref{eq:isooA}}\\
& \phantomeq+\frac{\gamma\cdot(k-1)}{\alpha}\cdot\underbrace{(w_{n}^{o}(0,k-1)-w_{n}^{o}(0,k-2))}_{\geq0,~~\text{by }\eqref{eq:isooA}}
+\frac{\mu+\gamma\cdot(b-k)}{\alpha}\cdot\underbrace{(w_{n}^{o}(0,k)-w_{n}^{o}(0,k-1))}_{\geq0,~~\text{by }\eqref{eq:isooA}}\geq0.
\end{align*}
For $m=0$ and $k=b$ it holds
\begin{align*}
& \phantomeq w_{n+1}^{o}(0,b)-w_{n+1}^{o}(0,b-1)\\
& =\left[\frac{\lambda}{\alpha}\cdot w_{n}^{o}(1,b)+\frac{\gamma\cdot b}{\alpha}\cdot w_{n}^{o}(0,b-1)+\frac{\mu+\nu}{\alpha}\cdot w_{n}^{o}(0,b)\right]\\
& \phantomeq-\left[\frac{\lambda}{\alpha}\cdot w_{n}^{o}(1,b-1)+\frac{\nu}{\alpha}\cdot w_{n}^{o}(0,b)+\frac{\gamma\cdot(b-1)}{\alpha}\cdot w_{n}^{o}(0,b-2)+\frac{\mu+\gamma}{\alpha}\cdot w_{n}^{o}(0,b-1)\right]\\
& =\frac{\lambda}{\alpha}\cdot\underbrace{(w_{n}^{o}(1,b)-w_{n}^{o}(1,b-1))}_{\geq0,~~\text{by }\eqref{eq:isooA}}+\frac{\gamma\cdot(b-1)}{\alpha}\cdot\underbrace{(w_{n}^{o}(0,b-1)-w_{n}^{o}(0,b-2))}_{\geq0,~~\text{by }\eqref{eq:isooA}}\\
& \phantomeq+\frac{\mu}{\alpha}\cdot\underbrace{(w_{n}^{o}(0,b)-w_{n}^{o}(0,b-1))}_{\geq0,~~\text{by }\eqref{eq:isooA}}\geq0.
\end{align*}

\noindent
For $m\geq1$ and $k=1$ it holds
\begin{align*}
& \phantomeq w_{n+1}^{o}(m,1)-w_{n+1}^{o}(m,0)\\
& =\left[\mu+\frac{\lambda}{\alpha}\cdot w_{n}^{o}(m+1,1)+\frac{\mu}{\alpha}\cdot w_{n}^{o}(m-1,0)+\frac{\nu}{\alpha}\cdot w_{n}^{o}(m,2)+\frac{\gamma\cdot b}{\alpha}\cdot w_{n}^{o}(m,1)\right]\\
& \phantomeq-\left[0+\frac{\nu}{\alpha}\cdot w_{n}^{o}(m,1)+\frac{\lambda+\mu+\gamma\cdot b}{\alpha}\cdot w_{n}^{o}(m,0)\right]\\
& =\frac{\lambda}{\alpha}\cdot\underbrace{(w_{n}^{o}(m+1,1)-w_{n}^{o}(m,1))}_{\geq0,~~\text{by }\eqref{eq:isooB}}+\frac{\lambda}{\alpha}\cdot\underbrace{(w_{n}^{o}(m,1)-w_{n}^{o}(m,0))}_{\geq0,~~\text{by }\eqref{eq:isooA}}+\frac{\nu}{\alpha}\cdot\underbrace{(w_{n}^{o}(m,2)-w_{n}^{o}(m,1))}_{\geq0,~~\text{by }\eqref{eq:isooA}}\\
& \phantomeq+\frac{\gamma\cdot b}{\alpha}\cdot\underbrace{(w_{n}^{o}(m,1)-w_{n}^{o}(m,0))}_{\geq0,~~\text{by }\eqref{eq:isooA}}+\underbrace{\mu-\frac{\mu}{\alpha}\cdot\underbrace{(w_{n}^{o}(m,0)-w_{n}^{o}(m-1,0))}_{\in[0,\alpha],~~\text{by }\eqref{eq:isooB},\ \eqref{eq:isooC}}}_{\geq0}\geq0.
\end{align*}
For $m\geq1$ and $k\in\{2,\dots,b-1\}$ it holds
\begin{align*}
& \phantomeq w_{n+1}^{o}(m,k)-w_{n+1}^{o}(m,k-1)\\
& =\left[\mu+\frac{\lambda}{\alpha}\cdot w_{n}^{o}(m+1,k)+\frac{\mu}{\alpha}\cdot w_{n}^{o}(m-1,k-1)+\frac{\nu}{\alpha}\cdot w_{n}^{o}(m,k+1)\right.\\
& \phantomeq\left.+\frac{\gamma\cdot(k-1)}{\alpha}\cdot w_{n}^{o}(m,k-1)+\frac{\gamma\cdot(b-k+1)}{\alpha}\cdot w_{n}^{o}(m,k)\right]\\
& \phantomeq-\left[\mu+\frac{\lambda}{\alpha}\cdot w_{n}^{o}(m+1,k-1)+\frac{\mu}{\alpha}\cdot w_{n}^{o}(m-1,k-2)+\frac{\nu}{\alpha}\cdot w_{n}^{o}(m,k)\right.\\
& \phantomeq\left.+\frac{\gamma\cdot(k-2)}{\alpha}\cdot w_{n}^{o}(m,k-2)+\frac{\gamma\cdot(b-k+2)}{\alpha}\cdot w_{n}^{o}(m,k-1)\right]\\
& =\frac{\lambda}{\alpha}\cdot\underbrace{(w_{n}^{o}(m+1,k)-w_{n}^{o}(m+1,k-1))}_{\geq0,~~\text{by }\eqref{eq:isooA}}+\frac{\mu}{\alpha}\cdot\underbrace{(w_{n}^{o}(m-1,k-1)-w_{n}^{o}(m-1,k-2))}_{\geq0,~~\text{by }\eqref{eq:isooA}}\\
& \phantomeq+\frac{\nu}{\alpha}\cdot\underbrace{(w_{n}^{o}(m,k+1)-w_{n}^{o}(m,k))}_{\geq0,~~\text{by }\eqref{eq:isooA}}+\frac{\gamma\cdot(k-2)}{\alpha}\cdot\underbrace{(w_{n}^{o}(m,k-1)-w_{n}^{o}(m,k-2))}_{\geq0,~~\text{by }\eqref{eq:isooA}}\\
& \phantomeq+\frac{\gamma\cdot(b-k+1)}{\alpha}\cdot\underbrace{(w_{n}^{o}(m,k)-w_{n}^{o}(m,k-1))}_{\leq\alpha,~~\text{by }\eqref{eq:isooA}}\geq0.
\end{align*}

\noindent
For $m\geq1$ and $k=b$ it holds
\begin{align*}
& \phantomeq w_{n+1}^{o}(m,b)-w_{n+1}^{o}(m,b-1)\\
& =\left[\mu+\frac{\lambda}{\alpha}\cdot w_{n}^{o}(m+1,b)+\frac{\mu}{\alpha}\cdot w_{n}^{o}(m-1,b-1)+\frac{\gamma\cdot(b-1)}{\alpha}\cdot w_{n}^{o}(m,b-1)+\frac{\nu+\gamma}{\alpha}\cdot w_{n}^{o}(m,b)\right]\\
& \phantomeq-\left[\mu+\frac{\lambda}{\alpha}\cdot w_{n}^{o}(m+1,b-1)+\frac{\mu}{\alpha}\cdot w_{n}^{o}(m-1,b-2)+\frac{\nu}{\alpha}\cdot w_{n}^{o}(m,b)\right.\\
& \phantomeq\left.+\frac{\gamma\cdot(b-2)}{\alpha}\cdot w_{n}^{o}(m,b-2)+\frac{\gamma\cdot2}{\alpha}\cdot w_{n}^{o}(m,b-1)\right]\\
& =\frac{\lambda}{\alpha}\cdot\underbrace{(w_{n}^{o}(m+1,b)-w_{n}^{o}(m+1,b-1))}_{\geq0,~~\text{by }\eqref{eq:isooA}}+\frac{\mu}{\alpha}\cdot\underbrace{(w_{n}^{o}(m-1,b-1)-w_{n}^{o}(m-1,b-2))}_{\geq0,~~\text{by }\eqref{eq:isooA}}\\
& \phantomeq+\frac{\gamma\cdot(b-2)}{\alpha}\cdot\underbrace{(w_{n}^{o}(m,b-1)-w_{n}^{o}(m,b-2))}_{\geq0,~~\text{by }\eqref{eq:isooA}}+\frac{\gamma}{\alpha}\cdot\underbrace{(w_{n}^{o}(m,b)-w_{n}^{o}(m,b-1))}_{\geq0,~~\text{by }\eqref{eq:isooA}}\geq0.
\end{align*}
$\blacktriangleright$ Second, we check \eqref{eq:isooB}.
For $m=0$ and $k=0$ it holds 
\begin{align*}
& \phantomeq w_{n+1}^{o}(m+1,0)-w_{n+1}^{o}(m,0)\\
& =\left[\frac{\nu}{\alpha}\cdot w_{n}^{o}(m+1,1)+\frac{\lambda+\mu+\gamma\cdot b}{\alpha}\cdot w_{n}^{o}(m+1,0)\right]-\left[\frac{\nu}{\alpha}\cdot w_{n}^{o}(m,1)+\frac{\lambda+\mu+\gamma\cdot b}{\alpha}\cdot w_{n}^{o}(m,0)\right]\\
& =\frac{\nu}{\alpha}\cdot\underbrace{(w_{n}^{o}(m+1,1)-w_{n}^{o}(m,1))}_{\geq0,~~\text{by }\eqref{eq:isooB}}+\frac{\lambda+\mu+\gamma\cdot b}{\alpha}\cdot\underbrace{(w_{n}^{o}(m+1,0)-w_{n}^{o}(m,0))}_{\geq0,~~\text{by }\eqref{eq:isooB}}\geq0.
\end{align*}
For $m=0$ and $k\in\{1,\dots,b-1\}$ it holds
\begin{align*}
& \phantomeq w_{n+1}^{o}(1,k)-w_{n+1}^{o}(0,k)\\
& =\left[\mu+\frac{\lambda}{\alpha}\cdot w_{n}^{o}(2,k)+\frac{\mu}{\alpha}\cdot w_{n}^{o}(0,k-1)+\frac{\nu}{\alpha}\cdot w_{n}^{o}(1,k+1)\right.\\
& \phantomeq\left.+\frac{\gamma\cdot(k-1)}{\alpha}\cdot w_{n}^{o}(1,k-1)+\frac{\gamma\cdot(b-k+1)}{\alpha}\cdot w_{n}^{o}(1,k)\right]\\
& \phantomeq-\left[0+\frac{\lambda}{\alpha}\cdot w_{n}^{o}(1,k)+\frac{\nu}{\alpha}\cdot w_{n}^{o}(0,k+1)+\frac{\gamma\cdot k}{\alpha}\cdot w_{n}^{o}(0,k-1)+\frac{\mu+\gamma\cdot(b-k)}{\alpha}\cdot w_{n}^{o}(0,k)\right]\\
& =\frac{\lambda}{\alpha}\cdot\underbrace{(w_{n}^{o}(2,k)-w_{n}^{o}(1,k))}_{\geq0,~~\text{by }\eqref{eq:isooB}}+\frac{\nu}{\alpha}\cdot\underbrace{(w_{n}^{o}(1,k+1)-w_{n}^{o}(0,k+1))}_{\geq0,~~\text{by }\eqref{eq:isooB}}\\
& \phantomeq+\frac{\gamma\cdot(k-1)}{\alpha}\cdot\underbrace{(w_{n}^{o}(1,k-1)-w_{n}^{o}(0,k-1))}_{\geq0,~~\text{by }\eqref{eq:isooB}}+\frac{\gamma\cdot(b-k)}{\alpha}\cdot\underbrace{(w_{n}^{o}(1,k)-w_{n}^{o}(0,k))}_{\geq0,~~\text{by }\eqref{eq:isooB}}\\
& \phantomeq+\mu-\frac{\mu}{\alpha}\cdot\underbrace{(w_{n}^{o}(0,k)-w_{n}^{o}(0,k-1))}_{\geq0,~~\text{by }\eqref{eq:isooA}}+\frac{\gamma}{\alpha}\cdot\underbrace{(w_{n}^{o}(1,k)-w_{n}^{o}(0,k))}_{\geq0,~~\text{by }\eqref{eq:isooB}}+\frac{\gamma}{\alpha}\cdot\underbrace{(w_{n}^{o}(0,k)-w_{n}^{o}(0,k-1))}_{\geq0,~~\text{by }\eqref{eq:isooA}}\geq0.
\end{align*}

\noindent
For $m=0$ and $k=b$ it holds
\begin{align*}
& \phantomeq w_{n+1}^{o}(1,b)-w_{n+1}^{o}(0,b)\\
& =\left[\mu+\frac{\lambda}{\alpha}\cdot w_{n}^{o}(2,b)+\frac{\mu}{\alpha}\cdot w_{n}^{o}(0,b-1)+\frac{\gamma\cdot(b-1)}{\alpha}\cdot w_{n}^{o}(1,b-1)+\frac{\nu+\gamma}{\alpha}\cdot w_{n}^{o}(1,b)\right]\\
& \phantomeq-\left[0+\frac{\lambda}{\alpha}\cdot w_{n}^{o}(1,b)+\frac{\gamma\cdot b}{\alpha}\cdot w_{n}^{o}(0,b-1)+\frac{\mu+\nu}{\alpha}\cdot w_{n}^{o}(0,b)\right]\\
& =\frac{\lambda}{\alpha}\cdot\underbrace{(w_{n}^{o}(2,b)-w_{n}^{o}(1,b))}_{\geq0,~~\text{by }\eqref{eq:isooB}}+\frac{\nu}{\alpha}\cdot\underbrace{(w_{n}^{o}(1,b)-w_{n}^{o}(0,b))}_{\geq0,~~\text{by }\eqref{eq:isooB}}\\
& \phantomeq+\frac{\gamma\cdot b}{\alpha}\cdot\underbrace{(w_{n}^{o}(1,b-1)-w_{n}^{o}(0,b-1))}_{\geq0,~~\text{by }\eqref{eq:isooB}}+\underbrace{\mu+\frac{\mu}{\alpha}\cdot\underbrace{(w_{n}^{o}(0,b)-w_{n}^{o}(0,b-1))}_{\in[0,\alpha],~~\text{by }\eqref{eq:isooA},\eqref{eq:isooD}}}_{\geq0}\\
& \phantomeq+\frac{\gamma}{\alpha}\cdot\underbrace{(w_{n}^{o}(1,b)-w_{n}^{o}(1,b-1))}_{\geq0,~~\text{by }\eqref{eq:isooA}}\geq0.
\end{align*}
For $m\geq1$ and $k\in\{1,\dots,b-1\}$ it holds
\begin{align*}
& \phantomeq w_{n+1}^{o}(m+1,k)-w_{n+1}^{o}(m,k)\\
& =\left[\mu+\frac{\lambda}{\alpha}\cdot w_{n}^{o}(m+2,k)+\frac{\mu}{\alpha}\cdot w_{n}^{o}(m,k-1)+\frac{\nu}{\alpha}\cdot w_{n}^{o}(m+1,k+1)\right.\\
& \phantomeq\left.+\frac{\gamma\cdot(k-1)}{\alpha}\cdot w_{n}^{o}(m+1,k-1)+\frac{\gamma\cdot(b-k+1)}{\alpha}\cdot w_{n}^{o}(m+1,k)\right]\\
& \phantomeq-\left[\mu+\frac{\lambda}{\alpha}\cdot w_{n}^{o}(m+1,k)+\frac{\mu}{\alpha}\cdot w_{n}^{o}(m-1,k-1)+\frac{\nu}{\alpha}\cdot w_{n}^{o}(m,k+1)\right.\\
& \phantomeq\left.+\frac{\gamma\cdot(k-1)}{\alpha}\cdot w_{n}^{o}(m,k-1)+\frac{\gamma\cdot(b-k+1)}{\alpha}\cdot w_{n}^{o}(m,k)\right]\\
& =\frac{\lambda}{\alpha}\cdot\underbrace{(w_{n}^{o}(m+2,k)-w_{n}^{o}(m+1,k))}_{\geq0,~~\text{by }\eqref{eq:isooB}}+\frac{\mu}{\alpha}\cdot\underbrace{(w_{n}^{o}(m,k-1)-w_{n}^{o}(m-1,k-1))}_{\geq0,~~\text{by }\eqref{eq:isooB}}\\
& \phantomeq+\frac{\nu}{\alpha}\cdot\underbrace{(w_{n}^{o}(m+1,k+1)-w_{n}^{o}(m,k+1))}_{\geq0,~~\text{by }\eqref{eq:isooB}}+\frac{\gamma\cdot(k-1)}{\alpha}\cdot\underbrace{(w_{n}^{o}(m+1,k-1)-w_{n}^{o}(m,k-1))}_{\geq0,~~\text{by }\eqref{eq:isooB}}\\
& \phantomeq+\frac{\gamma\cdot(b-k+1)}{\alpha}\cdot\underbrace{(w_{n}^{o}(m+1,k)-w_{n}^{o}(m,k))}_{\geq0,~~\text{by }\eqref{eq:isooB}}\geq0.
\end{align*}

\noindent
For $m\geq1$ and $k=b$ it holds
\begin{align*}
& \phantomeq w_{n+1}^{o}(m+1,b)-w_{n+1}^{o}(m,b)\\
& =\left[\mu+\frac{\lambda}{\alpha}\cdot w_{n}^{o}(m+2,b)+\frac{\mu}{\alpha}\cdot w_{n}^{o}(m,b-1)+\frac{\gamma\cdot(b-1)}{\alpha}\cdot w_{n}^{o}(m+1,b-1)+\frac{\nu+\gamma}{\alpha}\cdot w_{n}^{o}(m+1,b)\right]\\
& \phantomeq-\left[\mu+\frac{\lambda}{\alpha}\cdot w_{n}^{o}(m+1,b)+\frac{\mu}{\alpha}\cdot w_{n}^{o}(m-1,b-1)+\frac{\gamma\cdot(b-1)}{\alpha}\cdot w_{n}^{o}(m,b-1)+\frac{\nu+\gamma}{\alpha}\cdot w_{n}^{o}(m,b)\right]\\
& =\frac{\lambda}{\alpha}\cdot\underbrace{(w_{n}^{o}(m+2,b)-w_{n}^{o}(m+1,b))}_{\geq0,~~\text{by }\eqref{eq:isooB}}+\frac{\mu}{\alpha}\cdot\underbrace{(w_{n}^{o}(m,b-1)-w_{n}^{o}(m-1,b-1))}_{\geq0,~~\text{by }\eqref{eq:isooB}}\\
& \phantomeq+\frac{\gamma\cdot(b-1)}{\alpha}\cdot\underbrace{(w_{n}^{o}(m+1,b-1)-w_{n}^{o}(m,b-1))}_{\geq0,~~\text{by }\eqref{eq:isooB}}+\frac{\nu+\gamma}{\alpha}\cdot\underbrace{(w_{n}^{o}(m+1,b)-w_{n}^{o}(m,b))}_{\geq0,~~\text{by }\eqref{eq:isooB}}\geq0.
\end{align*}
$\blacktriangleright$ Third, we check \eqref{eq:isooC}.
For $m\geq0$ and $k=0$ it holds 
\begin{align*}
& \phantomeq w_{n+1}^{o}(m+1,0)-w_{n+1}^{o}(m,0)\\
& =\left[\frac{\nu}{\alpha}\cdot w_{n}^{o}(m+1,1)+\frac{\lambda+\mu+\gamma\cdot b}{\alpha}\cdot w_{n}^{o}(m+1,0)\right]-\left[\frac{\nu}{\alpha}\cdot w_{n}^{o}(m,1)+\frac{\lambda+\mu+\gamma\cdot b}{\alpha}\cdot w_{n}^{o}(m,0)\right]\\
& =\frac{\nu}{\alpha}\cdot\underbrace{(w_{n}^{o}(m+1,1)-w_{n}^{o}(m,1))}_{\leq\alpha,~~\text{by }\eqref{eq:isooC}}+\frac{\lambda+\mu+\gamma\cdot b}{\alpha}\cdot\underbrace{(w_{n}^{o}(m+1,0)-w_{n}^{o}(m,0))}_{\leq\alpha,~~\text{by }\eqref{eq:isooC}}\leq\alpha.
\end{align*}
For $m=0$ and $k\in\{1,\dots,b-1\}$ it holds
\begin{align*}
& \phantomeq w_{n+1}^{o}(1,k)-w_{n+1}^{o}(0,k)\\
& =\left[\mu+\frac{\lambda}{\alpha}\cdot w_{n}^{o}(2,k)+\frac{\mu}{\alpha}\cdot w_{n}^{o}(0,k-1)+\frac{\nu}{\alpha}\cdot w_{n}^{o}(1,k+1)\right.\\
& \phantomeq\left.+\frac{\gamma\cdot(k-1)}{\alpha}\cdot w_{n}^{o}(1,k-1)+\frac{\gamma\cdot(b-k+1)}{\alpha}\cdot w_{n}^{o}(1,k)\right]\\
& \phantomeq-\left[0+\frac{\lambda}{\alpha}\cdot w_{n}^{o}(1,k)+\frac{\nu}{\alpha}\cdot w_{n}^{o}(0,k+1)+\frac{\gamma\cdot k}{\alpha}\cdot w_{n}^{o}(0,k-1)+\frac{\mu+\gamma\cdot(b-k)}{\alpha}\cdot w_{n}^{o}(0,k)\right]\\
& =\frac{\lambda}{\alpha}\cdot\underbrace{(w_{n}^{o}(2,k)-w_{n}^{o}(1,k))}_{\leq\alpha,~~\text{by }\eqref{eq:isooC}}+\frac{\nu}{\alpha}\cdot\underbrace{(w_{n}^{o}(1,k+1)-w_{n}^{o}(0,k+1))}_{\leq\alpha,~~\text{by }\eqref{eq:isooC}}\\
& \phantomeq+\frac{\gamma\cdot(k-1)}{\alpha}\cdot\underbrace{(w_{n}^{o}(1,k-1)-w_{n}^{o}(0,k-1))}_{\leq\alpha,~~\text{by }\eqref{eq:isooC}}+\frac{\gamma\cdot(b-k+1)}{\alpha}\cdot\underbrace{(w_{n}^{o}(1,k)-w_{n}^{o}(0,k))}_{\leq\alpha,~~\text{by }\eqref{eq:isooC}}\\
& \phantomeq+\mu+\underbrace{(\frac{\gamma}{\alpha}-\frac{\mu}{\alpha})}_{=0,\ \text{by }\gamma=\mu}\cdot(w_{n}^{o}(0,k)-w_{n}^{o}(0,k-1))\leq\alpha.
\end{align*}

\noindent
For $m=0$ and $k=b$ it holds
\begin{align*}
& \phantomeq w_{n+1}^{o}(1,b)-w_{n+1}^{o}(0,b)\\
& =\left[\mu+\frac{\lambda}{\alpha}\cdot w_{n}^{o}(2,b)+\frac{\mu}{\alpha}\cdot w_{n}^{o}(0,b-1)+\frac{\gamma\cdot(b-1)}{\alpha}\cdot w_{n}^{o}(1,b-1)+\frac{\nu+\gamma}{\alpha}\cdot w_{n}^{o}(1,b)\right]\\
& \phantomeq-\left[0+\frac{\lambda}{\alpha}\cdot w_{n}^{o}(1,b)+\frac{\gamma\cdot b}{\alpha}\cdot w_{n}^{o}(0,b-1)+\frac{\mu+\nu}{\alpha}\cdot w_{n}^{o}(0,b)\right]\\
& =\frac{\lambda}{\alpha}\cdot\underbrace{(w_{n}^{o}(2,b)-w_{n}^{o}(1,b))}_{\leq\alpha,~~\text{by }\eqref{eq:isooC}}+\frac{\gamma\cdot(b-1)}{\alpha}\cdot\underbrace{(w_{n}^{o}(1,b-1)-w_{n}^{o}(0,b-1))}_{\leq\alpha,~~\text{by }\eqref{eq:isooC}}\\
& \phantomeq+\frac{\nu+\gamma}{\alpha}\cdot\underbrace{(w_{n}^{o}(1,b)-w_{n}^{o}(0,b))}_{\leq\alpha,~~\text{by }\eqref{eq:isooC}}+\mu+\underbrace{(\frac{\gamma}{\alpha}-\frac{\mu}{\alpha})}_{=0,\ \text{by }\gamma=\mu}\cdot(w_{n}^{o}(0,b)-w_{n}^{o}(0,b-1))\leq\alpha.
\end{align*}
For $m\geq1$ and $k\in\{1,\dots,b-1\}$ it holds
\begin{align*}
& \phantomeq w_{n+1}^{o}(m+1,k)-w_{n+1}^{o}(m,k)\\
& =\left[\mu+\frac{\lambda}{\alpha}\cdot w_{n}^{o}(m+2,k)+\frac{\mu}{\alpha}\cdot w_{n}^{o}(m,k-1)+\frac{\nu}{\alpha}\cdot w_{n}^{o}(m+1,k+1)\right.\\
& \phantomeq\left.+\frac{\gamma\cdot(k-1)}{\alpha}\cdot w_{n}^{o}(m+1,k-1)+\frac{\gamma\cdot(b-k+1)}{\alpha}\cdot w_{n}^{o}(m+1,k)\right]\\
& \phantomeq-\left[\mu+\frac{\lambda}{\alpha}\cdot w_{n}^{o}(m+1,k)+\frac{\mu}{\alpha}\cdot w_{n}^{o}(m-1,k-1)+\frac{\nu}{\alpha}\cdot w_{n}^{o}(m,k+1)\right.\\
& \phantomeq\left.+\frac{\gamma\cdot(k-1)}{\alpha}\cdot w_{n}^{o}(m,k-1)+\frac{\gamma\cdot(b-k+1)}{\alpha}\cdot w_{n}^{o}(m,k)\right]\\
& =\frac{\lambda}{\alpha}\cdot\underbrace{(w_{n}^{o}(m+2,k)-w_{n}^{o}(m+1,k))}_{\leq\alpha,~~\text{by }\eqref{eq:isooC}}+\frac{\mu}{\alpha}\cdot\underbrace{(w_{n}^{o}(m,k-1)-w_{n}^{o}(m-1,k-1))}_{\leq\alpha,~~\text{by }\eqref{eq:isooC}}\\
& \phantomeq+\frac{\nu}{\alpha}\cdot\underbrace{(w_{n}^{o}(m+1,k+1)-w_{n}^{o}(m,k+1))}_{\leq\alpha,~~\text{by }\eqref{eq:isooC}}+\frac{\gamma\cdot(k-1)}{\alpha}\cdot\underbrace{(w_{n}^{o}(m+1,k-1)-w_{n}^{o}(m,k-1))}_{\leq\alpha,~~\text{by }\eqref{eq:isooC}}\\
& \phantomeq+\frac{\gamma\cdot(b-k+1)}{\alpha}\cdot\underbrace{(w_{n}^{o}(m+1,k)-w_{n}^{o}(m,k))}_{\leq\alpha,~~\text{by }\eqref{eq:isooC}}\leq\alpha.
\end{align*}

\noindent
For $m\geq1$ and $k=b$ it holds
\begin{align*}
& \phantomeq w_{n+1}^{o}(m+1,b)-w_{n+1}^{o}(m,b)\\
& =\left[\mu+\frac{\lambda}{\alpha}\cdot w_{n}^{o}(m+2,b)+\frac{\mu}{\alpha}\cdot w_{n}^{o}(m,b-1)+\frac{\gamma\cdot(b-1)}{\alpha}\cdot w_{n}^{o}(m+1,b-1)+\frac{\nu+\gamma}{\alpha}\cdot w_{n}^{o}(m+1,b)\right]\\
& \phantomeq-\left[\mu+\frac{\lambda}{\alpha}\cdot w_{n}^{o}(m+1,b)+\frac{\mu}{\alpha}\cdot w_{n}^{o}(m-1,b-1)+\frac{\gamma\cdot(b-1)}{\alpha}\cdot w_{n}^{o}(m,b-1)+\frac{\nu+\gamma}{\alpha}\cdot w_{n}^{o}(m,b)\right]\\
& =\frac{\lambda}{\alpha}\cdot\underbrace{(w_{n}^{o}(m+2,b)-w_{n}^{o}(m+1,b))}_{\leq\alpha,~~\text{by }\eqref{eq:isooC}}+\frac{\mu}{\alpha}\cdot\underbrace{(w_{n}^{o}(m,b-1)-w_{n}^{o}(m-1,b-1))}_{\leq\alpha,~~\text{by }\eqref{eq:isooD}}\\
& \phantomeq+\frac{\gamma\cdot(b-1)}{\alpha}\cdot\underbrace{(w_{n}^{o}(m+1,b-1)-w_{n}^{o}(m,b-1))}_{\leq\alpha,~~\text{by }\eqref{eq:isooC}}+\frac{\nu+\gamma}{\alpha}\cdot\underbrace{(w_{n}^{o}(m+1,b)-w_{n}^{o}(m,b))}_{\leq\alpha,~~\text{by }\eqref{eq:isooC}}\leq\alpha.
\end{align*}
$\blacktriangleright$ Fourth, we check \eqref{eq:isooD}.
For $m=0$ and $k=1$ it holds 
\begin{align*}
& \phantomeq w_{n+1}^{o}(0,1)-w_{n+1}^{o}(0,0)\\
& =\left[\frac{\lambda}{\alpha}\cdot w_{n}^{o}(1,1)+\frac{\nu}{\alpha}\cdot w_{n}^{o}(0,2)+\frac{\gamma}{\alpha}\cdot w_{n}^{o}(0,0)+\frac{\mu+\gamma\cdot(b-1)}{\alpha}\cdot w_{n}^{o}(0,1)\right]\\
& \phantomeq-\left[\frac{\nu}{\alpha}\cdot w_{n}^{o}(0,1)+\frac{\lambda+\mu+\gamma\cdot b}{\alpha}\cdot w_{n}^{o}(0,0)\right]\\
& =\frac{\lambda}{\alpha}v\underbrace{(w_{n}^{o}(1,1)-w_{n}^{o}(0,1))}_{\leq\alpha,~~\text{by }\eqref{eq:isooC}}+\frac{\lambda}{\alpha}\cdot\underbrace{(w_{n}^{o}(0,1)-w_{n}^{o}(0,0))}_{\leq\alpha,~~\text{by }\eqref{eq:isooD}}+\frac{\nu}{\alpha}\cdot\underbrace{(w_{n}^{o}(0,2)-w_{n}^{o}(0,1))}_{\leq\alpha,~~\text{by }\eqref{eq:isooD}}\\
& \phantomeq+\frac{\mu+\gamma\cdot(b-1)}{\alpha}\cdot\underbrace{(w_{n}^{o}(0,1)-w_{n}^{o}(0,0))}_{\leq\alpha,~~\text{by }\eqref{eq:isooD}}\leq(2\lambda+\mu+\nu+\gamma\cdot(b-1))\stackrel{(\lambda\leq\gamma)}{\leq}\alpha.
\end{align*}

\noindent
For $m=0$ and $k\in\{2,\dots,b-1\}$ it holds
\begin{align*}
& \phantomeq w_{n+1}^{o}(0,k)-w_{n+1}^{o}(0,k-1)\\
& =\left[\frac{\lambda}{\alpha}\cdot w_{n}^{o}(1,k)+\frac{\nu}{\alpha}\cdot w_{n}^{o}(0,k+1)+\frac{\gamma\cdot k}{\alpha}\cdot w_{n}^{o}(0,k-1)+\frac{\mu+\gamma\cdot(b-k)}{\alpha}\cdot w_{n}^{o}(0,k)\right]\\
& \phantomeq-\left[\frac{\lambda}{\alpha}\cdot w_{n}^{o}(1,k-1)+\frac{\nu}{\alpha}\cdot w_{n}^{o}(0,k)+\frac{\gamma\cdot(k-1)}{\alpha}\cdot w_{n}^{o}(0,k-2)+\frac{\mu+\gamma\cdot(b-k+1)}{\alpha}\cdot w_{n}^{o}(0,k-1)\right]\\
& =\frac{\lambda}{\alpha}\cdot\underbrace{(w_{n}^{o}(1,k)-w_{n}^{o}(1,k-1))}_{\leq\alpha,~~\text{by }\eqref{eq:isooD}}+\frac{\nu}{\alpha}\cdot\underbrace{(w_{n}^{o}(0,k+1)-w_{n}^{o}(0,k))}_{\leq\alpha,~~\text{by }\eqref{eq:isooD}}\\
& \phantomeq+\frac{\gamma\cdot(k-1)}{\alpha}\cdot\underbrace{(w_{n}^{o}(0,k-1)-w_{n}^{o}(0,k-2))}_{\leq\alpha,~~\text{by }\eqref{eq:isooD}}+\frac{\mu+\gamma\cdot(b-k)}{\alpha}\cdot\underbrace{(w_{n}^{o}(0,k)-w_{n}^{o}(0,k-1))}_{\leq\alpha,~~\text{by }\eqref{eq:isooD}} \leq\alpha-\gamma.
\end{align*}
For $m=0$ and $k=b$ it holds
\begin{align*}
& \phantomeq w_{n+1}^{o}(0,b)-w_{n+1}^{o}(0,b-1)\\
& =\left[\frac{\lambda}{\alpha}\cdot w_{n}^{o}(1,b)+\frac{\gamma\cdot b}{\alpha}\cdot w_{n}^{o}(0,b-1)+\frac{\mu+\nu}{\alpha}\cdot w_{n}^{o}(0,b)\right]\\
& \phantomeq-\left[\frac{\lambda}{\alpha}\cdot w_{n}^{o}(1,b-1)+\frac{\nu}{\alpha}\cdot w_{n}^{o}(0,b)+\frac{\gamma\cdot(b-1)}{\alpha}\cdot w_{n}^{o}(0,b-2)+\frac{\mu+\gamma}{\alpha}\cdot w_{n}^{o}(0,b-1)\right]\\
& =\frac{\lambda}{\alpha}\cdot\underbrace{(w_{n}^{o}(1,b)-w_{n}^{o}(1,b-1))}_{\leq\alpha,~~\text{by }\eqref{eq:isooD}}+\frac{\gamma\cdot(b-1)}{\alpha}\cdot\underbrace{(w_{n}^{o}(0,b-1)-w_{n}^{o}(0,b-2))}_{\leq\alpha,~~\text{by }\eqref{eq:isooD}}\\
&\phantomeq+\frac{\mu}{\alpha}\cdot\underbrace{(w_{n}^{o}(0,b)-w_{n}^{o}(0,b-1))}_{\leq\alpha,~~\text{by }\eqref{eq:isooD}}
\leq\alpha-\nu-\gamma.
\end{align*}
For $m\geq1$ and $k=1$ it holds
\begin{align*}
& \phantomeq w_{n+1}^{o}(m,1)-w_{n+1}^{o}(m,0)\\
& =\left[\mu+\frac{\lambda}{\alpha}\cdot w_{n}^{o}(m+1,1)+\frac{\mu}{\alpha}\cdot w_{n}^{o}(m-1,0)+\frac{\nu}{\alpha}\cdot w_{n}^{o}(m,2)+\frac{\gamma\cdot b}{\alpha}\cdot w_{n}^{o}(m,1)\right]\\
& \phantomeq-\left[0+\frac{\nu}{\alpha}\cdot w_{n}^{o}(m,1)+\frac{\lambda+\mu+\gamma\cdot b}{\alpha}\cdot w_{n}^{o}(m,0)\right]\\
& =\mu+\frac{\nu}{\alpha}\cdot\underbrace{(w_{n}^{o}(m,2)-w_{n}^{o}(m,1))}_{\leq\alpha,~~\text{by }\eqref{eq:isooD}}+\frac{\gamma\cdot b}{\alpha}\cdot\underbrace{(w_{n}^{o}(m,1)-w_{n}^{o}(m,0))}_{\leq\alpha,~~\text{by }\eqref{eq:isooD}}
+\frac{\lambda}{\alpha}\cdot\underbrace{(w_{n}^{o}(m+1,1)-w_{n}^{o}(m+1,0))}_{\leq\alpha,~~\text{by }\eqref{eq:isooD}}\\
& \phantomeq+\underbrace{\frac{\lambda}{\alpha}\cdot(w_{n}^{o}(m+1,0)-w_{n}^{o}(m,0))-\frac{\mu}{\alpha}\cdot(w_{n}^{o}(m,0)-w_{n}^{o}(m-1,0))}_{\leq0,~~\text{by }\lambda<\mu\text{ and }\eqref{eq:isooB},\ \eqref{eq:isooE}}{\leq}\alpha.
\end{align*}

\noindent
For $m\geq1$ and $k\in\{2,\dots,b-1\}$ it holds
\begin{align*}
& \phantomeq w_{n+1}^{o}(m,k)-w_{n+1}^{o}(m,k-1)\\
& =\left[\mu+\frac{\lambda}{\alpha}\cdot w_{n}^{o}(m+1,k)+\frac{\mu}{\alpha}\cdot w_{n}^{o}(m-1,k-1)+\frac{\nu}{\alpha}\cdot w_{n}^{o}(m,k+1)\right.\\
& \phantomeq\left.+\frac{\gamma\cdot(k-1)}{\alpha}\cdot w_{n}^{o}(m,k-1)+\frac{\gamma\cdot(b-k+1)}{\alpha}\cdot w_{n}^{o}(m,k)\right]\\
& \phantomeq-\left[\mu+\frac{\lambda}{\alpha}\cdot w_{n}^{o}(m+1,k-1)+\frac{\mu}{\alpha}\cdot w_{n}^{o}(m-1,k-2)+\frac{\nu}{\alpha}\cdot w_{n}^{o}(m,k)\right.\\
& \phantomeq\left.+\frac{\gamma\cdot(k-2)}{\alpha}\cdot w_{n}^{o}(m,k-2)+\frac{\gamma\cdot(b-k+2)}{\alpha}\cdot w_{n}^{o}(m,k-1)\right]\\
& =\frac{\lambda}{\alpha}\cdot\underbrace{(w_{n}^{o}(m+1,k)-w_{n}^{o}(m+1,k-1))}_{\leq\alpha,~~\text{by }\eqref{eq:isooD}}+\frac{\mu}{\alpha}\cdot\underbrace{(w_{n}^{o}(m-1,k-1)-w_{n}^{o}(m-1,k-2))}_{\leq\alpha,~~\text{by }\eqref{eq:isooD}}\\
& \phantomeq+\frac{\nu}{\alpha}\cdot\underbrace{(w_{n}^{o}(m,k+1)-w_{n}^{o}(m,k))}_{\leq\alpha,~~\text{by }\eqref{eq:isooD}}+\frac{\gamma\cdot(k-2)}{\alpha}\cdot\underbrace{(w_{n}^{o}(m,k-1)-w_{n}^{o}(m,k-2))}_{\leq\alpha,~~\text{by }\eqref{eq:isooD}}\\
& \phantomeq+\frac{\gamma\cdot(b-k+1)}{\alpha}\cdot\underbrace{(w_{n}^{o}(m,k)-w_{n}^{o}(m,k-1))}_{\leq\alpha,~~\text{by }\eqref{eq:isooD}}\leq\alpha-\gamma.
\end{align*}
For $m\geq1$ and $k=b$ it holds
\begin{align*}
& \phantomeq w_{n+1}^{o}(m,b)-w_{n+1}^{o}(m,b-1)\\
& =\left[\mu+\frac{\lambda}{\alpha}\cdot w_{n}^{o}(m+1,b)+\frac{\mu}{\alpha}\cdot w_{n}^{o}(m-1,b-1)+\frac{\gamma\cdot(b-1)}{\alpha}\cdot w_{n}^{o}(m,b-1)+\frac{\nu+\gamma}{\alpha}\cdot w_{n}^{o}(m,b)\right]\\
& \phantomeq-\left[\mu+\frac{\lambda}{\alpha}\cdot w_{n}^{o}(m+1,b-1)+\frac{\mu}{\alpha}\cdot w_{n}^{o}(m-1,b-2)+\frac{\nu}{\alpha}\cdot w_{n}^{o}(m,b)\right.\\
& \phantomeq\left.+\frac{\gamma\cdot(b-2)}{\alpha}\cdot w_{n}^{o}(m,b-2)+\frac{\gamma\cdot2}{\alpha}\cdot w_{n}^{o}(m,b-1)\right]\\
& =\frac{\lambda}{\alpha}\cdot\underbrace{(w_{n}^{o}(m+1,b)-w_{n}^{o}(m+1,b-1))}_{\leq\alpha,~~\text{by }\eqref{eq:isooD}}+\frac{\mu}{\alpha}\cdot\underbrace{(w_{n}^{o}(m-1,b-1)-w_{n}^{o}(m-1,b-2))}_{\leq\alpha,~~\text{by }\eqref{eq:isooD}}\\
& \phantomeq+\frac{\gamma\cdot(b-2)}{\alpha}\cdot\underbrace{(w_{n}^{o}(m,b-1)-w_{n}^{o}(m,b-2))}_{\leq\alpha,~~\text{by }\eqref{eq:isooD}}+\frac{\gamma}{\alpha}\cdot\underbrace{(w_{n}^{o}(m,b)-w_{n}^{o}(m,b-1))}_{\leq\alpha,~~\text{by }\eqref{eq:isooD}}\\
& \phantomeq\leq(\lambda+\mu+\gamma\cdot(b-1))=\alpha-\nu-\gamma.
\end{align*}

\noindent
$\blacktriangleright$ Fifth, we check \eqref{eq:isooE}.
For $m\geq1$ and $k=0$ it holds 
\begin{align*}
& \phantomeq w_{n+1}^{o}(m+1,0)-2\cdot w_{n+1}^{o}(m,0)+w_{n+1}^{o}(m-1,0)\\
& =\left[\frac{\nu}{\alpha}\cdot w_{n}^{o}(m+1,1)+\frac{\lambda+\mu+\gamma\cdot b}{\alpha}\cdot w_{n}^{o}(m+1,0)\right]-2\cdot\left[\frac{\nu}{\alpha}\cdot w_{n}^{o}(m,1)+\frac{\lambda+\mu+\gamma\cdot b}{\alpha}\cdot w_{n}^{o}(m,0)\right]\\
& \phantomeq+\left[\frac{\nu}{\alpha}\cdot w_{n}^{o}(m-1,1)+\frac{\lambda+\mu+\gamma\cdot b}{\alpha}\cdot w_{n}^{o}(m-1,0)\right]\\
& =\frac{\nu}{\alpha}\cdot\underbrace{(w_{n}^{o}(m+1,1)-2\cdot w_{n}^{o}(m,1)+w_{n}^{o}(m-1,1))}_{\leq0,~~\text{by }\eqref{eq:isooE}}\\
& \phantomeq+\frac{\lambda+\mu+\gamma\cdot b}{\alpha}\cdot\underbrace{(w_{n}^{o}(m+1,0)-2\cdot w_{n}^{o}(m,0)+w_{n}^{o}(m-1,0))}_{\leq0,~~\text{by }\eqref{eq:isooE}}\leq0.
\end{align*}

\noindent
For $m=1$ and $k\in\{1,\dots,b-1\}$ it holds
\begin{align*}
& \phantomeq w_{n+1}^{o}(2,k)-2\cdot w_{n+1}^{o}(1,k)+w_{n+1}^{o}(0,k)\\
& =\left[\mu+\frac{\lambda}{\alpha}\cdot w_{n}^{o}(3,k)+\frac{\mu}{\alpha}\cdot w_{n}^{o}(1,k-1)+\frac{\nu}{\alpha}\cdot w_{n}^{o}(2,k+1)+\frac{\gamma\cdot(k-1)}{\alpha}\cdot w_{n}^{o}(2,k-1)\right.\\
& \phantomeq\left.+\frac{\gamma\cdot(b-k+1)}{\alpha}\cdot w_{n}^{o}(2,k)\right]-2\cdot\left[\mu+\frac{\lambda}{\alpha}\cdot w_{n}^{o}(2,k)+\frac{\mu}{\alpha}\cdot w_{n}^{o}(0,k-1)+\frac{\nu}{\alpha}\cdot w_{n}^{o}(1,k+1)\right.\\
& \phantomeq+\left.\frac{\gamma\cdot(k-1)}{\alpha}\cdot w_{n}^{o}(1,k-1)+\frac{\gamma\cdot(b-k+1)}{\alpha}\cdot w_{n}^{o}(1,k)\right]\\
& \phantomeq+\left[0+\frac{\lambda}{\alpha}\cdot w_{n}^{o}(1,k)+\frac{\nu}{\alpha}\cdot w_{n}^{o}(0,k+1)+\frac{\gamma\cdot k}{\alpha}\cdot w_{n}^{o}(0,k-1)+\frac{\mu+\gamma\cdot (b-k)}{\alpha}\cdot w_{n}^{o}(0,k)\right]\\
& =-\mu+\frac{\lambda}{\alpha}\cdot(w_{n}^{o}(3,k)-2\cdot w_{n}^{o}(2,k)+w_{n}^{o}(1,k))+\frac{\nu}{\alpha}\cdot(w_{n}^{o}(2,k+1)-2\cdot w_{n}^{o}(1,k+1)+w_{n}^{o}(0,k+1))\\
& \phantomeq+\frac{\gamma\cdot (k-1)}{\alpha}\cdot(w_{n}^{o}(2,k-1)-2\cdot w_{n}^{o}(1,k-1)+w_{n}^{o}(0,k-1))+\frac{\gamma}{\alpha}\cdot w_{n}^{o}(0,k-1)\\
& \phantomeq+\frac{\gamma\cdot(b-k+1)}{\alpha}\cdot(w_{n}^{o}(2,k)-2\cdot w_{n}^{o}(1,k)+w_{n}^{o}(0,k))-\frac{\gamma}{\alpha}\cdot w_{n}^{o}(0,k)\\
& \phantomeq+\frac{\mu}{\alpha}\cdot(w_{n}^{o}(1,k-1)-2\cdot w_{n}^{o}(0,k-1)+w_{n}^{o}(0,k))\\
& =\frac{\lambda}{\alpha}\cdot\underbrace{(w_{n}^{o}(3,k)-2\cdot w_{n}^{o}(2,k)+w_{n}^{o}(1,k))}_{\leq0,~~\text{by }\eqref{eq:isooE}}+\frac{\nu}{\alpha}\cdot\underbrace{(w_{n}^{o}(2,k+1)-2\cdot w_{n}^{o}(1,k+1)+w_{n}^{o}(0,k+1))}_{\leq0,~~\text{by }\eqref{eq:isooE}}\\
& \phantomeq+\frac{\gamma\cdot (k-1)}{\alpha}\cdot\underbrace{(w_{n}^{o}(2,k-1)-2\cdot w_{n}^{o}(1,k-1)+w_{n}^{o}(0,k-1))}_{\leq0,~~\text{by }\eqref{eq:isooE}}\\
& \phantomeq+\frac{\gamma\cdot(b-k+1)}{\alpha}\cdot\underbrace{(w_{n}^{o}(2,k)-2\cdot(w_{n}^{o}(1,k)+w_{n}^{o}(0,k))}_{\leq0,~~\text{by }\eqref{eq:isooE}}+\frac{\mu}{\alpha}\cdot\underbrace{(w_{n}^{o}(1,k-1)-w_{n}^{o}(0,k-1))}_{\in[0,\alpha]~~\text{by }\eqref{eq:isooB},\ \eqref{eq:isooC}}-\mu\\
& \phantomeq+\underbrace{\frac{\mu}{\alpha}\cdot(w_{n}^{o}(0,k)-w_{n}^{o}(0,k-1))-\frac{\gamma}{\alpha}\cdot(w_{n}^{o}(0,k)-w_{n}^{o}(0,k-1))}_{=0,~~\text{by }\gamma=\mu}\leq0.
\end{align*}

\noindent
For $m=1$ and $k=b$ it holds
\begin{align*}
& \phantomeq w_{n+1}^{o}(2,b)-2\cdot w_{n+1}^{o}(1,b)+w_{n+1}^{o}(0,b)\\
& =\left[\mu+\frac{\lambda}{\alpha}\cdot w_{n}^{o}(3,b)+\frac{\mu}{\alpha}\cdot w_{n}^{o}(1,b-1)+\frac{\gamma\cdot(b-1)}{\alpha}\cdot w_{n}^{o}(2,b-1)+\frac{\nu+\gamma}{\alpha}\cdot w_{n}^{o}(2,b)\right]\\
& \phantomeq-2\cdot\left[\mu+\frac{\lambda}{\alpha}\cdot w_{n}^{o}(2,b)+\frac{\mu}{\alpha}\cdot w_{n}^{o}(0,b-1)+\frac{\gamma\cdot(b-1)}{\alpha}\cdot w_{n}^{o}(1,b-1)+\frac{\nu+\gamma}{\alpha}\cdot w_{n}^{o}(1,b)\right]\\
& \phantomeq+\left[0+\frac{\lambda}{\alpha}\cdot w_{n}^{o}(1,b)+\frac{\gamma\cdot b}{\alpha}\cdot w_{n}^{o}(0,b-1)+\frac{\mu+\nu}{\alpha}\cdot w_{n}^{o}(0,b)\right]\\
& =-\mu+\frac{\lambda}{\alpha}\cdot(w_{n}^{o}(3,b)-2\cdot w_{n}^{o}(2,b)+w_{n}^{o}(1,b))+\frac{\nu}{\alpha}\cdot(w_{n}^{o}(2,b)-2\cdot w_{n}^{o}(1,b)+w_{n}^{o}(0,b))\\
& \phantomeq+\frac{\gamma\cdot(b-1)}{\alpha}\cdot(w_{n}^{o}(2,b-1)-2\cdot w_{n}^{o}(1,b-1)+w_{n}^{o}(0,b-1))+\frac{\gamma}{\alpha}\cdot w_{n}^{o}(0,b-1)\\
& \phantomeq+\frac{\gamma}{\alpha}\cdot(w_{n}^{o}(2,b)-2\cdot w_{n}^{o}(1,b)+w_{n}^{o}(0,b))-\frac{\gamma}{\alpha}\cdot w_{n}^{o}(0,b)\\
& \phantomeq+\frac{\mu}{\alpha}\cdot(w_{n}^{o}(1,b-1)-2\cdot w_{n}^{o}(0,b-1)+w_{n}^{o}(0,b))\\
& =\frac{\lambda}{\alpha}\cdot\underbrace{(w_{n}^{o}(3,b)-2\cdot w_{n}^{o}(2,b)+w_{n}^{o}(1,b))}_{\leq0,~~\text{by }\eqref{eq:isooE}}+\frac{\nu}{\alpha}\cdot\underbrace{(w_{n}^{o}(2,b)-2\cdot w_{n}^{o}(1,b)+w_{n}^{o}(0,b))}_{\leq0,~~\text{by }\eqref{eq:isooE}}\\
& \phantomeq+\frac{\gamma\cdot(b-1)}{\alpha}\cdot\underbrace{(w_{n}^{o}(2,b-1)-2\cdot w_{n}^{o}(1,b-1)+w_{n}^{o}(0,b-1))}_{\leq0,~~\text{by }\eqref{eq:isooE}}\\
& \phantomeq+\frac{\gamma}{\alpha}\cdot\underbrace{(w_{n}^{o}(2,b)-2\cdot w_{n}^{o}(1,b)+w_{n}^{o}(0,b))}_{\leq0,~~\text{by }\eqref{eq:isooE}}\underbrace{-\mu+\underbrace{\frac{\mu}{\alpha}\cdot(w_{n}^{o}(1,b-1)-w_{n}^{o}(0,b-1))}_{\in[0,\alpha]~~\text{by }\eqref{eq:isooB},\eqref{eq:isooC}}}_{\leq0}\\
& \phantomeq+\underbrace{\frac{\mu}{\alpha}\cdot(w_{n}^{o}(0,b)-w_{n}^{o}(0,b-1))-\frac{\gamma}{\alpha}\cdot(w_{n}^{o}(0,b)-w_{n}^{o}(0,b-1))}_{=0,~~\text{by}~\gamma=\mu}\leq0.
\end{align*}

\noindent
For $m\geq2$ and $k\in\{1,\dots,b-1\}$ it holds
\begin{align*}
& \phantomeq w_{n+1}^{o}(m+1,k)-2\cdot w_{n+1}^{o}(m,k)+w_{n+1}^{o}(m-1,k)\\
& =\left[\mu+\frac{\lambda}{\alpha}\cdot w_{n}^{o}(m+2,k)+\frac{\mu}{\alpha}\cdot w_{n}^{o}(m,k-1)+\frac{\nu}{\alpha}\cdot w_{n}^{o}(m+1,k+1)\right.\\
& \phantomeq+\left.\frac{\gamma\cdot(k-1)}{\alpha}\cdot w_{n}^{o}(m+1,k-1)+\frac{\gamma\cdot(b-k+1)}{\alpha}\cdot w_{n}^{o}(m+1,k)\right]\\
& \phantomeq-2\cdot\left[\mu+\frac{\lambda}{\alpha}\cdot w_{n}^{o}(m+1,k)+\frac{\mu}{\alpha}\cdot w_{n}^{o}(m-1,k-1)+\frac{\nu}{\alpha}\cdot w_{n}^{o}(m,k+1)\right.\\
& \phantomeq\left.+\frac{\gamma\cdot(k-1)}{\alpha}\cdot w_{n}^{o}(m,k-1)+\frac{\gamma\cdot (b-k+1)}{\alpha}\cdot w_{n}^{o}(m,k)\right]\\
& \phantomeq+\left[[\mu+\frac{\lambda}{\alpha}\cdot w_{n}^{o}(m,k)+\frac{\mu}{\alpha}\cdot w_{n}^{o}(m-2,k-1)+\frac{\nu}{\alpha}\cdot w_{n}^{o}(m-1,k+1)\right.\\
& \phantomeq\left.+\frac{\gamma\cdot(k-1)}{\alpha}\cdot w_{n}^{o}(m-1,k-1)+\frac{\gamma\cdot(b-k+1)}{\alpha}\cdot w_{n}^{o}(m-1,k)\right]\\
& =\frac{\lambda}{\alpha}\cdot\underbrace{(w_{n}^{o}(m+2,k)-2\cdot w_{n}^{o}(m+1,k)+w_{n}^{o}(m,k))}_{\leq0,~~\text{by }\eqref{eq:isooE}}\\
& \phantomeq+\frac{\mu}{\alpha}\cdot\underbrace{(w_{n}^{o}(m,k-1)-2\cdot w_{n}^{o}(m-1,k-1)+w_{n}^{o}(m-2,k-1))}_{\leq0,~~\text{by }\eqref{eq:isooE}}\\
& \phantomeq+\frac{\nu}{\alpha}\cdot\underbrace{(w_{n}^{o}(m+1,k+1)-2\cdot w_{n}^{o}(m,k+1)+w_{n}^{o}(m-1,k+1))}_{\leq0,~~\text{by }\eqref{eq:isooE}}\\
& \phantomeq+\frac{\gamma \cdot (k-1)}{\alpha}\cdot\underbrace{(w_{n}^{o}(m+1,k-1)-2\cdot w_{n}^{o}(m,k-1)+w_{n}^{o}(m-1,k-1))}_{\leq0,~~\text{by }\eqref{eq:isooE}}\\
& \phantomeq+\frac{\gamma\cdot(b-k+1)}{\alpha}\cdot\underbrace{(w_{n}^{o}(m+1,k)-2\cdot w_{n}^{o}(m,k)+w_{n}^{o}(m-1,k))}_{\leq0,~~\text{by }\eqref{eq:isooE}}\leq0.
\end{align*}

\noindent
For $m\geq2$ and $k=b$ it holds
\begin{align*}
& \phantomeq w_{n+1}^{o}(m+1,b)-2\cdot w_{n+1}^{o}(m,b)+w_{n+1}^{o}(m-1,b)\\
& =\left[\mu+\frac{\lambda}{\alpha}\cdot w_{n}^{o}(m+2,b)+\frac{\mu}{\alpha}\cdot w_{n}^{o}(m,b-1)+\frac{\gamma\cdot(b-1)}{\alpha}\cdot w_{n}^{o}(m+1,b-1)+\frac{\nu+\gamma}{\alpha}\cdot w_{n}^{o}(m+1,b)\right]\\
& \phantomeq-2\cdot\left[\mu+\frac{\lambda}{\alpha}\cdot w_{n}^{o}(m+1,b)+\frac{\mu}{\alpha}\cdot w_{n}^{o}(m-1,b-1)+\frac{\gamma\cdot(b-1)}{\alpha}\cdot w_{n}^{o}(m,b-1)+\frac{\nu+\gamma}{\alpha}\cdot w_{n}^{o}(m,b)\right]\\
& \phantomeq+\left[\mu+\frac{\lambda}{\alpha}\cdot w_{n}^{o}(m,b)+\frac{\mu}{\alpha}\cdot w_{n}^{o}(m-2,b-1).+\frac{\gamma\cdot(b-1)}{\alpha}\cdot w_{n}^{o}(m-1,b-1)+\frac{\nu+\gamma}{\alpha}\cdot w_{n}^{o}(m-1,b)\right]\\
& =\frac{\lambda}{\alpha}\cdot\underbrace{(w_{n}^{o}(m+2,b)-2\cdot w_{n}^{o}(m+1,b)+w_{n}^{o}(m,b))}_{\leq0,~~\text{by }\eqref{eq:isooE}}\\
& \phantomeq+\frac{\mu}{\alpha}\cdot\underbrace{(w_{n}^{o}(m,b-1)-2\cdot w_{n}^{o}(m-1,b-1)+w_{n}^{o}(m-2,b-1))}_{\leq0,~~\text{by }\eqref{eq:isooE}}\\
& \phantomeq+\frac{\gamma\cdot(b-1)}{\alpha}\cdot\underbrace{(w_{n}^{o}(m+1,b-1)-2\cdot w_{n}^{o}(m,b-1)+w_{n}^{o}(m-1,b-1))}_{\leq0,~~\text{by }\eqref{eq:isooE}}\\
& \phantomeq+\frac{\nu+\gamma}{\alpha}\cdot\underbrace{(w_{n}^{o}(m+1,b)-2\cdot w_{n}^{o}(m,b)+w_{n}^{o}(m-1,b))}_{\leq0,~~\text{by }\eqref{eq:isooE}}\leq0.
\end{align*}

\endproof


\section*{Acknowledgements}
We thank the associated editor and two reviewers for their careful reading of previous versions and their constructive criticism which
enhanced the article.


\begin{thebibliography}{}
	
	\bibitem[{Anderson (1991)}]{anderson:91}
	Anderson WJ (1991)
	\newblock {\it Continuous-Time {M}arkov Chains - An Application-Oriented
		Approach}, volume~7 of {\it Springer Series in Statistics - Probability and
		its Applications}.
	\newblock Springer, New York, 1991.
	
	
	\bibitem[{Asmussen(2003)}]{asmussen:03}
	Asmussen S (2003)
	{\it Applied Probability and Queues}.
	volume 51 of  {\it Applications of Mathematics} (Springer, New York, second edition).
	
	\bibitem[{Basket et~al.(1975)}]{baskett;chandy;muntz;palacios:75}
	Baskett F, Chandy M, Muntz R, Palacios FG (1975)
	Open, closed and mixed networks of queues with different classes of customers.
	{\it Journal of the Association for Computing Machinery}
	22(2):248--260.
	
	\bibitem[{Berman and Kim(1999)}]{berman;kim:99}
	Berman O, Kim E (1999)
	Stochastic models for inventory management at service facilities.
	{\it Comm. Statist.-- Stochastic Models} 15(4):695--718.
	
	\bibitem[{Berman and Sapna(2000)}]{berman;sapna:00}
	Berman O, Sapna KP (2000)
	Inventory management at service facilities for systems with arbitrarily distributed service times.
	{\it Comm. Statist.-- Stochastic Models} 16(3-4):343--360.
	
	\bibitem[{Berman and Sapna(2002)}]{berman;sapna:02}
	Berman O, Sapna KP (2002)
	Optimal service rates of a service facility with perishable inventory items.
	{\it Naval Research Logistics} 49(5):464--482.
	
	\bibitem[{Boon, Boxma and Winands(2011)}]{boon;boxma;winands:11}
	Boon M, Boxma OJ, and  Winands EMM (2011)
	\newblock On open problems in polling systems.
	\newblock {\it Queueing Syst} 68(3-4):365--374.
	
	\bibitem[{Chung(1967)}]{chung:67}
	Chung KL (1967)
	{\it Markov Chains with Stationary Transition Probabilities}
	(Springer, Berlin).
	
	\bibitem[{Cogburn(1980)}]{cogburn:80}
	Cogburn R (1980)
	Markov chains in random environments: The case of {M}arkovian
	environments.
	{\it Annals of Probability} 8(5):908--916.
	
	\bibitem[{Cogburn(1984)}]{cogburn:84}
	Cogburn R (1984)
	The ergodic theory of Markov chains in random environments.
	{\it Zeitschrift f{\"u}r Wahrscheinlichkeitstheorie und verwandte
		Gebiete} 66(1):109--128.
	
	\bibitem[{Cornez(1987)}]{cornez:87}
	Cornez R (1987)
	Birth and death processes in random environments with feedback.
	{\it Journal of Applied Probability} 24(1):25--34.
	
	\bibitem[{Cogburn and Torrez(1981)}]{cogburn;torrez:81}
	Cogburn R, Torrez WC (1981).
	Birth and death processes with random environments in continuous time.
	{\it Journal of Applied Probability} 18(1):19--30.
	
	\bibitem[{Di Crescenzo et~al.(2012)}]{dicrescenzo:12}
	Di Crescenzo A, Iuliano A, Martinucci B (2012)
	On a bilateral birth-death process with alternating rates.
	{\it Ricerche di Matematica} 61(1):157--169.
	
	\bibitem[{Di Crescenzo et~al.(2014)}]{dicrescenzo:14}
	Di Crescenzo A, Macci C, Martinucci B (2014).
	Asymptotic results for random walks in continuous time with alternating rate.
	{\it Journal of Statistical Physics} 154(5):1352--1364.
	
	\bibitem[{Daduna(2016)}]{daduna:16} 
	Daduna H (2016)
	Moving queue on a network.
	In A.~Remke and B.R.~Haverkort, editors, {\it Measurement, Modelling and Evaluation of Dependable Computer and Communication Systems}, volume 9629
	of {\it Lecture Notes in Computer Science}
	(Springer, Cham), 40--54.
	
	\bibitem[{van Dijk and van Wal(1989)}]{dijk;wal:92}
	van Dijk NM, van der Wal J (1989) 
	Simple bounds and monotonicity results for multi-server exponential tandem queues. 
	{\it Queueing Syst} 4(1):1--16.
	
	\bibitem[{Doshi(1990)}]{doshi:90}
	Doshi BT (1990)
	\newblock Single server queues with vacations.
	\newblock In H.~Takagi, editor, {\it Stochastic Analysis of Computer and
		Communication Systems} (North--Holland, Amsterdam), 217--267.
	
	\bibitem[{Economou(2003)}]{economou:03a}
	Economou A (2003)
	A characterization of product-form stationary distributions for
	queueing systems in random environment.
	In {\it Proceedings of the 17th European Simulation Multiconference} (SCS-European Publishing House, Delft), 193--198.
	
	\bibitem[{Economou(2005)}]{economou:05}
	Economou A (2005)
	Generalized product-form stationary distributions for Markov chains in random environments with queueing applications. 
	{\it Advances in Applied Probability} 37(1):185--211.
	
	\bibitem[{Falin(1996)}]{falin:96}
	Falin G (1996)
	A heterogeneous blocking system in a random environment.
	{\it Journal of Applied Probability} 33(1):211--216.
	
	\bibitem[{Foss et~al.(2012)}]{foss;shneer;tyurlikov:12}
	Foss S, Shneer S, Tyurlikov A (2012)
	Stability of a Markov-modulated Markov Chain, with application to a wireless network governed by two protocols.
	{\it Stochastic Systems} 2(1):208–231.
	
	\bibitem[{Gannon et~al.(2016)}]{gannon;pechersky;suhov;yambartsev:16}
	Gannon M, Pechersky E, Suhov Y, Yambartsev V (2016).
	A random walk in a queueing network environment.
	{\it Journal of Applied Probability} 53(2):448--462.
	
	\bibitem[{Gaver et~al.(1984)}]{gaver;jacobs;latouche:84}
	Gaver DP, Jacobs RA, Latouche G (1984)
	Finite birth-and-death models in randomly changing environments.
	{\it Advances in Applied Probability} 16(4):715--731.
	
	\bibitem[{Helm and Waldmann(1984)}]{helm;waldmann:84}
	Helm W, Waldmann KH (1984)
	Optimal control of arrivals to multiserver queues in a random environment.
	{\it Journal of Applied Probability} 21(3):602--615.
	
	\bibitem[{Jackson(1957)}]{jackson:57}
	Jackson JR (1957)
	Networks of waiting lines.
	{\it Operations Research} 5(4):518--521.
	
	\bibitem[{Jeganathan(2014)}]{jeganathan:14}
	Jeganathan K (2014)
	Perishable inventory system at service facilities with multiple server vacations and impatient customers.
	{\it Journal of Statistics Applications \& Probability Letters} 3(3):63--73.
	
	\bibitem[{Kelly(1976)}]{kelly:76}
	Kelly FP (1976)
	Networks of queues.
	{\it Advances in Applied Probability} 8(2):416--432.
	
	\bibitem[{Kelly(1979)}]{kelly:79}
	Kelly FP (1979)
	{\it Reversibility and Stochastic Networks}
	(Wiley, Chichester).
	
	\bibitem[{Kelly and Yudovina(2014)}]{kelly;yudovina:14}
	Kelly  FP, Yudovina E (2014)
	{\it Stochastic Networks}.
	IMS Textbooks (Cambridge University Press, Cambridge).
	
	\bibitem[{Koroliuk et~al.(2017)}]{koroliuk;melikov;ponomarenko;rustamov:17}
	Koroliuk VS, Melikov AZ, Ponomarenko LA, Rustamov AM (2017)
	Asymptotic analysis of the system with server vacation and perishable inventory.
	{\it Cybernetics and Systems Analysis} 53(4):543--553.
	
	\bibitem[{Koroliuk et~al.(2018)}]{koroliuk;melikov;ponomarenko;rustamov:18}
	Koroliuk VS, Melikov AZ, Ponomarenko LA, Rustamov AM (2018)
	Models of perishable queueing-inventory systems with server vacation.
	{\it Cybernetics and Systems Analysis} 54(1):31--44.
	
	
	\bibitem[{Krenzler(2016)}]{krenzler:16}
	Krenzler R (2016)
	\newblock { Queueing systems in a random environment}.
	\newblock PhD thesis, Universit{\"a}t Hamburg, Department of Mathematics,
	Hamburg.
	
	\bibitem[{Krenzler and Daduna(2012)}]{krenzler;daduna:12}
	Krenzler R, Daduna H (2012)
	Loss systems in a random environment - steady state analysis.
	Preprint, no. 2012-04, Center of Mathematical Statistics and Stochastic Processes,
	Department of Mathematics, Hamburg University. 
	
	\bibitem[{Krenzler and Daduna(2014)}]{krenzler;daduna:14}
	Krenzler R, Daduna H (2014)
	Modeling and performance analysis of a node in fault tolerant wireless sensor networks.
	In K.~Fischbach and U.R.~Krieger, editors, {\it Measurement, Modelling, and Evaluation of Computing Systems and Dependability and  Fault-Tolerance}, (Springer, Berlin), 73--87.
	
	\bibitem[{Krenzler and Daduna(2015a)}]{krenzler;daduna:15}
	Krenzler R, Daduna H (2015a)
	Loss systems in a random environment - steady state analysis.
	{\it Queueing Syst} 80(1-2):127--153.
	
	\bibitem[{Krenzler and Daduna(2015b)}]{krenzler;daduna:15a}
	Krenzler R, Daduna H (2015b)
	Performability analysis of an unreliable M/M/1-type queue.
	In B.E.~Wolfinger and K.-D.~Heidtmann, editors, {\it {L}eistungs-, {Z}uverl{\"a}ssigkeits- und {V}erl{\"a}sslichkeitsbewertung
		von {K}ommunikationsnetzen und verteilten {S}ystemen}, volume 302, pages 90--95, Hamburg. Berichte des FB Informatik der Universit{\"a}t Hamburg.
	
	\bibitem[{Krishnamoorthy et~al.(2011)}]{krishnamoorthy;lakshmy;manikandan:11}
	Krishnamoorthy A, Lakshmy B, Manikandan R (2011)
	A survey on inventory models with positive service time.
	{\it OPSEARCH} 48(2):153--169.
	
	\bibitem[{Krishnamoorthy et~al.(2014)}]{krishnamoorthy;pramod;chakravarthy:14}
	Krishnamoorthy A, Pramod PK, Chakravarthy SR (2014)
	Queues with interruptions: A survey.
	{\it TOP} 22(1):290--320.
	
	\bibitem[{Krishnamoorthy et~al.(2019)}]{krishnamoorthy;shajin;viswanath:19}
	Krishnamoorthy A, Shajin D, Viswanath CN (2019).
	Inventory with positive service time: a survey.
	In V.~Anisimov and N.~Limnios, editors, {\it Advanced trends in
		queueing theory: series of books “Mathematics and Statistics”}, (ISTE \& Wiley, London), Chapter 5, 171--208.
	
	
	\bibitem[{Manuel et~al.(2007)}]{manuel;sivakumar;arivarignan:07}
	Manuel P, Sivakumar B, Arivarignan G (2007)
	A perishable inventory system with service facilities, MAP arrivals and PH-service times.
	{\it Journal of Systems Science and Systems Engineering} 16(1):62--73.
	
	\bibitem[{Manuel et~al.(2008)}]{manuel;sivakumar;arivarignan:08}
	Manuel P, Sivakumar B, Arivarignan G (2008)
	A  perishable inventory system with service facilities and retrial customers.
	{\it Computers and Industrial Engineering} 54(3):484–501.
	
	\bibitem[{Melikov and Molchanov(1992)}]{melikov;molchanov:92}
	Melikov AZ, Molchanov AA (1992)
	Stock optimization in transportation/storage systems.
	{\it Cybernetics and Systems Analysis}, 28(3):484--487.
	
	\bibitem[{Neuts(1981)}]{neuts:81}
	Neuts MF (1981)
	{\it Matrix Geometric Solutions in Stochastic Models - An Algorithmic Approach}.
	(Johns Hopkins University Press, Baltimore, MD).
	
	\bibitem[{Neuts and Lucantoni(1979)}]{neuts;lucantoni:79}
	Neuts MF, Lucantoni DM (1979). 
	Markovian queue with n servers subject to breakdowns and repairs.
	{\it Management Science} 25(9):849--861.
	
	\bibitem[{Neuts and Rao(1990)}]{neuts;rao:90}
	Neuts MF, Rao BM (1990) 
	Numerical investigation of a multiserver retrial model.
	{\it Queueing Syst} 7(2):169--189.
	
	\bibitem[{Otten(2018)}]{otten:18}
	Otten S (2018) 
	\newblock {\it Integrated Models for Performance Analysis and Optimization of Queueing-Inventory-Systems in Logistic Networks.}
	\newblock PhD thesis, Universit{\"a}t Hamburg, Department of Mathematics,
	Hamburg.
	
	\bibitem[{Pang et~al.(2020)}]{pang;sarantsev;belopolskaya;suhov:20}
	Pang G, Sarantsev A, Belopolskaya Y, Suhov Y (2020)
	Stationary distributions and convergence for {M/M/1} queues in
	interactive random environments.
	{\it Queueing Syst} 94(3-4):357--392.
	
	
	\bibitem[{Prabhu and Zhu(1989)}]{prabhu;zhu:89}
	Prabhu NU, Zhu Y (1989)
	Markov-modulated queueing systems.
	{\it Queueing Syst} 5(1):215--245.
	
	\bibitem[{Prabhu and Zhu(1995)}]{prabhu;zhu:95}
	Prabhu NU, Zhu Y (1995)
	Corrections to our paper: Markov-modulated queueing systems.
	{\it Queueing Syst} 19(4):449.
	
	\bibitem[{Saffari et~al.(2013)}]{saffari;asmussen;haji:13}
	Saffari M, Asmussen S, Haji R (2013)
	The M/M/1 queue with inventory, lost sale, and general lead times.
	{\it Queueing Syst} 75(1):65--77.
	
	\bibitem[{Saffari et~al.(2011)}]{saffari;haji;hassanzadeh:11}
	Saffari M,  Haji R, Hassanzadeh F (2011)
	A queueing system with inventory and mixed exponentially distributed lead times.
	{\it The International Journal of Advanced Manufacturing Technology} 53(9-12):1231--1237.
	
	\bibitem[{Sauer and Daduna(2003)}]{sauer;daduna:03}
	Sauer C, Daduna H (2003)
	Availability formulas and performance measures for separable degradable networks.
	{\it Economic Quality Control} 18(2):165--194.
	
	\bibitem[{Schwarz et~al.(2006)}]{schwarz;sauer;daduna;kulik;szekli:06}
	Schwarz M, Sauer C, Daduna H,  Kulik R, Szekli R (2006) 
	M/M/1 Queueing systems with inventory.
	{\it Queueing Syst} 54(1):55--78.
	
	\bibitem[{Sigman and Simchi-Levi(1992)}]{sigman;simchi-levi:92}
	Sigman K, Simchi-Levi D (1992)
	Light traffic heuristic for an M/G/1 queue with limited inventory.
	{\it Annals of Operations Research} 40(1):371--380.
	
	\bibitem[{Spieksma and Tweedie(1994)}]{spieksma;tweedie:94}
	Spieksma FM,  Tweedie RL (1994)
	\newblock Strengthening ergodicity to geometric ergodicity for {M}arkov chains.
	\newblock {\it Comm. Statist.-- Stochastic Models}, 10(1):45--74.
	
	\bibitem[{Sznitman(2002)}]{sznitman:02}
	Sznitman A-S (2002)
	Lectures on random motions in random media. In 
	A.-S.~Sznitman and E.~Bolthausen, editors, 
	{\it Ten Lectures on Random Media, volume 32 of {DMV}-Seminar}
	(Birkh{\"a}user Verlag, Basel),  9--51.
	
	\bibitem[{Takagi(1990)}]{takagi:90a}
	Takagi H (1990)
	\newblock Queueing analysis of polling models: An update.
	\newblock In H.~Takagi, editor, {\it Stochastic Analysis of Computer and
		Communications Systems} (North-Holland, Amsterdam), 267--318.
	
	
	
	\bibitem[{van der Wal(1989)}]{vanderwal:89}
	van der Wal J (1989)
	Monotonicity of the throughput of a closed exponential
	queueing network in the number of jobs. 
	{\it OR Spektrum} 11(6):97--100.
	
	\bibitem[{van Dijk(1993)}]{dijk:93}
	van Dijk NM (1993). 
	{\it Queueing Networks and Product Forms -- A Systems Approach}
	(Wiley, Chichester).
	
	\bibitem[{van Dijk(1998)}]{dijk:98}
	van Dijk NM (1998)
	Bounds and error bounds for queueing networks. 
	{\it Annals of Operations Research} 79(0):295--319.
	
	\bibitem[{van Dijk(2011)}]{dijk:11b}
	van Dijk NM (2011)
	On practical product form characterizations. In R.J.
	Boucherie and N.M. van Dijk, editors, 
	{\it Queueing Networks: A Fundamental Approach, volume 154 of International Series in Operations Research and
		Management Science}  (Springer, New York), Chapter 1, 1--83.
	
	\bibitem[{van Dijk and Korezlioglu(1992)}]{dijk;koreuioglu:92}
	van Dijk NM, Korezlioglu H (1992). 
	On product form approximations for communication networks with losses: error bounds. 
	{\it Annals of Operations Research} 35(1):69--94.
	
	
	\bibitem[{Vineetha(2008)}]{vineetha:08}
	Vineetha K (2008)
	Analysis of inventory systems with positive and  negligible service time.
	PhD thesis, Department of Statistics, University of Calicut, India.
	
	\bibitem[{Yadavalli et~al.(2007)}]{yadavalli;sivakumar;arivarignan:07}
	Yadavalli VSS, Sivakumar, B, Arivarignan, G. (2007)
	\newblock Stochastic inventory management at a service facility with a set of reorder levels.
	\newblock {\it {ORiON}}, 28(2):137 -- 149.
	
	\bibitem[{Yadavalli et~al.(2012)}]{yadavalli;sivakumar;arivarignan;adetunji:12}
	Yadavalli VSS, Sivakumar, B, Arivarignan, G, Adetunji, O (2012)
	\newblock A finite source multi-server inventory system with service facility.
	\newblock {\it Computers \& Industrial Engineering}, 63(4):739 -- 753, 2012.
	
	
	\bibitem[{Yadavalli et~al.(2015)}]{yadavalli;anbazhagan;jeganathan:15a}
	Yadavalli VSS, Anbazhagan N, Jeganathan K (2015)
	A two heterogeneous servers perishable inventory system of a finite population with one unreliable server and repeated attempts.
	{\it Pakistan Journal of Statistics} 31(1):135–158.
	
	\bibitem[{Yechiali(1973)}]{yechiali:73}
	Yechiali U (1973)
	A queuing-type birth-and-death process defined on a continuous-time
	Markov chain.
	{\it Operations Research} 21(2):604--609.
	
	\bibitem[{Zhu(1994)}]{zhu:94}
	Zhu Y (1994)
	Markovian queueing networks in a random environment.
	{\em Operat. Res. Lett.} 15(1):11--17.
	
	
\end{thebibliography}
\end{document}